\documentclass[reqno, 11pt, a4paper]{amsart}
\usepackage{latexsym,ifthen,xspace}
\usepackage{amsmath,amssymb,amsthm}
\usepackage{enumerate, calc, bbm}
\usepackage[usenames,dvipsnames]{xcolor}
\usepackage{paralist}
\usepackage{comment}
\usepackage[colorlinks=true, pdfstartview=FitV, linkcolor=blue,
  citecolor=blue, urlcolor=blue, pagebackref=false]{hyperref}
\usepackage[text={33pc,605pt},centering]{geometry}    
\usepackage{graphicx}
\usepackage{todonotes}
\usepackage{dsfont}
\usepackage{mathrsfs}  
\newtheorem{theorem}{Theorem}[section]
\newtheorem{lemma}[theorem]{Lemma}
\newtheorem{prop}[theorem]{Proposition}
\newtheorem{assumption}[theorem]{Assumption}
\newtheorem{coro}[theorem]{Corollary}

\theoremstyle{definition}

\newtheorem{example}[theorem]{Example}

\theoremstyle{remark}
\newtheorem{remark}[theorem]{Remark}

\numberwithin{equation}{section}

\DeclareMathAlphabet{\mathsl}{OT1}{cmss}{m}{sl}
\SetMathAlphabet{\mathsl}{bold}{OT1}{cmss}{bx}{sl}

\newcommand{\al}{\ensuremath{\alpha}}
\newcommand{\be}{\ensuremath{\beta}}
\newcommand{\ga}{\ensuremath{\gamma}}

\renewcommand{\th}{\ensuremath{\theta}}

\newcommand{\ka}{\ensuremath{\kappa}}
\newcommand{\la}{\ensuremath{\lambda}}

\newcommand{\si}{\ensuremath{\sigma}}

\newcommand{\om}{\ensuremath{\omega}}

\newcommand{\ve}{\ensuremath{\varepsilon}}

\newcommand{\Ga}{\ensuremath{\Gamma}}

\newcommand{\cC}{\ensuremath{\mathcal C}}

\newcommand{\cE}{\ensuremath{\mathcal E}}
\newcommand{\cF}{\ensuremath{\mathcal F}}

\newcommand{\cI}{\ensuremath{\mathcal I}}

\newcommand{\cL}{\ensuremath{\mathcal L}}

\newcommand{\cO}{\ensuremath{\mathcal O}}

\newcommand{\cV}{\ensuremath{\mathcal V}}

\newcommand{\bbE}{\ensuremath{\mathbb E}}

\newcommand{\bbN}{\ensuremath{\mathbb N}}

\newcommand{\bbP}{\ensuremath{\mathbb P}}
\newcommand{\bbQ}{\ensuremath{\mathbb Q}}
\newcommand{\bbR}{\ensuremath{\mathbb R}}

\newcommand{\bbZ}{\ensuremath{\mathbb Z}}
\newcommand{\md}{\ensuremath{\mathrm{d}}}

\newcommand{\scpr}[3]{%
  \ensuremath{%
    \big\langle
      #1, #2
    \big\rangle_{\raisebox{-0ex}{$\scriptstyle\ell^{\raisebox{.1ex}{$\scriptscriptstyle 2$}} (#3)$}}
  }
}

\newcommand{\Norm}[2]{%
  \ensuremath{%
    \mathchoice{\big\lVert #1 \big\rVert}
     {\lVert #1 \rVert}
     {\lVert #1 \rVert}
     {\lVert #1 \rVert}_{\raisebox{-.0ex}{$\scriptstyle#2$}}
  }
}

\DeclareMathOperator{\mean}{\mathbb{E}}
\DeclareMathOperator{\Mean}{\mathrm{E}}
\DeclareMathOperator{\prob}{\mathbb{P}} 
\DeclareMathOperator{\Prob}{\mathrm{P}} 

\DeclareMathOperator{\supp}{\mathrm{supp}}

\newcommand{\av}[1]{\mathop{\mathrm{av}}(#1)}

\newcommand{\ldef}{\ensuremath{\mathrel{\mathop:}=}}
\newcommand{\rdef}{\ensuremath{=\mathrel{\mathop:}}}

\newcommand{\indicator}{\mathbbm{1}}

\newcommand{\tend}[2]{\displaystyle\mathop{\longrightarrow}_{#1\rightarrow#2}}
\newcommand{\eps}{\varepsilon}
\newcommand{\1}{\mathds{1}}

\begin{document}

\title[FPP with long-range correlations and applications]{First passage percolation with long-range correlations and applications to random Schr\"odinger operators}


\author{Sebastian Andres}
\address{Technische Universit\"at Braunschweig}
\curraddr{Institut f\"ur Mathematische Stochastik, Universit\"atsplatz 2, 38106 Braunschweig}
\email{sebastian.andres@tu-braunschweig.de}
\thanks{}

\author{Alexis Pr\'evost}
\address{University of Geneva}
\curraddr{24, rue du G\'en\'eral Dufour, 1211 Geneva}
\email{alexis.prevost@unige.ch}
\thanks{}



\date{\today}

\dedicatory{}

\begin{abstract} 
 We consider first passage percolation (FPP) with passage times generated by a general class of models with long-range correlations on $\mathbb{Z}^d$, $d\geq 2$, including discrete Gaussian free fields, Ginzburg-Landau $\nabla \phi$ interface models or  random interlacements as prominent examples. We show that the associated time constant is positive, the FPP distance is comparable to the Euclidean distance, and we obtain a shape theorem. We also present two applications for  random conductance models (RCM) with possibly unbounded and strongly correlated conductances. Namely, we obtain a Gaussian heat kernel upper bound for RCMs with a general class of speed measures, and an exponential decay estimate for the Green's function of RCMs with random killing measures. 
\end{abstract}

\maketitle
\setcounter{tocdepth}{1}
\tableofcontents

\section{Introduction}\label{sec:INTRO}
\label{sec:intro}
\subsection{Positivity of the time constant in first passage percolation}
First passage percolation (FPP) was originally introduced in the 1960s by Hammersley and Welsh as a model of fluid flow in a randomly porous material. Since its origin it was a central topic in probability theory and it is still an area of active research. We refer to \cite{ADH17} for a recent survey. To define the model, let us consider the lattice $\bbZ^d,$ $d\geq2,$ with edge set $E_d,$ and $(t^\om_e)_{e\in E_d }$ be a family of non-negative random weights, also called passage times, which we allow to be possibly infinite, under some family of probability measures $(\bbP^u)_{u\in I}$ indexed by a parameter $u$ in some open interval $I\subset\bbR$. The random pseudo-metric $d^{\omega}(x,y)_{x,y\in{\bbZ^d}}$ associated with first passage percolation (or FPP distance in short) is given by 
\begin{equation}
\label{intro:defT}
    d^{\omega}(x,y)=\inf_{\pi:\,x\stackrel{\pi}\leftrightarrow y}\sum_{e\in\pi}t^\om_e,
\end{equation}
where the infimum is taken over all simple paths $\pi$ of edges connecting $x$ to $y.$ Under some mild conditions on the law of $(t^\om_e)_{e\in E_d}$, see for instance \eqref{eq:boundonmement} and condition~{\bf P1} in Section~\ref{sec:result},  by using the sub-additive ergodic theorem (see e.g.\ \cite{MR254907, Li85}) one can classically prove that, for all $u\in{I}$ and $x\in{\bbZ^d\setminus\{0\}}$, there exists a constant $\mu_u(x)\in{[0,\infty)}$ such that
\begin{equation}
    \label{intro:deftimeconstant}
    \lim_{n\rightarrow\infty}\frac{d^{\omega}(0,nx)}{n}=\mu_u(x),\quad \bbP^u\text{-a.s.\ and in }L^1.
\end{equation}
The constant $\mu_u(x)$ is called the time constant and depends on the choice of the direction $x$ and the law $\bbP^u$. Our aim is to find conditions under which the time constant is strictly positive so that $d^{\omega}(0,nx)$ grows linearly. In the case where the weights  $(t^\om_e)_{e\in E_d}$ are i.i.d.\ a simple criterion \cite[Theorem~6.1]{MR876084} is that
\begin{equation}
\label{intro:mupositiveiid}
    \mu_u(x)>0\text{ if and only if }\bbP^u(t^\om_e=0)< p_c, \qquad \text{for any $x\in{\bbZ^d\setminus\{0\}}$},
\end{equation}
where $p_c$ is the critical parameter for i.i.d.\ Bernoulli bond percolation on $\bbZ^d.$ The time-constant can actually be described more precisely via a variational formula using techniques from stochastic homogenization, see \cite{KrishnanArjun2016VFft} and \cite{CoNa,10.1214/21-AIHP1200} for recent related works. 

For a general ergodic family $(t^\om_e)_{e\in E_d}$ the following lower bound on $d^\om$ can be shown by the same arguments as in the proof of \cite[Theorem~2.4]{ADS19}. Within such a general framework this lower bound also turns out to be optimal  up to an arbitrarily small correction in the exponent, see \cite[Theorem~2.5]{ADS19}.

\begin{prop}[\cite{ADS19}] \label{prop:FPP_lower_gen}
For any $u\in I$, suppose that $(t^\om_e)_{e\in E_d}$ is a family of ergodic random variables under $\prob^u$, taking values in $(0,\infty)\cup\{\infty\}$ such that  $\mean^u\big[(t^\om_e)^{-q}\big]<\infty$, $e\in E_d$, for any $q>d-1$. 
 Then, there exists $c_1 > 0$ such that the following holds. For $\prob^u$-a.e.\ $\om$ and  every $x \in \bbZ^d$, there exists $N_1(\om, x) < \infty$ such that  for any $y \in \bbZ^d$ with $|x-y| \geq N_1(\om, x)$,
  \begin{align*}
   d^{\omega}(x,y)
    \;\geq\;
    c_1 \, |x-y|^{1-\frac{d-1}{q}}.
  \end{align*}
\end{prop}

Hence, when the variables $(t^\om_e)_{e\in{E_d}}$ are correlated, it is harder to find a criterion similar to \eqref{intro:mupositiveiid} which implies $\mu_u(x)>0.$ In fact, it is possible to find a probability $\bbP^u$ with $\bbP^u(t^\om_e=0)=0$ but $\mu_u(x)=0$ for some $x\in\bbZ^d,$ see \cite[Example~2.1]{MR1379157}. 
Under some large deviations inequality, it is still possible to find conditions under which $\mu_u(x)>0,$ see  \cite[ Theorem~4.3]{MR2085613}. However, these conditions do not seem to capture the whole subcritical phase of $\{e\in{E_d}:\,t^\om_e=0\},$ contrary to the i.i.d.\ case \eqref{intro:mupositiveiid}, see the Remark at the end of  \cite[Section~4]{MR2085613}. Moreover, they do not always hold for models with long-range correlations, see for instance the condition in \cite[Theorem~7.2]{MR2085613} which does not hold for the Gaussian free field. 

A result similar to \eqref{intro:mupositiveiid} for correlated fields has also recently been obtained in \cite{DeGa}. It however requires stretched-exponential decay of the correlations, which typically do not hold for the kind of models with long-range correlations we study here. In Theorem~\ref{The:main} below, we are going to prove positivity of the time constant  for a large class of ergodic, unbounded passage times on percolation models with long-range correlations. Our conditions are very similar to the ones introduced in \cite{DRS14}, i.e.\ an ergodicity condition {\bf P1}, some monotonicity condition {\bf P2}, and a sufficiently strong decay of the correlations {\bf P3} and {\bf P3'}, respectively, see Section~\ref{sec:result} below for the precise definitions. Similar conditions on the decay of the correlations have been proved, for instance, for the discrete Gaussian free field in \cite{MR3325312}, for Ginzburg-Landau $\nabla\phi$ interface models  in \cite{2016arXiv161202385R}, and for random interlacements in \cite{MR3420516}. These results however do not provide us directly with the exact strong decay of the correlations required in {\bf P3} or {\bf P3'}, and we adapt them in Section \ref{sec:examples} to show that all these models indeed satisfy our conditions {\bf P1}--{\bf P3} (or {\bf P3'}). In particular, we prove the ergodicity condition {\bf P1} for  $\nabla\phi$-interface models with strictly convex potentials  (see Lemma~\ref{lemma:uniquenesslimitGibbsmeasure}), whereas this condition had to be assumed in \cite{2016arXiv161202385R}. Finally, our results also apply under some upper ratio weak mixing condition satisfied, for instance, by the Ising model or the massive two-dimensional Gaussian free field under certain conditions, see Section~\ref{sec:ratioweakmixing}.

In order to illustrate our results, let us now focus on the special case of the level sets of the Gaussian free field. We denote by $(\phi_x)_{x\in{{\bbZ^d}}}$ the Gaussian free field on $\bbZ^d$ in $d\geq 3$ with unit weights, see \eqref{defGFF} below, and by $h_*$ the critical parameter for the percolation of the associated level sets $\{x\in{\bbZ^d}:\,\phi_x\geq h\},$ $h\in\bbR.$ Moreover, let us extend the definition \eqref{intro:defT} of $d^{\omega}(x,y)$ and \eqref{intro:deftimeconstant} of $\mu_u$ to the case where, instead of weights  $(t^\om_e)_{e\in{E_d}}$ on the edges, we have weights  $(t^\om_x)_{x\in{\bbZ^d}}$ on the vertices by simply considering the minimal length over paths of vertices instead of paths of edges.

\begin{theorem}
\label{the:mainGFF}
Let $d\geq 3$. Assume that $(t^\om_x)_{x\in{\bbZ^d}}$ has the same law under $\bbP^h$ as $(\indicator_{\{\phi_x< h\}})_{x\in{\bbZ^d}}$.  Then, for all $x\in{\bbZ^d\setminus\{0\}}$,
\begin{equation}
\label{app:mu>0}
    \mu_h(x)>0\text{ if }h>h_*\text{ and }\mu_h(x)=0\text{ if }h<h_*.
\end{equation}
Moreover, for all $h>h_*$ and $\delta>4$, there exist positive constants $c_2$ and $c_3$ such that for all $n\in\bbN$ and $x\in{\bbZ^d\setminus\{0\}}$,
\begin{equation}
\label{app:Tfsmallercn2}
    \bbP^h\left(d^{\omega}(0,nx)\leq c_2 n\right) \leq \exp\left(-\frac{c_3n}{\log(n)^{\delta\indicator_{d=3}}}\right).
\end{equation}
\end{theorem}

Theorem \ref{the:mainGFF} can thus be seen as the equivalent of \eqref{intro:mupositiveiid} but for the level sets of the Gaussian free field, except for the unknown and probably complicated case $h=h_*$. Let us now quickly comment on the proof, and we refer to Theorem~\ref{The:main} and Corollary \ref{app:1stcorGFF} for details. The case $h<h_*$ follows from the FKG inequality together with a simple adaptation of  \cite[Proposition~6]{MR3531712}, see Proposition \ref{Prop:condformu=0} below. When $h>h_*,$ it was recently proved in \cite{DuGoRoSe} that
\begin{equation}
\label{eq:equalitycriticalparameter}
    \lim\limits_{L\rightarrow\infty}\bbP\left(Q_L\leftrightarrow Q_{2L}^c\text{ in } \big\{x\in{\bbZ^d}:\,\phi_x\geq h\big\}\right)=0,
\end{equation}
where  $Q(x,L) \ldef\{y\in{\bbZ^d}:|x-y|_{\infty}\leq L\}$ denotes the cube centered at $x$ and with radius $L,$ and we abbreviate $Q_L=Q(0,L).$ The equality \eqref{eq:equalitycriticalparameter} will be the starting point of a renormalization scheme similar to the one in \cite[Section~7]{MR3420516}, see Section~\ref{sec:renormalization}. In particular, we will define a sequence of scales $L_k$ of order $2^k,$ $k\in\bbN,$ so that the FPP distance between $Q_{L_{k+1}}$ and $Q_{2L_{k+1}^c}$ will be larger than the sum of the FPP distance between $Q(x,L_k)$ and $Q(x,2L_k)^c$ and between $Q(y,L_k)$ and $Q(y,2L_k)^c$, for some distant enough vertices $x$ and  $y$. Moreover, it  follows from the decoupling inequalities in \cite{MR3325312} that on $Q(x,2L_k)$ and $Q(y,2L_k)$ the field behaves independently under an event with high probability. Using \eqref{eq:equalitycriticalparameter} for the initialization, one can thus show recursively that the FPP distance between $Q(x,L_k)$ and $Q(x,2L_k)^c$ is of order at least $2^k$ with high probability, and \eqref{app:Tfsmallercn2} follows readily. 

Note that once \eqref{eq:equalitycriticalparameter} holds, the only property of the Gaussian free field that we used are the decoupling inequalities from \cite{MR3325312}. Therefore as explained above, this reasoning can be adapted to any model satisfying appropriate decoupling inequalities, and we refer to Theorem~\ref{The:main} for details.

One can use the results from \cite{MR1074741} to deduce from Theorem~\ref{the:mainGFF} a shape theorem, similar to the i.i.d.\ case in \cite{MR624685}: for all $h>h_*,$ there exists a compact and convex deterministic set $K_h$ with non-empty interior such that, under the probability measure $\bbP^h$ from Theorem~\ref{the:mainGFF},  
\begin{equation*}
    \bbP^h\Big((1-\eps)K_h\subset\frac{B_{d^\omega}(t)}{t}\subset(1+\eps)K_h\text{ for all large }t\Big)=1,
\end{equation*}
where $B_{d^\omega}(t)=\{x\in{\bbZ^d}:d^{\omega}(0,x)\leq t\}$ is the ball of radius $t$ with respect to the FPP distance,  see Corollary~\ref{cor:shape}. Again this can be extended to more general percolation models with long-range correlations.

\subsection{Green kernel bounds for random Schr\"odinger operators}
We now present our second main result which concerns the random conductance models (RCMs) with killing,  discussed in detail in Subsections~\ref{sec:Agmon_result} and \ref{sec:Agmon} below. 
On the graph $(\bbZ^d,E_d),$ $d\geq 2,$ let $(a^\om(e))_{e\in E_d}$ be a family of non-negative ergodic and unbounded conductances and  $(\kappa^\om(x))_{x\in{\bbZ^d}}$ be a positive random killing measure (or potential). Then consider a random walk   $X = (X_t)_{t\geq 0}$ on $\bbZ^d$  with generator $\cL^{\om}$ given by a random Schr\"odinger operator of the form
\begin{align} \label{eq:defL_intro}
\big(\cL^{\om} f)(x)
  \;=\; 
  \sum_{y \sim x} a^\om(x,y) \, \big(f(y) - f(x)\big) - h \, \kappa^\om(x) \, f(x),
\end{align}
where $h\in [0,1]$ is a scalar, $x\sim y$ means $\{x,y\}\in{E_d}$ and $a^{\om}(x,y)$ is short for $a^{\om}(\{x,y\})$. In other words, $X$ is the time-homogeneous continuous-time Markov process  on $\bbZ^d$, which jumps from $x$ to $y$ at rate $a^\om(x,y)$ and is killed when visiting $x$ at rate $h \kappa^\om(x)$. 
Our goal is to establish an exponential decay estimate for the Green's function of RCMs with ergodic unbounded conductances and killing measures with a precise rate of convergence as $h\searrow0$.
In Section~\ref{sec:Agmon} we show a general upper bounds on the Green kernel $g^{\om}(x,y)$ (see \eqref{eq:defgreen} for its definition) associated with a random walk with fixed coefficient fields $a^\om$ and $\kappa^\om$, see Theorem~\ref{thm:Agmon}, which in the special case of random ergodic conductances and killing measures on $(\bbZ^d, E_d)$ reads as follows.

\begin{theorem} \label{thm:Agmon_intro}
Let $d\geq 3$. Suppose that $(a^\om(e))_{e\in E_d}$ and $(\kappa^\om(x))_{x\in \bbZ^d}$ are stationary and ergodic with respect to space shifts on a probability space $(\Omega, \cF,\prob)$,  and suppose there exist $p,q\in (1,\infty]$ with $1/p+1/q<2/d$ such that
\begin{align} \label{eq:Agmon_moment}
C_{\mathrm{int}} \ldef \max\Big\{
\mean\big[a^\om(e)^p \big],
 \mean\big[ a^\om(e)^{-q}],
\mean\big[ \kappa^\om(0)
^p \big],
\mean\big[ \kappa^\om(0)^{-1} \big]
  \Big\}
    < \infty, \quad e\in E_d.
\end{align}
Then, there exist $\gamma=\gamma(d,p,q)\in (1,\infty)$,   $c_{4}=c_{4}(d,p,q,C_{\mathrm{int}})\in (0,\infty)$ and $c_5=c_5(d)\in{(0,1)}$ such that  the following holds. For $\prob$-a.e.\ $\om$ and any $x\in \bbZ^d$, there exist $N_{2}(\om,x)$ such that for all $y\in \bbZ^d$ with $ |x-y|  \geq N_{2}(\om,x)$ and all $h\in{[0,1]},$
\begin{align} \label{eq:Agmon_intro}
g^\om(x,y)  \; \leq \;  c_{4} \,
F_\gamma \big(h \, |x-y|^2 \big) \, |x-y|^{2-d} \!\max_{z\in{B^{\om}(x,n)^c}}\Big(e^{-c_5 \sqrt{h}  d_\kappa^\om(x,z)}\Big),
\end{align}
with $n=|x-y|/4$, $F_\gamma(r)\ldef (1+r)^\ga (1+1/r)^{1/2}$. Here the distance $d_{\kappa}^{\om}$ is given by
\begin{align} \label{eq:def_dKappa}
  d_{\kappa}^{\om}(x,y)
  \ldef
  \inf_{\gamma\in \Gamma_{xy}^{\om}}
  \Bigg\{
    \sum_{i=0}^{l_{\gamma}-1} 
    \bigg( 
      1 \wedge \frac{\kappa^\om(z_i) \wedge \kappa^\om(z_{i+1})}{a^{\om}(z_i,z_{i+1})}
    \bigg)^{\!\!1/2}
  \Bigg\}, \quad x,y\in \bbZ^d,
\end{align}
where $\Ga_{xy}^{\om}$ denotes the set of all nearest-neighbor paths $\ga = (z_0, \ldots, z_{l_\ga})$ in $\bbZ^d$  connecting $x$ and $y$.
\end{theorem}

Note that \eqref{eq:Agmon_intro} provides an exponential decay estimate governed by the $d_{\kappa}^{\om}$-distance defined in \eqref{eq:def_dKappa}, which is adapted to the conductances and the killing measure of the random walk.
For the proof of Theorem~\ref{thm:Agmon_intro}, given in Subsection~\ref{sec:Agmon}, we  follow the strategy established by Agmon in \cite{Ag82} to show exponential decay bounds on eigenfunctions of Schr\"odinger operators in $\mathbb{R}^d$. In this paper we transfer the method into the discrete setting of graph endowed with unbounded weights and killing measures only satisfying an integrability condition.
The main idea is to consider a perturbation $v=u\cdot \varphi$ of a harmonic function $u$, where $\varphi$ is contained in a certain class of test functions. Then, the key step, see Lemma~\ref{lem:agmon} below, is to show that, for any $B'\subset B \subset \bbZ^d$, the $\ell^2$-norm of $v$ on $B'$ weighted by $\kappa$ can be bounded from above by a weighted $\ell^2$-norm of $v$ on $B\setminus B'$.  The exponential decay is then obtained by choosing a suitable test function $\varphi$. The resulting $\ell^2$-bound on $e^{c d^\om_\ka(x,\cdot)} g^\om(x,\cdot)$ is then improved to a pointwise bound by a maximal inequality for the perturbed harmonic function $v$, see Proposition~\ref{prop:max_ineq} below. For the proof of the maximal inequality we use a Moser iteration scheme for discrete operators with degenerate coefficient similar to the one developed in \cite{ADS15, ADS16, ADS19}. Those iteration schemes require to choose linear cutoff functions between balls of various sizes, which is not possible a-priori for balls with respect to the $d_\kappa^\om$-distance. Then, as we need to use Euclidean balls instead, this leads to the additional term $F_\gamma \big(h \, |x-y|^2 \big)$ in \eqref{eq:Agmon_intro}.
 In a sense, Agmon's technique may be regarded as an elliptic version of Davies' perturbation method, commonly used to  obtain Gaussian heat kernel bounds, see e.g.\  \cite{ADS16a,ADS19} and references therein. Recently, Agmon's method has been used in the graph setting to study eigenfunctions of discrete Schr\"odinger operators in \cite{KP21}.

\subsection{Links between FPP and RCM}

In this paper we discuss two applications of our first passage percolation results to the random conductance model, presented in detail in Section~\ref{sec:RCM}. Our main illustration of the link between these two models will be Corollary~\ref{cor:optidecay} below, in which we use our main FPP result in Theorem~\ref{the:mainGFF} to bound the distance \eqref{eq:def_dKappa} governing the exponential decay in our main Green kernel bound \eqref{eq:Agmon_intro}.

As a first immediate application of our first passage percolation results, discussed in Subsection~\ref{sec:RCMhk}, we consider the RCM with a general speed measure $\theta$. For this model, we improve under certain conditions the heat kernel upper bounds obtained in \cite{ADS19}  for random walks on weighted graphs with unbounded conductances satisfying some integrability conditions, see Theorem~\ref{thm:hke} below. Similarly as in Theorem~\ref{thm:Agmon_intro},  the exponential decay in those bounds is governed by the so-called intrinsic  distance $d_{\theta}^{\omega}(x,y)$, see \eqref{eq:def_intMetric}, instead of the Euclidean distance $|x-y|$, but with a polynomial pre-factor in terms of $|x-y|$.
 When the conductances are random, $d_{\theta}^{\omega}$ can be seen as a first passage percolation distance, as defined in \eqref{intro:defT}. Hence, in view of Proposition~\ref{prop:FPP_lower_gen}, the bounds in \cite{ADS19} are not of Gaussian type in a general ergodic framework. 
As indicated by Theorem~\ref{the:mainGFF}, when, for instance, the weights depend on the Gaussian free field, our methods however show that the distance $d_{\theta}^{\omega}$ and the Euclidean distance are comparable. Hence, as an immediate  though important consequence of our results combined with the heat kernel bounds in \cite{ADS19}, we get genuine Gaussian  upper heat kernel estimates, see Theorem~\ref{thm:hkeGauss}, which significantly extend the upper bounds for simple random walks on correlated percolation clusters obtained in \cite{Sa17} under slightly different conditions.

For our second application, notice that the distance function $d_\kappa^\om$ in Theorem~\ref{thm:Agmon_intro} is an FPP distance, that is it is of the form \eqref{intro:defT} with $t_e^{\om}$ given by $(t_e^{\om})^2=1\wedge\frac{\kappa^{\om}(x)\wedge\kappa^{\om}(y)}{a^{\om}(e)}$ for each $e=\{x,y\}\in{E_d}.$ Moreover, under the conditions from Theorem~\ref{thm:Agmon_intro}, the time constant $\mu_{\kappa}^{\om}(x),$ $x\in{\bbZ^d\setminus\{0\}},$ defined similarly as in \eqref{intro:deftimeconstant} but for the previous choice of weights under the probability $\bbP$, exists and is finite. When this time constant is positive,  for instance under the conditions from Theorem~\ref{the:mainGFF}, the metric $d_\kappa^{\om}$ is comparable to the Euclidean metric, and \eqref{eq:Agmon_intro} directly yields the following exponential decay with respect to the Euclidian distance for the Green kernel.

\begin{coro}
\label{cor:optidecay}
Under the conditions from Theorem~\ref{thm:Agmon_intro}, if there exists $\mu>0$ such that $\mu_{\kappa}^{\om}(x)>\mu$ for all $x\in{\bbZ^d\setminus\{0\}}$, then there exist $c_6=c_6(d,p,q,C_{\mathrm{int}})$ and $c_7=c_7(d,\mu)$  and for $\prob$-a.e.\ $\om$ and any $x\in \bbZ^d$, there exist $N_{3}(\om,x)$ such that for all $y\in \bbZ^d$ with $ |x-y|  \geq N_{3}(\om,x)$ and all $h\in{[0,1]},$
\begin{equation}
\label{intro:green_bound}
    g^{\om}(x,y)\leq c_6 |x-y|^{2-d}\exp\Big(-c_{7} \sqrt{h} \, |x-y|\Big).
\end{equation}
In particular, \eqref{intro:green_bound} is satisfied when $a^\om(y,z)=e^{\beta (\phi_y+\phi_z)}$  and $\kappa^{\om}(y)=e^{\beta \phi_y}$, $\beta>0$,  where $(\phi_x)_{x\in \bbZ^d}$  denotes again the Gaussian free field on $\bbZ^d$, $d\geq 3$.
\end{coro}

We refer to the proof of Theorem~\ref{thm:GKdecay_killing} as to how \eqref{intro:green_bound} can be deduced from \eqref{eq:Agmon_intro} when the distance $d_{\kappa}^{\om}$ is comparable to the Euclidian distance, that is when $\mu_{\kappa}^{\om}$ is positive. Moreover, when the weights and killing measure depend on the Gaussian free field as below \eqref{intro:green_bound}, the positivity of the time constant $\mu_{\kappa}^{\om}$ can either be deduced from Theorem~\ref{the:mainGFF} or directly proved, see below Remark~\ref{rk:optidecay} for details.

The bound \eqref{intro:green_bound}  still holds when $a^{\om}$ and $\kappa^{\om}$ are other functions of the Gaussian free field than the one considered in Corollary~\ref{cor:optidecay}, even if the weights $a^{\om}$ are not strictly positive under certain conditions. For instance one can take $a^{\om}(x,y)=\indicator_{\{\phi_x\geq h,\phi_y\geq h\}}$ for $h<h_*,$ see Remark~\ref{rk:finalGFF}-(i). The  choice of the function in Corollary~\ref{cor:optidecay} is however particularly interesting as it has been studied,  for instance, for the two-dimensional pinned Gaussian free field in \cite{MR4050092} due to its link to Liouville Brownian motion. Moreover, similarly as for Theorem~\ref{the:mainGFF}, \eqref{intro:green_bound} can be proved not only for the Gaussian free field, but also when the conductances and killing rates are random and ergodic, and the passage times in \eqref{eq:def_dKappa} satisfy some stochastic monotonicity condition and a  weak decoupling inequality for monotone events. In particular, our method to obtain \eqref{intro:green_bound} applies to all the examples from Section~\ref{sec:examples}, and we refer to Theorem~\ref{thm:GKdecay_killing} for more details.

The main interest in the bound \eqref{intro:green_bound} is the explicit exponential decay in $h$ as $h\searrow0$ of the form
 \begin{align} \label{eq:exp_decay}
    \limsup_{|x| \to \infty} \frac{1}{|x|} \log g^\om(0,x) \leq - c \sqrt{h}.
\end{align}
This scaling of the Lyapunov exponent of order $\sqrt{h}$ as $h$ tends to zero is optimal and refers to diffusive behavior. Precise asymptotics for the Green kernel  of the Laplacian on $\mathbb{R}^d$ with a deterministic periodic potential have been shown in \cite{Schr88} by using large deviation techniques. For simple random walks on $\bbZ^d$ associated with the discrete Laplacian with non-negative i.i.d.\ potentials 
directional exponential decay of the Green's function has been derived in \cite{Ze98}, and exact asymptotics of the Lyapunov exponents have been obtained in \cite{KMZ11}.  In particular, it is shown there that both the quenched and the annealed Lyapunov exponents scale  like $c \sqrt{h}$ as $h$ tends to zero. This has been extended to the case when the potentials are not integrable in \cite{MM13,MM15}. In a sense, our Green kernel bound in Corollary~\ref{cor:optidecay}  partially extends the results in \cite{KMZ11} to random walks under random conductances with correlated potentials. Parabolic equations involving the discrete Laplacian and random potentials have been intensively studied under the banner of the parabolic Anderson model, see \cite{Ko16} for an overview.

Finally, we end this introduction by pointing out that the Green kernel of RCMs with unbounded ergodic conductances and with random killing rates can be used to analyse the correlations in certain supersymmetric spin systems.

\begin{example}[Supersymmetric spin models]
\label{ex:supersymmetric}
The Green kernel $g^\om(x,y)$ of an RCM of the form \eqref{eq:defL_intro} with random killing  appears in a respresentation  for the two-point function of the supersymmetric hyperbolic sigma model, or $\mathbb{H}^{2|2}$-model, see \cite{BH21} for a recent survey. The $\mathbb{H}^{2|2}$-model is a spin model introduced in \cite{Zi91} as a more tractable model for the  Anderson transition. The first mathematical results for $\mathbb{Z}^d$ were obtained in \cite{DSZ10}.
  Its two-point function can be represented as 
  \begin{align} \label{eq:twoPoint}
    G_{\beta,h}(x,y) = \frac 1 \beta \, \mean\big[ {a^\omega(x,y) \,  g^\omega(x,y)}\big].
  \end{align}
 Here $g^\omega(x,y)$ is the Green kernel of an RCM with conductances $a^\omega(x,y)= \beta e^{t(x)+t(y)}$ and killing rates $\kappa(x)= e^{t(x)}$ for a certain random field $(t(x))_{x\in \bbZ^d}$  depending on parameters $\beta >0$ (with the interpretation of inverse temperature)  and $h>0$ (with the interpretation of an external field).   For $d\geq 3$ and $\beta \geq \beta_0$,
  strong bounds on the field $(t(x))$ have been obtained in \cite{DSZ10}.
  The correlation function $G_{\beta,h}$ is also exactly the two-point function of the  vertex-reinforced jump process (VRJP) with initial weights $\beta$ and killing rate $h$ (see  \cite{BHS19}).  
It is an  interesting  yet challenging open problem  to transfer the exponential decay with rate $\sqrt{h}$ as $h\to 0$ in \eqref{eq:exp_decay}, obtained in Theorem~\ref{thm:GKdecay_killing} below for the quenched Green kernel $g^\om(x,y)$ under stronger assumptions, to the two-point function $G_{\beta,h}(x,y)$.  Similar estimates are expected very generally for spin models with continuous symmetry   such as $O(n)$-models
  at low temperature, but the best results for any model give rate $h$ instead of $\sqrt{h}$.   The square root corresponds to diffusive behavior in the interpretation of the $\mathbb{H}^{2|2}$ model   as a toy model for the Anderson model. For the $\mathbb{H}^{2|2}$-model, the improvement from $h$ to $\sqrt{h}$ is of particular
physical relevance because it corresponds to diffusive time evolution in the motivation of the $\mathbb{H}^{2|2}$-model as a model for random band matrices.
\end{example}

The rest of the paper is organized as follows. In Section~\ref{sec:FPP} we define a general set of conditions { \bf P1--P3} (and {\bf P3'}, respectively) under which we then derive the positivity of the time constant together with the comparability of FPP and Euclidean distances and a shape theorem. The applications to the random conductance model are discussed in Section~\ref{sec:RCM}. Finally, in Section~\ref{sec:examples} we give a number of examples of relevant models  which fulfill the required conditions.
Throughout the paper we write $c$ or $c'$ to denote a positive constant which may change on each appearance. Constants denoted $c_i$ or named constants (for instance $C_{\mathrm{fpp}}$) will remain the same. We use the same convention for random constants $N_i$ and $N$, respectively.

\section{First passage percolation with long-range correlations}
\label{sec:FPP}
\subsection{Results}
\label{sec:result}
We start by introducing the setup and notation under which we are going to work. We fix some measurable and partially ordered spaces $(\Omega_E,\cF_E)$ and $(\Omega_V,\cF_V)$.
For each $e\in{E_d}$ let $(\Omega_e,\cF_e)$ be a copy of $(\Omega_E,\cF_E)$, and for each $x\in{\bbZ^d}$ let  $(\Omega_x,\cF_x)$ be a copy of $(\Omega_V,\cF_V)$. Moreover, set $\Omega\ldef \prod_{e\in{E_d}}\Omega_e\times\prod_{x\in{\bbZ^d}}\Omega_x$ and  let $\cF$ be the associated product $\sigma$-algebra. For $\omega\in{\Omega},$ $e\in{E_d}$ and $x\in{\bbZ^d},$ we denote by $\omega_e$ and $\omega_x$ the respective canonical projections on $\Omega_e$ and $\Omega_x.$ Let us further from now on fix some set $\mathcal{N}\subset E_d\cup\bbZ^d$ around the origin, write $\Omega_{\mathcal{N}}=\prod_{\bar{e}\in{\mathcal{N}}}{\Omega_{\bar{e}}}$ and  $\mathcal{N}+x=\{x+\bar{e}:\,\bar{e}\in{\mathcal{N}}\}$ for all $x\in{\bbZ^d}.$ We also fix a measurable function $\mathbf{t}:\Omega_{\mathcal{N}}\times\Omega_{\mathcal{N}}\times\Omega_E\rightarrow[0,\infty)\cup\{\infty\},$ symmetric in the first two coordinates, such that
\begin{equation}
\label{eq:temonotone}
    (\omega_{\mathcal{N}}^1,\omega_{\mathcal{N}}^2,\omega_E)\mapsto \mathbf{t}\big(\omega_{\mathcal{N}}^1,\omega_{\mathcal{N}}^2,\omega_E\big)\text{ is monotone},
\end{equation}
and let
\begin{equation}
\label{eq:defte}
    t^\om_e:=\mathbf{t}\big(\omega_{|\mathcal{N}+x},\omega_{|\mathcal{N}+y},\omega_{e}\big), \qquad  e=\{x,y\}\in{E_d}, \, \omega\in \Omega.
\end{equation}
Here, we endow $[0,\infty)\cup\{\infty\}$ with the Borel $\sigma$-algebra generated by the intervals $[x,\infty)$ and $[x,\infty)\cup\{\infty\},$ $x\in{[0,\infty)}.$ We further equip $\Omega$ with a group of space shifts $\big\{\tau_z : z \in \bbZ^d\big\}$ so that
\begin{align} \label{eq:def:space_shift}
  \omega_{\bar{e}}\circ\tau_z=\omega_{\bar{e}+z}
 \qquad \forall\, \bar{e}\in {E_d\cup\bbZ^d}.
\end{align}

The canonical choice of parametrization will simply be $\Omega_E=[0,\infty)\cup\{\infty\}$ and $\mathbf{t}$ the projection on the third coordinate, so that $t^\om_e$ is simply the projection on $\Omega_e$. Our more general setup is particularly adapted to the purposes of Section~\ref{sec:RCM}, see for instance Assumption~\ref{ass:moment}-(i) therein. Moreover, when $e=\{x,y\},$ allowing $t^\om_e$ to also depend on $\omega_{|\mathcal{N}+x}$ and $\omega_{|\mathcal{N}+y}$ will let us study first passage percolation on the vertices by considering $\mathcal{N}=\{0\}$ and $\mathbf{t}(\omega_{V}^1,\omega_{V}^2,\omega_E)=\frac12(f(\omega_V^1)+f(\omega_V^2))$, for some measurable and monotone function $f:\Omega_V\rightarrow[0,\infty)\cup\{\infty\},$ see Remark \ref{rk:fppvertices}-(i) for details.

%

For some open interval $I\subset\bbR,$ we let $(\bbP^u)_{u\in{I}}$ be a family of probability measures on $(\Omega,\cF)$ and write $\mean^u$ for the expectation  with respect to $\prob^u$, $u\in I.$ For each $u\in{I},$ using Kingman's subadditive ergodic theorem \cite{MR254907, Li85}, the existence of the time constant $\mu_u$ from \eqref{intro:deftimeconstant} is guaranteed if the family $(t_e^\om)_{e\in E_d}$ is ergodic (cf.\ condition \textbf{P1} below) and 
\begin{equation}
    \label{eq:boundonmement}
    \bbE^u[t^\om_e]<\infty\text{ for all }e\in{E_d}.
\end{equation}
Note that \eqref{eq:boundonmement} implies in particular that $t^\om_e$ is finite $\bbP^u$-a.s., but we will also sometimes consider first passage percolation models where $t^\om_e$ can be infinite, and it is then not clear whether the time constant $\mu_u$ exists or not. When the variables $(t^\om_e)_{e\in{E_d}}$ are independent, we refer to \cite{MR2085613} and \cite{MR3531712} for conditions which imply the existence of the time constant. Let us now start with a criterion under which $\mu_u(x)=0,$ following the proof of \cite[Proposition~6]{MR3531712}. For $A,B\subset\bbZ^d$ and $C\subset E_d$ (or sometimes $C\subset\bbZ^d$) we write $\{A\longleftrightarrow B\text{ in }C\}$ for the event that there exists a connected path $\pi\subset C$ starting in $A$ and ending in $B.$ 

\begin{prop}
\label{Prop:condformu=0}
For any $u\in{I}$, assume \eqref{eq:boundonmement} and that $(t_e^\om)_{e\in E_d}$ is ergodic under $\prob^u$. Then, for all $x\in{\bbZ^d\setminus\{0\}}$,
\begin{equation*}
    \lim_{n\rightarrow\infty}\bbP^u\big(0\longleftrightarrow nx\text{ in }\{e\in{E_d}:\,t^\om_e=0\}\big)>0\quad\Longrightarrow\quad\mu_u(x)=0.
\end{equation*}
\end{prop}
\begin{proof}
Let us assume that $\mu_u(x)>0,$ then there exists a continuous and bounded function $f:[0,\infty)\rightarrow[0,\infty)$ with $f(0)=1$ and $f(\mu_u(x))=0.$ Therefore
\begin{equation}
\label{eq:proofProp1.1}
    \bbE^u\left[f\left(\frac{d^{\omega}(0,nx)}{n}\right)\right]\geq\bbP^u(d^{\omega}(0,nx)=0)=\bbP^u\big(0\longleftrightarrow nx\text{ in }\{e\in{E_d}:\,t^\om_e=0\}\big).
\end{equation}
Moreover it follows from dominated convergence and \eqref{intro:deftimeconstant} that the left-hand side of \eqref{eq:proofProp1.1} converges to 0, and we can conclude.
\end{proof}

Proposition \ref{Prop:condformu=0} states that $\mu_u(x)=0$ when the set $\{e\in{E_d}:\,t^\om_e=0\}$ percolates in a strong enough sense, and thus provides us with a result similar to the converse implication in \eqref{intro:mupositiveiid}. Let us now present our main result, which corresponds to a class of percolation models with long-range correlations for which a result similar to the direct implication in \eqref{intro:mupositiveiid} holds. 

For each $x,y\in{\bbZ^d}$ and $L\in{\bbN}$,  we write $d(x,y)=|x-y|_{\infty}$ and set $Q(x,L)\ldef \{y\in{\bbR^d}:\,d(x,y)\leq L\}$ and $Q(L) \ldef Q(0,L)$. Further, we denote by $\bar{Q}(x,L)$ the union of $Q(x,L)$ and the set of edges with both endpoints in $Q(x,L)$. For each $A,B\subset{E_d\cup\bbZ^d},$ we define $d(A,B)$ as the minimal distance $d(x,y)$ over all $x,y\in{\bbZ^d}$ such that $x$ is either in $A\cap\bbZ^d$ or an endpoint of an edge in $A\cap E_d,$ and $y$ is either in $B\cap\bbZ^d$ or an endpoint of an edge in $B\cap E_d.$  Moreover, for each $u\in{I},$ we define $\hat{\bbP}^u={\bbP^u}\otimes\bbP^u,$ and for each $\bar{e}\in \bbZ^d\cup E_d$ let $\omega_{\bar{e}}^{(i)}$ be the canonical projections of $\omega=(\omega^{(1)},\omega^{(2)})\in{\Omega\times \Omega}$ onto the edge or vertex $\bar{e}$ of the $i$-th coordinate. In other words, $\big(\omega_e^{(1)},\omega_x^{(1)}\big)_{e\in E_d, x\in \bbZ^d}$ and $\big(\omega_e^{(2)},\omega_x^{(2)}\big)_{e\in E_d, x\in \bbZ^d}$ under $\hat{\bbP}^u$ are independent copies of  $(\omega_e,\omega_x)_{e\in E_d, x\in \bbZ^d}$ under ${\bbP}^u.$ Finally for each $F\subset\bbZ^d\cup E_d,$ we say that an event $A\subset\prod_{e\in{F}}\Omega_e$ is increasing if for all $\omega\in{A}$ and $\omega'\in{\prod_{e\in{F}}\Omega_e}$ such that $\omega_e\leq \omega'_e$ for all $e\in{F},$ we have $\omega'\in{A}.$ We now introduce our conditions.

\begin{changemargin}{4mm}
{\bf P1 }{\it(Invariance and ergodicity)}.
    For each $u\in{I},$ $\bbP^u$ is invariant and ergodic with respect to lattice shifts, that is $\mathbb{P}^u\circ \tau_x^{-1}=\mathbb{P}^u$ for all $x\in\mathbb{Z}^d$ and, for each $x\in{\bbZ^d},$ $\mathbb{P}^u(A)\in\{0,1\}$ for any $A\in\mathcal{F}$ such that $\tau_x(A)=A.$
\end{changemargin}

\begin{changemargin}{4mm}
{\bf P2 }{\it(Monotonicity)}.
    For all increasing functions $f:\Omega\rightarrow[0,1]$ and $u,u'\in{I}$ with $u\leq u',$ $\bbE^{u}[f]\leq\bbE^{u'}[f]$.
 \end{changemargin}

\begin{changemargin}{4mm}
{\bf P3} {\it(Decoupling inequality)}. 
After possibly extending the probability space underlying $\bbP^u$, we assume there exist positive constants $R_P,L_P<\infty,$ $\xi_P>1,$ $\eps_P,a_P,C_P>0$ and $\chi_P>0$ such that for all $R\geq R_P,$ $L\geq L_P,$ $x_1,x_2\in{\bbZ^d}$ with $d\big(\bar{Q}(x_1,L^{\xi_P}),\bar{Q}(x_2,L^{\xi_P})\big)\geq RL$, and any $u,\hat{u}\in{I}$ with 
\begin{equation*}
    u\geq\hat{u}+R^{-\chi_P},
\end{equation*}
the following holds.  There exists an event $B$ with 
\begin{equation}
\label{eq:boundonprobaB^c}
    \bbP^{\hat{u}}(B^c)\leq \exp\big(-C_Pf_P(L)\big),
\end{equation} 
with $f_P:[0,\infty)\rightarrow[0,\infty)$ satisfying $f_P(L)\geq \exp\big(a_P\log(L)^{\eps_P}\big),$ such that for all increasing events $A\subset\prod_{e\in{{\bar{Q}(x_1,L^{\xi_P})}}}\Omega_{e}\times \prod_{e\in{\bar{Q}(x_2,L^{\xi_P})}}\Omega_{e},$
\begin{equation}
\label{eq:P3'replacebyindependentinc}
\begin{split}
    &\bbP^{\hat{u}}\big(B,\big((\omega_e)_{e\in{\bar{Q}(x_1,L^{\xi_P})}},(\omega_e)_{e\in{\bar{Q}(x_2,L^{\xi_P})}}\big)\in{A}\big)
    \\&\leq\hat{\bbP}^{{u}}\big(\big((\omega_e^{(1)})_{e\in{\bar{Q}(x_1,L^{\xi_P})}},(\omega_e^{(2)})_{e\in{\bar{Q}(x_2,L^{\xi_P})}}\big)\in{A}\big);
\end{split}
\end{equation}
and there exists an event $B'$ with probability bounded as in \eqref{eq:boundonprobaB^c} but replacing $\hat{u}$ by $u,$ such that for all decreasing events $A\subset\prod_{e\in{{\bar{Q}(x_1,L^{\xi_P})}}}\Omega_e\times \prod_{e\in{\bar{Q}(x_2,L^{\xi_P})}}\Omega_e,$
\begin{equation}
\label{eq:P3'replacebyindependentdec}
\begin{split}
    &\bbP^{{u}}\big(B',\big((\omega_e)_{e\in{\bar{Q}(x_1,L^{\xi_P})}},(\omega_e)_{e\in{\bar{Q}(x_2,L^{\xi_P})}}\big)\in{A}\big)
    \\&\leq\hat{\bbP}^{\hat{u}}\big(\big((\omega_e^{(1)})_{e\in{\bar{Q}(x_1,L^{\xi_P})}},(\omega_e^{(2)})_{e\in{\bar{Q}(x_2,L^{\xi_P})}}\big)\in{A}\big).
\end{split}
\end{equation}
\end{changemargin}

In essence condition {\bf P3} states that one can replace the weights on $\bar{Q}(x_1,L^{\xi_P})$ and $\bar{Q}(x_2,L^{\xi_P})$ by independent weights under an event $B$ which happens with large probability when $x_1$ and $x_2$ are far apart, at the cost of adjusting slightly the level $u$ by a sprinkling parameter.  We allow for a possible extension of the probability space in condition {\bf P3} so that the event $B$ can depend on some other variables than $\omega\in{\Omega},$ see for instance the proof of Proposition \ref{prop:upperratioweakmixing}. Conditions {\bf P1} and {\bf P2} have been initially introduced in \cite{DRS14}, and our condition {\bf P3} is a stronger version of the condition {\bf P3} in  \cite{DRS14}. Indeed, one can combine \eqref{eq:P3'replacebyindependentinc} with \eqref{eq:boundonprobaB^c} and the independence of $\omega^{(1)}$ and $\omega^{(2)}$ to show that for all increasing events $A_1\subset[0,\infty)^{\bar{Q}(x_1,L^{\xi_P})}$ and $A_2\in{[0,\infty)^{\bar{Q}(x_2,L^{\xi_P})}}$, 
\begin{equation*}
\begin{split}
    &\bbP^{\hat{u}}\big((\omega_e)_{e\in{\bar{Q}(x_1,L^{\xi_P})}}\in{A_1},(\omega_e)_{e\in{\bar{Q}(x_2,L^{\xi_P})}}\big)\in{A_2}\big)
    \\&\leq{\bbP}^{{u}}\big((\omega_e)_{e\in{\bar{Q}(x_1,L^{\xi_P})}}\in{A_1}\big) \,\bbP^u\big((\omega_e)_{e\in{\bar{Q}(x_2,L^{\xi_P})}}\in{A_2}\big)+\exp(-C_Pf_P(L)),
\end{split}
\end{equation*}
and similarly for decreasing events. All the examples from \cite{DRS14} of models satisfying condition {\bf P3} in \cite{DRS14} also satisfy our condition {\bf P3}, see Section~\ref{sec:examples}. In \cite{2016arXiv161202385R}, it is proved that the Ginzburg-Landau $\nabla \phi$ interface model also satisfies condition {\bf P3} in \cite{DRS14}. It is however not clear whether this model also satisfies our stronger condition {\bf P3}, and we now introduce another condition on the correlations of $\omega,$ which is satisfied by the interface model, see Section~\ref{sec:ginzburglandau}.

\begin{changemargin}{4mm}
{\bf P3'} {\it(Decoupling inequality)}. 
There exist positive constants $R_P,L_P<\infty,$ $\xi_P>1$ and $C_P,\chi_P>0$ such that for all $R\geq R_P,$ $L\geq L_P,$ $x_1,x_2\in{\bbZ^d}$ and any  $u,\hat{u}\in{I}$ with 
\begin{equation*}
    u\geq\hat{u}+R^{-\chi_P},
\end{equation*}
the following holds. For all sets $S_i\subset \bar{Q}(x_i,L^{\xi_P})$ with $|S_i|\leq L^{\xi_P},$ $i\in{\{1,2\}},$ satisfying $d(S_1,S_2)\geq RL,$ and all increasing functions $f_i:\Omega\rightarrow[0,1]$ supported on $S_i,$ $i\in{\{1,2\}},$
\begin{equation}
\label{eq:P3''replacebyindependentinc}
    \bbE^{\hat{u}}[f_1f_2]\leq\bbE^{u}[f_1] \, \bbE^{{u}}[f_2]+\exp\big(-C_PL^{\xi_P}\big),
\end{equation}
and for all decreasing functions $f_i:\Omega\rightarrow[0,1]$ supported on $\prod_{e\in{S_i}}\Omega_e,$ $i\in{\{1,2\}},$
\begin{equation}
\label{eq:P3''replacebyindependentdec}
    \bbE^{{u}}[f_1f_2]\leq\bbE^{\hat{u}}[f_1] \,\bbE^{\hat{u}}[f_2]+\exp\big(-C_PL^{\xi_P}\big).
\end{equation}
\end{changemargin}
Condition {\bf P3'} is on the one hand weaker than condition {\bf P3} since it only requires to decouple sets $S_i$ of cardinality $L^{\xi_P}$ instead of balls of cardinality $L^{\xi_Pd},$ but stronger on the other hand since it requires at least super-exponential decay of the correlation, after sprinkling. As we now explain, either of the conditions ${\bf P3}$ and ${\bf P3'}$ is in fact enough to get a comparison result of the FPP distance and the Euclidean metric. For any $\delta>1$ and $L>1$ let us introduce
\begin{equation}
    \label{eq:defgp}
    g_P^{(\delta)}(L)=\begin{cases}
    L\wedge\left({\log(L)^{-\delta(1+a_P+\frac{a_P}{\chi_P})}}\exp\big(a_P\log(L)^{\eps_P}\big)\right)&\text{ if condition {\bf P3} holds,}
    \\L&\text{ if condition {\bf P3'} holds.}
    \end{cases}
\end{equation}
If both condition {\bf P3} and {\bf P3'} hold at the same time, we just take $g_P^{(\delta)}(L)=L.$ 
\begin{theorem}
\label{The:main}
Suppose that $(\bbP^u)_{u\in I}$ satisfies {\bf P2} and either {\bf P3} or {\bf P3'}. Further, assume 
\begin{equation}
\label{intro:BLnotconnectedtoB2L}
    \liminf_{L\rightarrow\infty}\sup_{x\in{\bbZ^d}}\bbP^u\Big(Q(x,L)\longleftrightarrow Q(x,2L)^c\text{ in }\{e\in{E_d}:\,t^\om_e=0\}\Big)=0, \quad \text{for all }u\in{I}.
\end{equation}
Then, for all $u\in{I}$ and $\delta>1$, there exist constants $c_8,c_9,C_{\mathrm{fpp}}>0$ such that for all $x\in{\bbZ^d\setminus\{0\}}$,
\begin{equation}
\label{intro:boundonT0nx}
    \bbP^u\Big(d^{\omega}(0,nx)\leq C_{\mathrm{fpp}}n\Big) \leq c_8 \exp\Big(-c_9 \, g_P^{(\delta)}(n)\Big), \qquad \forall n\in \bbN.
\end{equation}
In particular, if in addition $(\bbP^u)_{u\in I}$ satisfies {\bf P1}, then, for each $u\in{I}$ such that \eqref{eq:boundonmement} holds, we have that $\mu_u(x)\geq C_{\mathrm{fpp}}>0$ for all $x\in\bbZ^d\setminus\{0\}.$
\end{theorem}
We prove Theorem~\ref{The:main} in the remainder of this section, namely in Subsection~\ref{sec:proofunderP3}  under condition {\bf P3} and in Subsection~\ref{sec:proofunderP3'} under condition {\bf P3'}. In fact, our proof of Theorem~\ref{The:main} under condition {\bf P3} still works if $f_P(L)$ only increases as $\log(L)^{\delta_P}$ for some $\delta_P>1,$ and we refer to Remark~\ref{rk:weakerP3} for a weaker condition {\bf P3''} under which one can still prove Theorem~\ref{The:main}. 

When considering first passage percolation, Theorem~\ref{The:main} is more general than \cite[Theorem~2.5]{DeGa}. Indeed, the quasi-independence hypothesis therein is not satisfied for any of the examples we consider in Sections~\ref{sec:gff}, \ref{sec:ginzburglandau} and \ref{sec:inter}, as the correlations of all these models are polynomial, and not stretched-exponential as required in \cite[Assumption~2.4.6]{DeGa}. This improvement is mainly due to the additional sprinkling parameter in conditions {\bf P3} and {\bf P3'}, which boosts the polynomial correlations to superpolynomial decoupling inequalities, up to sprinkling.  Moreover, our condition \eqref{intro:BLnotconnectedtoB2L} is also more general than the decay of instant one-arms condition from \cite[Assumption~2.4.5]{DeGa}, as we do not need any explicit rate on the decay of the probability in \eqref{intro:BLnotconnectedtoB2L}. Finally we only consider first passage percolation distance on $\bbZ^d$ for simplicity, but we believe that our results could easily be extended to general pseudo-metric on $\mathbb{R}^d$ similarly as in \cite[Theorem~2.5]{DeGa}. 

Note that Theorem \ref{The:main} only gives us the inequality $d^{\omega}(0,nx)\geq C_{\text{fpp}}n$ with high probability, but it does not tell whether $d^{\omega}(0,nx)\leq cn$ with high probability for some constant $c<\infty.$ In fact, this is often not the case, for instance when $\{e\in{E_d}:\,t^\om_e<\infty\}$ does not percolate. Under the additional assumption that the time constant $\mu_u$ exists and is finite, for instance under conditions \eqref{eq:boundonmement} and {\bf P1}, it is however clear that $d^{\omega}(0,nx)\leq cn.$ In the special case when $t^\om_e$ can only either be $1$ or infinity, the distance $d^{\omega}$ reduces to the chemical distance on the subgraph induced by $\{e\in{E_d}:\,t^\om_e<\infty\},$ and it is then proved in \cite{DRS14} that $d^{\omega}(0,nx)\leq cn$ with high probability under conditions {\bf P1}, {\bf P2} and {\bf P3} together with additional conditions {\bf S1} and {\bf S2} which essentially correspond to the existence of a locally unique infinite cluster in $\{e\in{E_d}:\,t^\om_e<\infty\}.$ This will be useful in Section~\ref{sec:RCM} below.

\begin{remark}
\label{rk:fppvertices}
(i) One can directly deduce from Theorem~\ref{The:main} similar results for first passage percolation on the vertices of $\bbZ^d.$ More precisely, suppose that, for each $x\in{\bbZ^d},$ $t^\om_x=f(\omega_x)$ for some measurable, monotone function $f:\Omega_V\rightarrow[0,\infty)\cup\{\infty\},$ and denote by $d^{\omega}_V(x,y)$ the infimum of $\sum_{z\in{\pi}}t^\om_z,$ where the infimum is taken over all simple paths $\pi$ of vertices connecting $x$ to $y,$ similarly as in \eqref{intro:defT}. Then if $(\bbP^u)_{u\in I}$ satisfies {\bf P2}, either {\bf P3} or {\bf P3'}, and \eqref{intro:BLnotconnectedtoB2L} (replacing $\{e\in{E_d}:\,t^\om_e=0\}$ by $\{x\in{\bbZ^d}:\,t^\om_x=0\}$), then \eqref{intro:boundonT0nx} also holds for $d^{\omega}_V$ instead of $d^{\omega}.$ Indeed, one can simply take $\mathcal{N}=\{0\}$ and $\mathbf{t}(\omega_V^1,\omega_V^2,\omega_E)=\frac12(f(\omega_V^1)+f(\omega_V^2)),$ which satisfies \eqref{eq:temonotone} and implies $d^{\omega}(x,y)\leq d^{\omega}_V(x,y)\leq 2d^{\omega}(x,y)$ for all $x,y\in{\bbZ^d},$ and apply Theorem~\ref{The:main} for this choice of $\mathbf{t}.$

(ii) If {\bf P3}  (resp.\ {\bf P3'}) actually holds without the monotonicity assumption on the event $A$ (resp.\ functions $f_i$) then we can  remove the monotonicity assumption  \eqref{eq:temonotone} on $\mathbf{t}$ in order for Theorem~\ref{The:main} to hold. Indeed, one can then simply consider the trivial partial order $x\leq y$ if and only if $x=y$ for all $x,y\in{\Omega_V},$ and similarly on ${\Omega_E},$ for which there are only monotone functions, and apply Theorem~\ref{The:main} for this choice of partial order. We refer to Section \ref{sec:ratioweakmixing} for examples of models which satisfy condition {\bf P3} without the monotonicity assumption on $A.$

(iii) It might be possible to prove that the time constant $\mu_u$ is positive under the same assumptions as in Theorem~\ref{The:main}, but replacing the condition {\bf P3} by its weaker version from \cite{MR3390739}, or even by the condition {\bf D} from \cite[Section~6]{AlvesCaio2019Dias}. More precisely, one could define a notion of a ``good'' box as a box of size length $2L_0$ such that the first passage percolation distance between  this box and the twice bigger concentric box is larger than some constant, which occurs with good probability for $L_0$ large enough by \eqref{intro:BLnotconnectedtoB2L}. Then considering a renormalization scheme similar to the one introduced in \cite[Section~2]{Sa17}, one can show that for any $N\in{\bbN}$ and with high probability, any path from $Q(0,L_0N)$ to $Q(0,2L_0N)$ will cross $cN$ good boxes along the ``perforated lattice'' illustrated in \cite[Figure~2]{Sa17} (for some constant $c$), and thus the first passage percolation distance between these two boxes is at least proportional to $N.$ However this renormalization scheme from \cite{Sa17} would only give superpolynomial decay of the probability in \eqref{intro:boundonT0nx} (similar to the one in \cite[Equation~(8)]{Sa17}). Our bound \eqref{intro:boundonT0nx} is much sharper as it can be for instance exponential when $f_P$ is superlinear, and is in fact often optimal for our examples, see for instance Remark~\ref{rk:optidecay}. Moreover having exponential decay in \eqref{intro:boundonT0nx} can be useful in future applications, for instance if one wants to prove annealed versions of the heat kernel or Green's function bounds from Theorems~\ref{thm:hkeGauss} and \ref{thm:GKdecay_killing}.
\end{remark}

Next we explain how the results in \cite{MR1074741} can be used to deduce a shape theorem from Theorem~\ref{The:main}, see \cite{MR624685} for the i.i.d.\ case. Let $L(d,1)$ denote the space of functions on $\Omega$ with finite Lorentz norm.  In \cite{MR1074741} it has been shown  that if there exists $u\in{I}$ such that $t^\om_e\in{L(d,1)}$ for all $e\in{E_d}$ and if $\bbP^u$ is invariant and ergodic with respect to the lattice shifts \eqref{eq:def:space_shift}, then there exists a continuous and non-negative function $m_u$ on $\{x\in{\bbR^d}:|x|_1=1\}$ such that
\begin{align}
\label{intro:T0xconvergesuniformly}
    \lim_{\substack{|x|_1\rightarrow\infty\\x\in{\bbZ^d}}}\frac{d^{\omega}(0,x)}{|x|_1}-m_u\Big(\frac{x}{|x|_1}\Big)=0, \qquad \text{$\bbP^u$-a.s.}
\end{align}
By \eqref{intro:deftimeconstant} we have $\mu_u(x)=|x|_1m_u(x/|x|_1)$ for all $x\in\bbZ^d,$ and so $m_u$ is strictly positive under the conditions of Theorem \ref{The:main}. Taking $K_u=\{x\in{\bbR^d}:\,|x|_1m_u(x/|x|_1)\leq1\}$ and defining $B_{d^\omega}(t)=\{x\in{\bbZ^d}:d^{\omega}(0,x)\leq t\}$ the following result is a classical consequence of \eqref{intro:T0xconvergesuniformly}, see e.g.\ the proof of  \cite[Proposition~1]{MR3531712} for details.

\begin{coro}
\label{cor:shape}
Suppose that $(\bbP^u)_{u\in I}$ satisfies  {\bf P1}, {\bf P2}, either {\bf P3} or {\bf P3'}, and \eqref{intro:BLnotconnectedtoB2L}. Then, for any $u\in{I}$ such that $t^\om_e\in{L(d,1)}$ for all $e\in{E_d}$ under $\bbP^u,$ there exists a compact and convex deterministic set $K_u\subset\bbR^d$ with non-empty interior such that
\begin{equation*}
    \bbP^u\Big((1-\eps)K_u\subset\frac{B_{d^\omega}(t)}{t}\subset(1+\eps)K_u \quad\text{for all large }t\Big)=1, \qquad \text{ for all $\eps>0$}.
\end{equation*}
\end{coro}

\begin{remark}
We can use the stronger convergence result \eqref{intro:T0xconvergesuniformly} to obtain, under a suitable moment condition, a partial converse of the statement  that \eqref{intro:BLnotconnectedtoB2L} implies $\mu_u(x)>0$ for all $x\in{\bbZ^d}$ and $u\in{I}$, formulated in  Theorem~\ref{The:main}. More precisely, for any fixed $u\in{I}$, suppose that $\bbP^u$ is invariant and ergodic with respect to lattice shifts,  $t^\om_e\in{L(d,1)}$ for all $e\in{E_d}$ under $\bbP^u.$ There exists a constant $c=c(u)\in{(1,\infty)}$ such that if
\begin{equation}
\label{eq:BLconnectedtoB2L}
    \liminf_{L\rightarrow\infty}\sup_{x\in{\bbZ^d}}\bbP^u\big(Q(x,L)\longleftrightarrow Q(x,cL)^c\text{ in }\{e\in{E_d}:\,t^\om_e=0\}\big)>0,
\end{equation}
then there exists $x\in{\bbZ^d}$ such that $\mu_u(x)=0.$ Note that, contrary to \eqref{intro:BLnotconnectedtoB2L}, we do not require that \eqref{eq:BLconnectedtoB2L} holds for all $u$ in an interval, and so one could ignore  the fixed parameter $u$ here. Let us now explain the proof. Write $d^{\omega}(A,B)=\inf_{x\in{A},y\in{B}}d^{\omega}(x,y)$ for all $A,B\subset\bbZ^d$. Then it  follows from \eqref{intro:T0xconvergesuniformly} that
\begin{equation*}
    \frac{d^{\omega}(0,Q(L)^c)}{L}=\inf_{x\in{Q(L)^c}}\frac{d^{\omega}(0,x)}{|x|_1}\times\frac{|x|_1}{L}\tend{L}{\infty}\inf_{y:|y|_1=1}\frac{m_u(y)}{|y|_{\infty}},
\end{equation*}
and a similar statement holds  for $\sup_{x\in{\partial Q(L)}}d^{\om}(0,x)/L$. Further,
\begin{equation*}
\begin{split}
d^{\om}(0,Q(cL)^c)&=\inf_{x\in{\partial Q(L)}}\Big(d^{\om}(x,Q(cL)^c)+d^{\om}(0,x)\Big)
\\&\leq d^{\om}(Q(L),Q(cL)^c)+\sup_{x\in{\partial Q(L)}}d^{\om}(0,x)
\end{split}
\end{equation*}
since $d^{\om}(\partial Q(L),Q(cL)^c)=d^{\om}(Q(L),Q(cL)^c).$
Therefore, if $\mu_u(x)>0$ for all $x\in{\bbZ^d},$  one can take the constant $c$ large enough so that
\begin{equation*}
\begin{split}
    \liminf_{L\rightarrow\infty}\frac{d^{\omega}(Q(L),Q(cL)^c)}{L}&\geq  \liminf_{L\rightarrow\infty}\frac{d^{\omega}(0,Q(cL)^c)}{L}-\frac{\sup_{x\in{\partial Q(L)}}d^{\omega}(0,x)}{L}
    \\&=c\inf_{y:|y|_1=1}\frac{m_u(y)}{|y|_{\infty}}-\sup_{y:|y|_1=1}\frac{m_u(y)}{|y|_{\infty}}\geq \inf_{y:|y|_1=1}\frac{m_u(y)}{|y|_{\infty}}.
\end{split}
\end{equation*}
Taking a continuous and bounded function $f:[0,\infty)\rightarrow[0,\infty)$ with $f(0)=1$ and $f(t)=0$ for all $t\geq \inf_{y:|y|_1=1}\frac{m_u(y)}{|y|_{\infty}},$ and we can conclude that \eqref{eq:BLconnectedtoB2L} does not hold similarly as in \eqref{eq:proofProp1.1}.
\end{remark}

\subsection{Renormalization}
\label{sec:renormalization}
In order to prove our main result, Theorem \ref{The:main}, we first introduce a renormalization scheme similar to the one in  \cite[Section~7]{MR3420516}. For some $\delta>1,$ $\rho>0,$ $K\in\bbN_0,$ $L_0>0$ we define recursively
\begin{equation}
\label{eq:defLk}
    L_{k+1}\ldef 
    2L_k\Big(1+\frac{\rho_k}{(k+6)^\delta}\Big),
\end{equation}
where
\begin{equation}
\label{eq:defrhok}
\rho_k=
\begin{cases}
    \rho&\text{ if } k\in{\{0,\dots,K-1\}},
    \\1&\text{ if } k\geq K.
    \end{cases}
\end{equation}
Note that there exists a constant $c_{10} <\infty$ only depending on $\delta,$ $\rho$ and $K$ such that
\begin{equation}
\label{ren:boundonL_k}
    2^kL_0\leq L_k\leq c_{10} 2^k L_0, \quad \forall k\in\bbN.
\end{equation}
 For all $x\in \bbZ^d$ and $n\in \bbN$ we also define
\begin{equation*}
    C_x^k\ldef x+\big[0,L_k\big)^d\cap\bbZ^d, \qquad D_x^k\ldef x+\big[-L_k,2L_k\big)^d\cap\bbZ^d.
\end{equation*}
Proceeding similarly as in \cite{MR3420516}, one can show that for all $k\in\bbN_0$ there exist two sequences $(x_i^k)_{i\in{\{1,\dots,(2+\rho_k)^d\}}}$ and $(y_i^k)_{i\in{\{1,\dots,2d(6+\rho_k)^{d-1}\}}}$ such that for every paths $\pi$ of edges between $C_0^{k+1}$ and $(D_0^{k+1})^c,$ there exist $i\in{\{1,\dots,(2+\rho_k)^d\}}$ and $j\in{\{1,\dots,2d(6+\rho_k)^{d-1}\}}$ such that
\begin{itemize}
    \item $\pi$ connects $C_{x_i^k}^k$ to $(D_{x_i^k}^k)^c,$ 
    \item $\pi$ connects $C_{y_j^k}^k$ to $(D_{y_j^k}^k)^c,$ 
    \item and
    \begin{equation}
    \label{ren:Dxiyj}
        d\big(D_{x_i^k}^k,D_{y_j^k}^k\big)\geq\frac{L_k\rho_k}{(k+6)^\delta}.
    \end{equation}
\end{itemize}
Let us now define the length of the shortest path between $C^k_x$ and $D^k_x$ by
\begin{equation}
\label{ren:defTK}
    d^{\omega}_k(x)=\inf_{\pi:\,C^k_x\stackrel{\pi}\leftrightarrow (D^k_x)^c}\sum_{e\in\pi}t^\om_e,
\end{equation}
where the infimum is taken over all simple paths $\pi$ of edges connecting $C^k_x$ to $(D^k_x)^c.$ Note that this infimum is necessarily reached for a path included in $D^k_x\setminus C^k_x,$ except at its end points. Moreover, for all $k\in\bbN_0$ and $x\in{\bbZ^d},$ there exist $i\in{\{1,\dots,(2+\rho_k)^d\}}$ and $j\in{\{1,\dots,2d(6+\rho_k)^{d-1}\}}$ such that \eqref{ren:Dxiyj} holds and 
\begin{equation}
    \label{ren:TkvsTk-1}
    d^{\omega}_{k+1}(x)\geq d^{\omega}_k(x+x_i^k)+d^{\omega}_k(x+y_j^{k}),
\end{equation}
see Figure~\ref{F:DR} for details.

\begin{figure}[ht!]
  \centering 
  \includegraphics[scale=0.7]{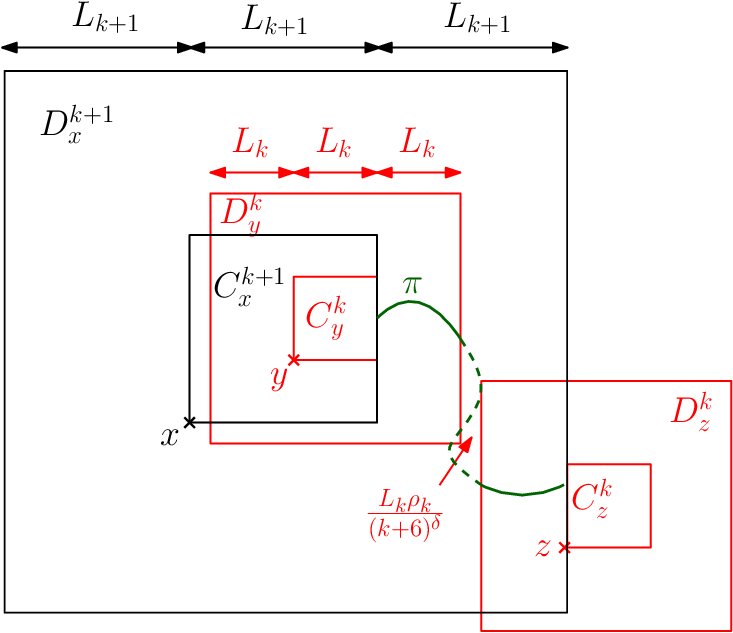}
  \caption{Illustration of the renormalization scheme. For each path $\pi$ (in green) which connects $C_x^{k+1}$ to $(D_x^{k+1})^c$ there exist $i\in{\{1,\dots,(2+\rho_k)^d\}}$ and $j\in{\{1,\dots,2d(6+\rho_k)^{d-1}\}}$ such that $\pi$ connects both $C_y^k$ to $(D_y^k)^c,$ for $y=x+x_i^k,$ and $C_z^k$ to $(D_z^k)^c,$ for $z=x+y_j^k.$    
 \label{F:DR}}
\end{figure}

\subsection{Proof of Theorem \ref{The:main} under condition {\bf P3}}
\label{sec:proofunderP3}
Throughout this subsection, we assume that conditions {\bf P2}, {\bf P3} and \eqref{intro:BLnotconnectedtoB2L} hold. The general strategy to prove Theorem~\ref{The:main} will be to bound $d_k^{\om}$ recursively on $k$ using \eqref{ren:TkvsTk-1} combined with the decoupling inequality {\bf P3}, and \eqref{intro:BLnotconnectedtoB2L} will let us initiate the recursion. At each step of the iteration, one needs to change the parameter $u$  in the probability $\bbP^u$ we consider in {\bf P3} by a sprinkling parameter, and in order to not change the parameter $u$ too drastically at the end of the iteration we will consider a converging sequence of sprinkling parameters $\eps_{k}-\eps_{k+1},$ see \eqref{pr2:defepskak}. This leads to bounding $d_{k+1}^{\om}$ by two independent copies of $d_k^{\om},$ up to sprinkling. In order to complete the iteration at all levels, we will thus actually need bounds on the sum of $p$ independent copies of $d_k^{\om}$ for each $p\in{\bbN}$ (or each $p$ which is a power of $2$), see Lemma~\ref{pr2:mainLemma}. The final bound we obtain will either be dominated by the error \eqref{eq:boundonprobaB^c} in the decoupling inequality if $f_P$ is small, or will simply be exponential as in the case of independent weights if $f_P$ is large, which leads to \eqref{intro:boundonT0nx} with $g_p^{(\delta)}$ from \eqref{eq:defgp}. 

For any $\delta>1$ and $\eps>0$ set
\begin{equation}
\label{pr2:defepskak}
    \eps_k\ldef \sum_{p=k}^{\infty}\frac{\eps}{(p+6)^\delta}, \quad 
    a_k\ldef 2^k\prod_{i=0}^{k-1}\Big(1-\frac{1}{(i+6)^{\delta}}\Big), \quad k\in\bbN_0.
\end{equation}
We will use $\eps_k-\eps_{k+1}$ as a sprinkling parameter at each step of the iteration, and $a_k$ as a bound on $d_k^{\om}$ (up to constants), which is of order $L_k$ in view of \eqref{ren:boundonL_k}. The reason we do not simply take $a_k=2^k,$ which could intuitively be enough in view of \eqref{ren:TkvsTk-1}, is that in the proof of Lemma~\ref{pr2:mainLemma} below we will only be able to use {\bf P3} to decouple a good proportion of $p$ independent copies of $d_k^{\om}(x+x_i^k)$ and $d_k^{\om}(x+y_j^k),$ and not all of them at once, see \eqref{pr2:usedecoupling} and \eqref{pr2:boundonCij}.

Let us denote by $\omega_e^{(i)}$ (resp.\ $\omega_x^{(i)}$) the canonical projection on the edge  $e$ (resp.\ vertex $x$) of the $i$-th coordinate of an element in $(\Omega^{\bbN},\cF^{\otimes\bbN},\tilde{\bbP}^u),$ where $\tilde{\bbP}^u=(\bbP^u)^{\otimes\bbN}$. Let $t_e^{\om, (i)}:=\mathbf{t}\big(\omega_{|\mathcal{N}+x}^{(i)},\omega_{|\mathcal{N}+y}^{(i)},\omega_e^{(i)}\big)$ for all $e=\{x,y\}\in{E_d}$ and $i\in{\bbN},$ so that, under $\tilde{\bbP}^u,$ $\big(t_e^{\om, (i)}\big)_{e\in{E_d}},$ $i\in{\bbN},$ are independent copies of $(t^\om_e)_{e\in{E_d}}$ under $\bbP^u.$ Let the random variables $d^{\omega,(i)}_k(\cdot)$ be defined as in \eqref{ren:defTK} but with $t^\om_e$ replaced by $t_e^{\om,(i)},$ which are i.i.d.\ in $i\in{\bbN}$.  For $\delta>1$  one can easily check that, setting
\begin{equation}
\label{eq:defhp}
    h_P^{(\delta)}(L)\ldef \frac{1}{\log(L)^{\delta}} \, f_P\Big(\frac{L}{\log(L)^{\delta(1+1/\chi_P)}}\Big),
\end{equation}
the function $g_P^{(\delta)}$ defined in \eqref{eq:defgp} satisfies
\begin{equation}
\label{intro:gputsmallerugpt}
    g_P^{(\delta)}(uL)\leq u \, g_P^{(\delta)}(L) \quad \text{ for all }u\geq1\text{ and }L\text{ large enough}
\end{equation}
and 
\begin{equation}
\label{intro:gpsmallthanhp}
g_P^{(\delta)}(L)\leq h_P^{(\delta)}(L) \quad \text{ for all }L\text{ large enough}.
\end{equation} 
Recall that the definition of $d_k^{\omega}$ in \eqref{ren:defTK} depends on the choice of the scale $L_k,$ and thus on $\rho,$ $K,$ $\delta$ and $L_0$ in view of \eqref{eq:defLk}. We are now going to use \eqref{intro:gputsmallerugpt} and \eqref{intro:gpsmallthanhp} together with conditions {\bf P2}, {\bf P3} and the  renormalization scheme introduced in Subsection~ \ref{sec:renormalization} to prove inductively the following. 

\begin{lemma}
\label{pr2:mainLemma}
Fix $\rho=1$ and $K=0.$ For all $\delta>1$ and $u\in{I}$, there exist $\eps$, $L_0=L_0(\eps)$,  and constants $c_{11}=c_{11}(\eps,L_0)$, $c_{12}>0$, all also depending on $\delta$ and $u$, such that $u-\eps_0\in{I}$ and for all $k\in\bbN_0$, 
\begin{equation}
    \label{pr2:induction}
    \tilde{\bbP}^{u-\eps_k}\bigg(\!\sum_{n=1}^pd^{\omega,(n)}_k(x)\leq c_{11} pa_k \!\bigg)
    \leq\bigg(\!\frac{1}{(4d\cdot21^d)^2}\exp\Big(\!-c_{12} g_P^{(\delta)}(a_k)\Big)\!\bigg)^{\! p}, \forall\,p\in\bbN,x\in{\bbZ^d}.
\end{equation}
\end{lemma}

Before we prove Lemma \ref{pr2:mainLemma}, let us first explain how it implies Theorem~\ref{The:main} under condition {\bf P3}.

\begin{proof}[Proof of Theorem \ref{The:main} under condition {\bf P3}]
Fix $u\in{I},$ $\delta>1,$ and take $L_0,$ $\eps$, $\rho$ and $K$ as in Lemma \ref{pr2:mainLemma}. For all $x\in{\bbZ^d\setminus\{0\}}$ and $n\geq L_0$ there exists $k\in\bbN_0$ such that $nx\in{D^{k+1}_0\setminus D^{k}_0},$ and then
\begin{equation*}
    d^{\omega}(0,nx)\geq d^{\omega}_k(0).
\end{equation*}
Moreover by \eqref{ren:boundonL_k} and \eqref{pr2:defepskak} there exist constants $c,c'>0$ depending on $x,$ $u,$ $L_0$ and $\delta,$ such that $c'a_k\leq n\leq ca_k,$ and so \eqref{intro:boundonT0nx} follows readily from \eqref{intro:gputsmallerugpt} and \eqref{pr2:induction} with $p=1$ and condition {\bf P2}. Assume now that the time constant $\mu_u(x)$ from \eqref{intro:deftimeconstant} exists, for instance under the conditions \eqref{eq:boundonmement} and {\bf P1}. Using \eqref{eq:defgp} and \eqref{intro:boundonT0nx}, by the Borel-Cantelli lemma $d^{\omega}(0,nx)\geq C_{\text{fpp}}n$ for all $n$ large enough  $\bbP^u$-a.s, and so $\mu_u(x)>0.$ 
\end{proof}

\begin{proof}[Proof of Lemma \ref{pr2:mainLemma}]
We fix $\rho=1,$ $K=0,$ $\delta>1$ (see \eqref{eq:defLk}) and $u\in{I},$ on which all the constants in the rest of this proof depend. We shall prove \eqref{pr2:induction} by induction. For $k=0,$ assuming that $\eps$ is small enough so that $u-\eps_0\in{I},$ we have by the Chernoff bound for all $p\in\bbN,$ $x\in{\bbZ^d}$ and $\zeta>0$,
\begin{equation}
\label{pr2:Chernov}
     \tilde{\bbP}^{u-\eps_0}\Big(\sum_{n=1}^p d^{\omega,(n)}_0(x)\leq c_{11} pa_0\Big)\leq \exp(c_{11} \zeta a_0 p) \, \bbE^{u-\eps_0}\big[\exp(-\zeta d^{\omega}_0(x))\big]^p.
\end{equation}
Moreover by dominated convergence,
\begin{align*}
   &  \bbE^{u-\eps_0}\big[\exp(-\zeta d^{\omega}_0(x))\big]
   \,\tend{\zeta}{\infty} \, 
  \bbP^{u-\eps_0}\big(d^{\omega}_0(x)=0\big) \\
    & \mspace{36mu}
   \leq\sup_{y\in{\bbZ^d}}\bbP^{u-\eps_0}\Big(Q(y,\lceil L_0/2\rceil)\longleftrightarrow Q(y,2\lceil L_0/2\rceil)^c\text{ in }\big\{e\in{E_d}:\,t^\om_e=0\big\}\Big).
\end{align*}
By \eqref{intro:BLnotconnectedtoB2L} and \eqref{pr2:Chernov}, we can thus fix $L_0=L_0(\eps),$ $\zeta=\zeta(\eps,L_0)$ and $c_{11}=c_{11}(\eps,L_0,\zeta)$ such that  \eqref{pr2:induction} holds for $k=0$ (for a constant $c_{12}$ to be determined, independently of $\eps,$ $L_0,$ $\zeta$ or $c_{11}$). Let us now assume that \eqref{pr2:induction} holds for any $k\in\bbN_0,$ and let us fix some $p\in\bbN.$ Then by \eqref{ren:TkvsTk-1} we have
\begin{equation}
\label{pr2:Tk+1vsTk}
\begin{split}
    &\tilde{\bbP}^{u-\eps_{k+1}}\bigg(\sum_{n=1}^pd^{\omega,(n)}_{k+1}(x)\leq c_{11} pa_{k+1}\bigg)
    \\&\leq \sum_{i=1}^{3^d}\sum_{j=1}^{2d\cdot7^d} \tilde{\bbP}^{u-\eps_{k+1}}\bigg(\sum_{n=1}^pd^{\omega,(n)}_{k}(x+x_i^k)+d^{\omega,(n)}_{k}(x+y_j^k)\leq c_{11} pa_{k+1}\bigg).
\end{split}
\end{equation}
Let us take
\begin{equation}
\label{pr2:defRL}
    R=\left(\frac{(k+6)^{\delta}}{\eps}\right)^{\frac1{\chi_P}} \quad \text{and} \quad L=\frac{\eps^{1/\chi_P}2^{k-1}L_0}{(k+6)^{\delta(1+1/\chi_P)}}.
\end{equation}
After choosing $\eps$ small enough and $L_0=L_0(\eps)$ large enough, independently of $k,$ one can use \eqref{ren:boundonL_k} and \eqref{pr2:defepskak} to check that
\begin{equation*}
    u-\eps_{k+1}\geq u-\eps_k+R^{-\chi_P},\quad R\geq R_P,\quad L\geq L_P, \quad L^{\xi_P}\geq4L_k.
\end{equation*}
Using \eqref{ren:boundonL_k}, \eqref{ren:Dxiyj} and {\bf P3} we have by independence that for each $x\in{\bbZ^d},$ $i\in{\{1,\dots,3^d\}}$ and $j\in{\{1,\dots,2d\cdot7^d\}}$ there exists a family of events $(B^{(n)}_{i,j})_{n\in{\{1,\dots,p\}}}$ such that
\begin{equation}
\label{pr2:boundonPBi}
\tilde{\bbP}^{u-\eps_{k+1}}\big( \big(B^{(n)}_{i,j}\big)^c\big)\leq\exp(-C_Pf_P(L)) \quad \text{ for all }n\in{\{1,\dots,p\}},
\end{equation}
and for all sets $A\subseteq\{1,\dots,p\}$,
\begin{equation}
\label{pr2:usedecoupling}
\begin{split}
    &\tilde{\bbP}^{u-\eps_{k+1}}\bigg(\bigcap_{n\in{A}}B^{(n)}_{i,j},\sum_{n\in{A}}d^{\omega,(n)}_k(x+x_i^k)+d^{\omega,(n)}_k(x+y_j^k)\leq c_{11}pa_{k+1}\bigg)
    \\& \mspace{36mu} \leq\sup_{y\in{\bbZ^d}}\tilde{\bbP}^{u-\eps_{k}}\bigg(\sum_{n\in{A\cup(p+A)}}d^{\omega,(n)}_k(y)\leq c_{11}pa_{k+1}\bigg).
\end{split}
\end{equation}
Next, introducing the event
\begin{equation*}
    C_{i,j}\ldef \bigg\{\sum_{n=1}^p\indicator_{B_{i,j}^{(n)}}\geq p\Big(1-\frac{1}{(k+6)^\delta}\Big)\bigg\},
\end{equation*}
we shall first show that, for all $x\in{\bbZ^d}$,
\begin{equation}
\label{pr2:1sttoprove}
\begin{split}
    & (2d\cdot21^d) \, \tilde{\bbP}^{u-\eps_{k+1}}\Big(C_{i,j},\sum_{n=1}^pd^{\omega,(n)}_k(x+x_i^k)+d^{\omega,(n)}_k(x+y_j^k)\leq c_{11} pa_{k+1}\Big)
    \\& \mspace{36mu}
    \leq \frac12 \bigg(\frac{1}{(4d\cdot21^d)^2} \, \exp\Big(-c_{12}g_P^{(\delta)}(a_{k+1})\Big)\bigg)^p.
\end{split}
\end{equation}
Fix some set $A\subseteq\{1,\dots,p\}$ with $|A|\geq p\big(1-\frac{1}{(k+6)^\delta}\big).$ It follows from \eqref{pr2:usedecoupling} that
\begin{equation}
\label{pr2:stbound}
\begin{split}
     &\tilde{\bbP}^{u-\eps_{k+1}}\Big(\big\{n\leq p:B_{i,j}^{(n)}\text{ occurs}\big\} \!= \!A,\, \sum_{n=1}^p d^{\omega,(n)}_k(x+x_i^k)+d^{\omega,(n)}_k(x+y_j^k)\leq c_{11} pa_{k+1}\Big) 
     \\&\leq\tilde{\bbP}^{u-\eps_{k+1}}\Big(\bigcap_{n\in{A^c}}(B_{i,j}^{(n)})^c\cap\bigcap_{n\in{A}}B_{i,j}^{(n)},\sum_{n\in{A}} d^{\omega,(n)}_k(x+x_i^k)+d^{\omega,(n)}_k(x+y_j^k)\leq c_{11} pa_{k+1}\Big)
     \\
     &\! \leq \tilde{\bbP}^{u-\eps_{k+1}}\Big(\bigcap_{n\in{A^c}}(B_{i,j}^{(n)})^c\Big) \sup_{y\in{\bbZ^d}}\tilde{\bbP}^{u-\eps_{k}}\Big(\!\!\sum_{n\in{A\cup(p+A)}}d^{\omega,(n)}_k(y)\leq c_{11} pa_{k+1}\Big).
\end{split}
\end{equation}
By \eqref{pr2:defepskak} we have $2|A|a_k\geq pa_{k+1}$. Hence, by the induction hypothesis we obtain for all $y\in{\bbZ^d}$,
\begin{equation}
\label{pr2:ndbound}
\begin{split}
   & \tilde{\bbP}^{u-\eps_{k}}\bigg(\sum_{n\in{A\cup(p+A)}}d^{\omega,(n)}_k(y)\leq c_{11}pa_{k+1}\bigg)\leq \tilde{\bbP}^{u-\eps_{k}}\bigg(\sum_{n=1}^{2|A|}d^{\omega,(n)}_k(y)\leq 2c_{11}|A|a_k\bigg)
    \\&
    \leq\bigg(\frac{1}{(4d\cdot21^d)^2}\exp\!\Big(-c_{12} g_P^{(\delta)}(a_k)\Big)\bigg)^{\! 2|A|} 
   \! \leq \frac{1}{(4d\cdot21^d)^{4|A|}}\exp\bigg(\!-\frac{c_{12}p a_{k+1} g_P^{(\delta)}(a_k)}{a_k}\bigg).
\end{split}
\end{equation}
Since $|A|\geq 4p/5,$ note that by \eqref{pr2:defRL} and \eqref{pr2:boundonPBi}, after choosing $L_0=L_0(\eps)$ large enough independently of $k$ and $p$, we have
\begin{equation*}
    \frac{2d\cdot21^d}{(4d\cdot21^d)^{4|A|}}
\leq  \frac{1}{2(4d\cdot21^d)^{(2+1/5)p}}
\end{equation*}
and by \eqref{pr2:defRL} and \eqref{pr2:boundonPBi}, after choosing $L_0=L_0(\eps)$ large enough independently of $k$ and $p$,
\begin{align*}
\sum_{A\subset\{1,\dots,p\}}\tilde{\bbP}^{u-\eps_{k+1}}\Big(\bigcap_{n\in{A^c}}(B_{i,j}^{(n)})^c\Big)
&=\sum_{A\subset\{1,\dots,p\}}\frac{\tilde{\bbP}^{u-\eps_{k+1}}\Big(\big\{n\leq p:B_{i,j}^{(n)}\text{ occurs}\big\}
     \!= \! A\Big)}{\tilde{\bbP}^{u-\eps_{k+1}}\Big(\bigcap_{n\in{A}}B_{i,j}^{(n)}\Big)} 
     \\&\leq\frac{1}{\left(1-\exp\left(-C_Pf_P(L)\right)\right)^p}\leq (4d\cdot21^d)^{p/5}.
\end{align*}
Hence, the inequality \eqref{pr2:1sttoprove} follows from \eqref{intro:gputsmallerugpt}, \eqref{pr2:stbound} and \eqref{pr2:ndbound} by summing over all possible sets $A\subset\{1,\dots,p\}$ with $|A|\geq p\big(1-\frac{1}{(k+6)^\delta}\big).$ Finally, we need to show that the events $C_{i,j}$ happen with sufficiently high probability, which will follow from the following Chernoff bound
\begin{equation*}
\begin{split}
  &  \tilde{\bbP}^{u-\eps_{k+1}}\big((C_{i,j})^c\big) \\
  & \mspace{36mu} \leq \tilde{\bbE}^{u-\eps_{k+1}}\Big[\exp\Big(-C_P f_P(L)\indicator_{B_{i,j}^{(1)}}\Big)\Big]^p \, \exp\bigg(p C_P f_P(L)\Big(1-\frac{1}{(k+6)^{\delta}}\Big)\bigg)
    \\& \mspace{36mu} \leq2^p\exp\left(-\frac{pC_Pf_P(L)}{(k+6)^{\delta}}\right),
\end{split}
\end{equation*}
where we used \eqref{pr2:boundonPBi}.
Using \eqref{pr2:defRL} and \eqref{eq:defhp}  we deduce that there exist a constant $c_{12}>0$ and $L_0=L_0(\eps),$ independent of $k,$ such that
\begin{align} \label{pr2:boundonCij}
    \tilde{\bbP}^{u-\eps_{k+1}}\big((C_{i,j})^c\big)&\leq\frac1{(4d\cdot21^d)^{2p+1}}\exp\bigg(-\frac{pC_P}{2(k+6)^\delta} \, f_P\Big(\frac{\eps^{1/\chi_P}2^{k-1}L_0}{(k+6)^{\delta(1+1/\chi_P)}}\Big)\bigg) \nonumber
    \\&\leq\frac1{(4d\cdot21^d)^{2p+1}}\exp\Big(-c_{12}ph_P^{(\delta)}(a_{k+1})\Big),
\end{align}
where we used $2^{k+1}\geq a_{k+1}\geq c 2^k$ in the last inequality. By combining \eqref{intro:gpsmallthanhp}, \eqref{pr2:Tk+1vsTk}, \eqref{pr2:1sttoprove} and \eqref{pr2:boundonCij} we obtain \eqref{pr2:induction} for $k+1.$
\end{proof}

\begin{remark}
\label{rk:weakerP3}
As the attentive reader will have noticed, we never actually used the assumption $f_P(L)\geq \exp\big(a_P\log(L)^{\eps_P}\big)$ from condition {\bf P3}, but only the existence of a function $g_P^{(\delta)}$ satisfying \eqref{intro:gputsmallerugpt} and \eqref{intro:gpsmallthanhp}, and in \eqref{pr2:boundonCij} the fact that $f_P(L)$ was increasing to infinity faster than $\log(L)^{c}$ for some $c>1$ (upon choosing $\delta$ small enough).  We also never used \eqref{eq:P3'replacebyindependentdec} when the weight $\mathbf{t}$ was chosen decreasing in \eqref{eq:temonotone}, or \eqref{eq:P3'replacebyindependentinc} when it was chosen increasing. Moreover, we only had to consider events $A$ of the type $\{f_1+f_2\leq s\}$ for monotone $f_1,f_2:\Omega\rightarrow[0,\infty]$ and $s>0$. The conditions \eqref{intro:gputsmallerugpt} and \eqref{intro:gpsmallthanhp} are satisfied, for instance, when 
\begin{equation}
\label{eq:defgP2}
    g_P^{(\delta)}(L)=h_P^{(\delta)}(L_P)\exp\left(\int_{L_P}^L\frac{1}{t}\wedge\big(\log(h_P^{(\delta)})\big)'(t) \, \mathrm{d}t\right)
\end{equation}
with $h_P^{(\delta)}$  as in \eqref{eq:defhp}. Note that, if $f_P(L)\geq \log(L)^c$ for some $c>1,$ then $g_P^{(\delta)}(L)\rightarrow\infty$ as $L\rightarrow\infty$ when $\delta>1$ is small enough. Therefore, one can still obtain Theorem~\ref{The:main} under the following weaker assumption.
\begin{changemargin}{4mm}
{\bf P3''} {\it(Weaker decoupling inequality)}. 
After possibly extending the probability space underlying $\bbP^u$, assume that there exist constants $R_P,L_P\in(0,\infty)$, $\delta_P,\xi_P>1,$ $C_P>0$ and $\chi_P>0$ such that for all $R\geq R_P,$ $L\geq L_P,$ $x_1,x_2\in{\bbZ^d}$ with $d\big(\bar{Q}(x_1,L^{\xi_P}),\bar{Q}(x_2,L^{\xi_P})\big)\geq RL$, and any $u,\hat{u}\in{I}$ with $u\geq\hat{u}+R^{-\chi_P},$ the following holds.  There exists an event $B$ satisfying \eqref{eq:boundonprobaB^c} with $f_P:[0,\infty)\rightarrow[0,\infty)$ satisfying $f_P(L)\geq \log(L)^{\delta_P}$, such that for all increasing functions $f_i:\Omega\rightarrow[0,\infty]$ supported on $\bar{Q}(x_i,L^{\xi_P}),$ $i\in{\{1,2\}}$ and all $s>0$,  \eqref{eq:P3'replacebyindependentinc} holds for $A=\{f_1+f_2\leq s\}$.
\end{changemargin}
Then, one can replace condition {\bf P3} in  Theorem \ref{The:main} by the weaker condition {\bf P3''}, and then, if $\mathbf{t}$ is actually chosen decreasing in \eqref{eq:temonotone}, Theorem \ref{The:main} still holds for $g_P^{(\delta)}$ as in \eqref{eq:defgP2}. One could also assume that $\mathbf{t}$ is increasing in \eqref{eq:temonotone} when replacing the condition \eqref{eq:P3'replacebyindependentinc} by \eqref{eq:P3'replacebyindependentdec} in {\bf P3''}. We chose to focus on the condition {\bf P3} instead of the weaker condition {\bf P3''} in this article, since it allows for a simpler definition of $g_P^{(\delta)}$ (cf.\ \eqref{eq:defgp}), and it makes our condition {\bf P3} stronger than condition {\bf P3} in \cite{DRS14}, which will be useful in Section~\ref{sec:RCM}. We refer to Proposition~\ref{prop:upperratioweakmixing} for an example where the weaker condition {\bf P3''} is useful.
\end{remark}

\subsection{Proof of Theorem \ref{The:main} under condition {\bf P3'}}
\label{sec:proofunderP3'}
Throughout this subsection, we assume that $(\bbP^u)_{u\in I}$ satisfies conditions {\bf P2}, {\bf P3'} and \eqref{intro:BLnotconnectedtoB2L}.
Recall the renormalization scheme introduced in Subsection~ \ref{sec:renormalization}, which depends on  $\rho$, $K$, $\delta$ and $L_0,$  as well as on $(\eps_k)_{k\geq0}$ in \eqref{pr2:defepskak}, which in turn depends on  $\eps$.

 Under condition {\bf P3'}, one can decouple two sets at distance $L$ from each other only when their cardinality is smaller than $L^{\xi_P}$ for some $\xi_P>1.$ In Section~\ref{sec:proofunderP3} we used condition {\bf P3} to decouple the sets $D_{x_i^k}^k$ and $D_{y_j^k}^k$ and then concluded using \eqref{ren:TkvsTk-1}, but the cardinality of these sets is $L_k^d,$ which is too large to apply condition {\bf P3'} (unless $\xi_P>d$). In order to be able to use condition {\bf P3'}, instead of comparing the scales $k+1$ and $k$ as in \eqref{ren:TkvsTk-1}, we will directly compare the scale $k$ with the scale $0,$ see \eqref{eq:TkvsTildek}. Lemma~\ref{pr3:lemmaembedding} below highlights this procedure via the notion of proper embeddings, and is tailored so that at scale $k+1$ one typically only has to decouple two ``tubes'' of length $L_k$ and width $L_0,$ see \eqref{eq:defS1S2}, which essentially follow the path $\pi$ minimizing $d_{k+1}^{\omega}.$ The cardinality of these tubes is $L_kL_0^{d-1},$ which is much smaller than $L_k^{\xi_P}$ for large $k,$ and thus {\bf P3'} can be applied, see \eqref{pr3:usedecoupling}. One can tailor the parameters $\rho$ and $K$ from \eqref{eq:defrhok}  so that {\bf P3'} can also be used for small $k,$ see \eqref{eq:defRL} and \eqref{eq:conditionsaresatisfied} for details. This strategy yields the following result.

\begin{lemma}
\label{pr3:mainLemma}
For all $\delta>1$ and $u\in{I},$ there exist positive $\eps,$ $L_0=L_0(\eps,\rho),$ $\zeta=\zeta(\eps,L_0),$ $\rho=\rho(\eps,L_0)$ and $K=K(\eps,L_0),$ all depending on $\delta$ and $u$, such that $u-\eps_0\in{I}$ and for all $k\in\bbN_0$,
\begin{equation}
    \label{pr3:induction}
    \bbE^{u-\eps_k}\big[\exp(-\zeta d^{\omega}_k(x))\big]\leq c_{\rho,d}^{2^K-1}\exp\big(\! -2^k\big) \quad \text{ for all }x\in{\bbZ^d},
\end{equation}
where $c_{\rho,d}=2d(2+\rho)^d(6+\rho)^{d-1}.$
\end{lemma}

Note that the bound $\exp(-2^k)$ in \eqref{pr3:induction} is the same as the bound one would obtain for independent weights, which is due to the fact that we require superexponential decay in the decoupling inequalities \eqref{eq:P3''replacebyindependentinc} and \eqref{eq:P3''replacebyindependentdec} (contrary to condition {\bf P3}). This is actually essential to deduce Theorem~\ref{The:main} under condition {\bf P3'} from Lemma~\ref{pr3:mainLemma} as shown in the following proof.

\begin{proof}[Proof of Theorem \ref{The:main} under condition {\bf P3'}]
Fix some $\delta>1,$ $u\in{I}$ and $L_0,$ $\zeta,$ $\eps,$ $\rho$ and $K$ as in Lemma \ref{pr3:mainLemma}. Using a Chernoff bound, condition {\bf P2} and \eqref{pr3:induction}, we get
\begin{equation}
\label{pr3:Chernov}
\begin{split}
    \bbP^u\big(d^{\omega}_k(0)\leq \zeta^{-1}2^{k-1}\big)&\leq\exp\big(2^{k-1}\big) \, \bbE^{u-\eps_k}\big[\exp(-\zeta d^{\omega}_k(0))\big]
    \leq c_{\rho,d}^{2^K-1}\exp\big(\! -2^{k-1}\big).
    \end{split}
\end{equation}
The rest of the proof is similar to the proof of Theorem~\ref{The:main} under condition {\bf P3} in Subsection~\ref{sec:proofunderP3}.
\end{proof}

Let us now turn to the proof of Lemma~\ref{pr3:mainLemma}. As explained above, we will need to refine the renormalization scheme in Subsection~\ref{sec:renormalization} to take into account all the levels below level $k$ at once. We define the binary tree $T_k$ of depth $k$ by $T_k=\bigcup_{n=0}^kT^{(n)},$ where $T^{(0)}=\{\varnothing\}$ and $T^{(n)}=\{0,1\}^n$ for all $n\geq1.$ A map $\tau:T_k\rightarrow\bbZ^d$ will be called a proper embedding with base $x\in{\bbZ^d}$ if $\tau(\varnothing)=x$ and for all $n\in\{0,\dots,k-1\}$ and $\sigma\in{T^{(n)}}$ there exist $i\in{\{1,\dots,(2+\rho_{k-n-1})^d\}}$ and $j\in{\{1,\dots,2d(6+\rho_{k-n-1})^{d-1}\}}$ such that
\begin{equation*}
    \tau(\sigma1)=\tau(\sigma)+x_i^{k-n-1}\text{ and }\tau(\sigma0)=\tau(\sigma)+y_j^{k-n-1}
\end{equation*}
and 
\begin{equation}
\label{eq:defproper}
    D_{\tau(\sigma)}^{k-n}\cap D_{\tau(\sigma\sigma')}^0\neq\varnothing\text{ for all }n\in{\{0,\dots,k-1\}},\sigma\in{T^{(n)}}\text{ and }\sigma'\in{T^{(k-n)}},
\end{equation}
where $\sigma\sigma'\in{T^{(k)}}$ denotes the concatenation of $\sigma$ and $\sigma'.$ The following lemma shows that connections at level $k$ imply connections along the leafs of a proper embedding.

\begin{lemma}
\label{pr3:lemmaembedding}
For all $k\in\bbN_0,$ $x\in{\bbZ^d}$ and paths $\pi$ of edges between $C_x^k$ and $(D_x^k)^c,$ there exists a proper embedding $\tau:T_k\rightarrow\bbZ^d$ with base $x$ such that $\pi$ connects $C_{\tau(\sigma)}^0$ to $(D_{\tau(\sigma)}^0)^c$ for all $\sigma\in{T^{(k)}}.$
\end{lemma}
\begin{proof}
We proceed recursively on $k.$ The claim is trivial for $k=0,$ and let us now assume that it holds for any $k\in\bbN_0.$ Let $\pi$ be a path of edges between $C_x^{k+1}$ and $(D_x^{k+1})^c,$ stopped on hitting $(D_x^{k+1})^c,$ then by the discussion above \eqref{ren:Dxiyj} $\pi$ connects $C^k_{x+x_{i}^k}$ to $(D^k_{x+x_{i}^k})^c$ for some $i\in{\{1,\dots,(2+\rho_k)^d\}}$ and so there exists a proper embedding $\tau_1:T_k\rightarrow\bbZ^d$ with base $x+x_{i}^k$ such that $\pi$ connects $C^0_{\tau_1(\sigma)}$ to $(D^0_{\tau_1(\sigma)})^c$ for all $\sigma\in{T^{(k)}}.$ Similarly, there exist $j\in{\{1,\dots,2d(6+\rho_k)^{d-1}\}}$ and a proper embedding $\tau_0:T_k\rightarrow\bbZ^d$ with base $x+y_{j}^k$ such that $\pi$ connects $C^0_{\tau_0(\sigma)}$ to $(D^0_{\tau_0(\sigma)})^c$ for all $\sigma\in{T^{(k)}}.$ We then define $\tau:T_{k+1}\rightarrow\bbZ^d$ by $\tau(\varnothing)=x,$ $\tau(1\sigma)=\tau_1(\sigma)$ and $\tau(0\sigma)=\tau_0(\sigma)$ for all $n\in{\{0,\dots,k\}}$ and $\sigma\in{T^{(n)}}.$

It is clear that $\pi$ connects $C_{\tau(\sigma)}^0$ to $(D_{\tau(\sigma)}^0)^c$ for all $\sigma\in{T^{(k+1)}},$ and we only need to show that $\tau$ is a proper embedding, that is it satisfies \eqref{eq:defproper}. For all $i\in{\{0,1\}},$ $n\in\{0,\dots,k-1\},$ $\sigma\in{T^{(n)}}$ and $\sigma'\in{T^{(k-n)}}$ we have
\begin{equation*}
    D_{\tau(i\sigma )}^{k+1-(n+1)}\cap D_{\tau(i\sigma\sigma')}^0=D_{\tau_i(\sigma)}^{k-n}\cap D_{\tau_i(\sigma\sigma')}^0\neq\varnothing,
\end{equation*}
since $\tau_i$ is a proper embedding. Moreover, for all $\sigma'\in{T^{(k+1)}}$ we have $D_{\tau(\varnothing)}^{k+1}\cap D_{\tau(\sigma')}^0\neq\varnothing$ since $\pi$ connects $C_{\tau(\sigma')}^0$ to $(D_{\tau(\sigma')}^0)^c$ and all the edges of $\pi$ contain a vertex in $D_x^{k+1}$. Hence,  $\tau$ satisfies \eqref{eq:defproper}.
\end{proof}

In view of Lemma \ref{pr3:lemmaembedding} we have
\begin{equation}
\label{eq:TkvsTildek}
    d^{\omega}_k(x)\geq \inf_{\tau:T_k\rightarrow\bbZ^d}\sum_{\sigma\in{T^{(k)}}}d^{\omega}_0(\tau(\sigma)), \qquad \forall k\in\bbN_0, \, x\in{\bbZ^d},
\end{equation}
where the infimum is taken over all proper embeddings $\tau$ with base $x$.

\begin{proof}[Proof of Lemma \ref{pr3:mainLemma}]
Fix $\delta>1$ and $u\in{I},$ on which all the constants in the rest of this proof depend. We shall prove by induction that for all $k\in\bbN_0,$ $x\in{\bbZ^d}$ and every proper embedding $\tau:T_k\rightarrow\bbZ^d$ with base $x$, 
\begin{equation}
\label{pr3:recursiontree}
    \bbE^{u-\eps_{k}}\bigg[\prod_{\sigma\in{T^{(k)}}}\exp\big(-\zeta d^{\omega}_0(\tau(\sigma))\big)\bigg]
    \leq
    \frac{1}{2c_{1,d}^{2^{k}}}\exp\big(\! -2^{k}\big).
\end{equation}
Once \eqref{pr3:recursiontree} is proved, \eqref{pr3:induction} will follow from \eqref{eq:TkvsTildek} and the fact that by \eqref{eq:defrhok} there are at most $c_{\rho,d}^{2^K-1}c_{1,d}^{2^k-2^K}$ proper embeddings $\tau:T_k\rightarrow\bbZ^d$ with base $x.$ For $k=0,$ assuming that $\eps$ is small enough so that $u-\eps_0\in{I},$ we have for all $x\in{\bbZ^d}$,
\begin{align*}
     \bbE^{u-\eps_0}&\big[\exp(-\zeta d^{\omega}_0(x))\big]
    \tend{\zeta}{\infty}\bbP^{u-\eps_0}\big(d^{\omega}_0(x)=0\big)
    \\&\leq\sup_{y\in{\bbZ^d}}\bbP^{u-\eps_0}\Big(Q(y,\lceil L_0/2\rceil)\longleftrightarrow Q(y,2\lceil L_0/2\rceil)^c\text{ in }\big\{e\in{E_d}:\,t^\om_e=0\big\}\Big).
\end{align*}
By \eqref{intro:BLnotconnectedtoB2L}, we can thus fix $L_0=L_0(\eps)$ and $\zeta=\zeta(\eps,L_0)$ so that \eqref{pr3:recursiontree} holds for $k=0$ for all $x\in{\bbZ^d}$ and every proper embedding $\tau:T_0\rightarrow\bbZ^d$ with base $x.$  Let us now assume that \eqref{pr3:recursiontree} holds for all $x\in{\bbZ^d}$ and every proper embedding $\tau:T_k\rightarrow\bbZ^d$ with base $x$ for any $k\in\bbN_0$.  Fix any $x\in{\bbZ^d}$ and any proper embedding $\tau:T_{k+1}\rightarrow\bbZ^d$ with base $x$.  For each $i\in{\{0,1\}}$, let $\tau_i(\sigma)=\tau(i\sigma)$ for all $\sigma\in{T_k}$. Then one can easily check that $\tau_i$ is a proper embedding with base $\tau(i)$. Set
\begin{equation}
\label{eq:defS1S2}
    S_1\ldef \bigcup_{\sigma\in{T^{(k)}}}\overline{D}_{\tau_0(\sigma)}^0\quad\text{ and }\quad S_2 \ldef \bigcup_{\sigma\in{T^{(k)}}}\overline{D}_{\tau_1(\sigma)}^0,
\end{equation}
where $\overline{D}_x^0$ denotes the union of $\mathcal{N}+y$ for $y\in{D_x^0}$ or $y$ in the outer boundary of $D_x^0$ (i.e.\ all vertices in $(D_x^0)^c$ that are nearest neighbors of $D_x^0$), and of all edges with at least one endpoint in $D_x^0.$

By \eqref{eq:defproper}, applied at level $k+1$ with $n=1,$ there exist $i\in{\{1,\dots,(2+\rho_k)^d\}}$ and $j\in{\{1,\dots,2d(6+\rho_k)^{d-1}\}}$ such that $D_{\tau_1(\sigma)}^0\cap D_{x+x_i^k}^k\neq\varnothing$ and $D_{\tau_0(\sigma)}^0\cap D_{x+y_j^k}^k\neq\varnothing$ for all $\sigma\in{T^{(k)}}$. Thus, by \eqref{eq:defS1S2} and \eqref{ren:Dxiyj}, upon choosing $L_0$ large enough,
\begin{align}
\label{eq:sizeSi}
 &   S_1\subset \bar{Q}(x+x_{i}^k,6L_k),\quad  S_2\subset \bar{Q}(x+y_{j}^k,6L_k),\quad  |S_i|\leq 2d(4L_0)^{d}2^k, \,i\in{\{1,2\}}, \nonumber \\
 & \mspace{36mu}  d(S_1,S_2)\geq \frac{L_k\rho_k}{(k+6)^\delta}-8L_0\geq \frac{L_k\rho_k}{2(k+6)^\delta},
\end{align}
where the last inequality holds by \eqref{eq:defrhok} and \eqref{ren:boundonL_k} after choosing $\rho\geq16$ and $K\geq8$ large enough. Take now
\begin{equation}
\label{eq:defRL}
    R=\bigg(\frac{(k+6)^{\delta}}{\eps}\bigg)^{\! \frac1{\chi_P}}\text{ and }L=\frac{\rho_k\eps^{1/\chi_P}L_k}{2(k+6)^{\delta(1+1/\chi_P)}}.
\end{equation}
After choosing $\eps$ small enough, and $L_0=L_0(\eps),$ $K=K(\eps,L_0)\geq8$ and $\rho=\rho(\eps,L_0)\geq16,$ large enough, independently of $k$, one can easily check that by \eqref{eq:defrhok} and \eqref{ren:boundonL_k}
\begin{equation}
\label{eq:conditionsaresatisfied}
    u-\eps_{k+1}\geq u-\eps_k+R^{-\chi_P},
    \quad
     R\geq R_P, \quad
     L\geq L_P, \quad L^{\xi_P}\geq6L_k\vee2d(4L_0)^d2^k.
\end{equation}
Noting that $d_0^{\omega}(\tau_i(\sigma))$ is supported on $S_{i+1}$, $i\in\{0,1\}$, in view of \eqref{eq:sizeSi}, \eqref{eq:defRL} and \eqref{eq:conditionsaresatisfied} we can now use {\bf P3'} and \eqref{pr3:recursiontree} to obtain
\begin{align}
\label{pr3:usedecoupling}
&    \bbE^{u-\eps_{k+1}}\bigg[\prod_{\sigma\in{T^{(k+1)}}}\exp\big(\! -\zeta d^{\omega}_0(\tau(\sigma))\big)\bigg]
    =
    \bbE^{u-\eps_{k+1}}\bigg[\prod_{i\in{\{0,1\}}}\prod_{\sigma\in{T^{(k)}}}
    \exp\big(\!-\zeta d^{\omega}_0(\tau_i(\sigma))\big)\bigg]  \nonumber \\
    &\mspace{36mu}
    \leq\Big(\frac{1}{2c_{1,d}^{2^{k}}}\exp\big(-2^{k}\big)\Big)^{\!2}+\exp\big(-C_PL^{\xi_P}\big).
\end{align}
Finally, in view of \eqref{eq:defrhok}, \eqref{ren:boundonL_k} and \eqref{eq:defRL}, one can choose $L_0=L_0(\eps)$ large enough, independently of $k\in\bbN_0,$ $\rho\geq1$ and $K\in\bbN,$ so that $C_PL^{\xi_P}\geq 2^{k+1}(1+\log(c_{1,d}))+\log(4)$. Then \eqref{pr3:recursiontree} for $k+1$  follows from \eqref{pr3:usedecoupling}.
\end{proof}


\section{Applications to the random conductance model}
\label{sec:RCM}

In this section we present two applications of the previous results to the random conductance model (RCM), which has been the subject of extensive research for more than a decade, see the surveys \cite{Bi11, Ku14} and references therein. Recall the setup introduced in the beginning of Section \ref{sec:FPP} (and in particular that $\mathcal{N}\subset E_d\cup\bbZ^d$ is a fixed set around the origin), and let $(\Omega, \cF)$ still denote the measurable space introduced there.   On $(\bbZ^d, E_d)$,  $d\geq 2$, we consider a family of random non-negative conductances constructed as follows. Let $\mathbf{a}:\Omega_{\mathcal{N}}\times\Omega_{\mathcal{N}}\times\Omega_E\rightarrow[0,\infty)$ be a measurable function, symmetric in the first two coordinates, such that 
\begin{equation}
\label{eq:amonotone}
    (\omega_{\mathcal{N}}^1,\omega_{\mathcal{N}}^2,\omega_E)\mapsto \indicator_{\{\mathbf{a}(\omega_{\mathcal{N}}^1,\omega_{\mathcal{N}}^2,\omega_E)>0\}}\text{ is monotone},
\end{equation} 
and let
\begin{align*}
  a^\om(x,y):= {\bf a}(\omega_{|\mathcal{N}+x},\omega_{|\mathcal{N}+y},\omega_{e}),\
   \forall\, e=\{x,y\} \in E_d,\quad
  a^\om(x,y) \ldef 0,\  \forall \{x,y\} \not\in E_d,
\end{align*}
  and define two measures on $\mathbb{Z}^d$,
\begin{align*}
\mu^\omega(x):=\sum_{y\sim x}a^\omega(x,y),\qquad
\nu^\omega(x):=\sum_{y\sim x}\frac{1}{a^\omega(x,y)} \indicator_{\{a^\om(x,y)>0\}}.
\end{align*}
Here and below we use the convention that $0/0 = 0$. We call an edge $e \in E_d$ \emph{open} if $a^\om(e) > 0$ and denote by $\cO(\om)$ the set of open edges, where $a^\om(e)=a^\om(x,y)$ if $e=\{x,y\}.$ We write $x \sim y$ if $\{x,y\} \in \cO(\om)$. 
 We still denote by $\big\{\tau_z : z \in \bbZ^d\big\}$ the group of space shifts as defined in \eqref{eq:def:space_shift}.
Throughout this section, we will only work with configurations $(a^\om(e))_{e\in E_d}$ for which there exists a unique infinite cluster $\cC_{\infty}(\om)$ of open edges and the origin is contained in  $\cC_{\infty}(\om)$.
We denote by $\rho^{\om}$ the graph distance on $(\cC_{\infty}(\om), \cO(\om))$, also known as chemical distance, i.e.\ for any $x, y \in \cC_{\infty}(\om)$, $\rho^{\om}(x,y)$ is the minimal length of a path between $x$ and $y$ that consists only of edges in $\cO(\om)$. Notice that the chemical distance is at least as large as the graph distance on $\bbZ^d$. For $x \in \cC_{\infty}(\om)$ and $r \geq 0$, let $B^{\om}(x,r) \ldef {\{ y \in \cC_{\infty}(\om) : \rho^{\om}(x, y) \leq \lfloor r \rfloor\}}$ be the closed ball with center $x$ and radius $r$ with respect to $\rho^{\om}$, and we write $B_{\mathrm{E}}(x,r)$ for the closed ball with respect to the  $\|\cdot\|_1$-norm which coincides with the graph distance on $\bbZ^d$. For a weight function $\varphi\!: \bbZ^d \to [0, \infty)$, $p \in [1, \infty)$ and any non-empty, finite $A \subset \bbZ^d$, we define space-averaged weighted $\ell^{p}$-norms on functions $f\!: A \to \bbR$ by 
\begin{align*}
  \Norm{f}{p, A, \varphi}
  \;\ldef\;
  \bigg(
    \frac{1}{|A|}\; \sum_{x \in A}\, |f(x)|^p\, \varphi(x)
  \bigg)^{\!\!1/p}
  \qquad \text{and} \qquad
  \Norm{f}{\infty, A} \;\ldef\; \max_{x \in A} |f(x)|,
\end{align*}
where $|A|$ denotes the cardinality of the set $A$. 
If $\varphi \equiv 1$, we simply write $\|f\|_{p, A} \ldef \|f\|_{p, A, \varphi}$.

Large parts of Section \ref{sec:RCMhk} below will be based on the results in \cite{ADS19}. Therefore, the underlying graph $(\cC_{\infty}(\om), \cO(\om))$ needs to satisfy \cite[Assumption~2.1]{ADS19} which in the present setting reads as follows. 
\begin{assumption} \label{ass:graph}
 There exist  $ C_{\mathrm{reg}}, C_{\mathrm{S}_1}\in [1,\infty)$ such that for all $x \in \cC_{\infty}(\om)$ the following hold.
  \begin{itemize}
  \item[(i)] Volume regularity of order $d$ for large balls. There exists $N_4(\om, x) < \infty$ such that for all $n \geq N_4(\om,x)$,
    \begin{align}\label{eq:reg:large_balls}
      C_{\mathrm{reg}}^{-1}\, n^d \;\leq\; |B^\om(x, n)| \;\leq\; C_{\mathrm{reg}}\, n^d.
    \end{align}

  \item[(ii)] Sobolev inequality.   There exist $d'\geq d$ and  $N_5(\om, x) < \infty$ such that for all $n \geq N_5(\om,x)$,
    \begin{align} \label{eq:sobolev}
      \Norm{u}{d'/(d'-1),B^\om(x,n)}
      \;\leq\;
      C_{\mathrm{S_1}}\, \frac{n}{|B^\om(x, n)|}\, \sum_{\substack{x,y \in  \cC_{\infty}(\om) \\ \{x,y\}\in  \cO(\om)}} |u(x)-u(y)|,
    \end{align}
    for every function $u\!: \cC_{\infty}(\om) \to \bbR$ with $\supp u \subset B^\om(x, n)$.  
  \end{itemize}  
\end{assumption}
\begin{remark}
The Euclidean lattice $(\bbZ^d, E_d)$ satisfies Assumption~\ref{ass:graph} with $d=d'$ and $N_4=N_5=1$. The Sobolev inequality in (ii) clearly follows from the classical Sobolev inequality $(S^1_d)$ which in turn is equivalent to the classical isoperimetric inequality. The weaker form in (ii) can be deduced on general graphs from a weaker isoperimetric inequality on large scales in conjunction with the volume regularity in (i), see \cite[Proposition~3.5]{DNS18} and the proof of Theorem~\ref{thm:hkeGauss} below for more details. Those conditions hold on a class of random graphs including correlated percolation clusters satisfying  {\bf P1--P3} and conditions {\bf S1--S2} introduced below.
\end{remark}

\subsection{Gaussian upper heat kernel bounds for RCMs with general speed measure} \label{sec:RCMhk}
We now introduce a possibly random speed measure $\theta^\om \!: \bbZ^d \to (0, \infty)$. More precisely we let $\boldsymbol{\theta}:\Omega_{\mathcal{N}}\rightarrow(0,\infty)$ be a measurable function  and $\theta^\om(x):=\boldsymbol{\th}(\omega_{|\mathcal{N}+x})$ for each $x\in \bbZ^d.$
We consider a continuous time Markov chain $X=(X_t)_{t\geq 0}$ on $\cC_\infty(\om)$ with generator $\cL_{\theta}^{\om}$ acting on bounded functions $f\!:  \cC_\infty(\om) \to \bbR$ as
\begin{align} \label{def:L}
  \big(\cL_{\th}^{\om} f)(x)
  \;=\; 
  \frac 1 {\th^\om(x)} \, \sum_{y \sim x} a^\om(x,y) \, \big(f(y) - f(x)\big).
\end{align}
We call this Markov chain the \textit{random conductance model (RCM)} with \textit{speed measure} $\theta^\omega$. As a key feature, the random walk is \emph{reversible} with respect to the speed measure $\th^\om$, and regardless of the particular choice of $\th^\om$ the jump probabilities of $X$ are given by $p^{\om}(x,y) \ldef a^\om(x,y) / \mu^{\om}(x)$, $x,y \in \cC_\infty$, and the various random walks corresponding to different speed measures are time-changes of each other.  If the random walk $X$ is currently at $x$, it will next move to $y$ with probability $a^\omega(x,y)/\mu^\omega(x)$, after waiting an exponential time with mean $\th^\omega(x)/\mu^\omega(x)$ at the vertex $x$. Perhaps the most natural choice for the speed measure is $\th^\omega\equiv\mu^\omega$ (which can be obtained via a function $\boldsymbol{\theta}$ on $\Omega_{\mathcal{N}}$ as above \eqref{def:L} upon increasing $\mathcal{N}$), for which we obtain the constant speed random walk (CSRW) that spends i.i.d.\ $\text{Exp}(1)$-distributed waiting times at all vertices it visits. Another well-studied process, the variable speed random walk (VSRW), is obtained by setting $\th^\omega\equiv 1$, so called because, in contrast to the CSRW, the waiting time at a vertex $x$ does depend on the location; it is an $\text{Exp}(\mu^\omega(x))$-distributed random variable.

For any choice of $\th^\om$, we denote by $P_x^\omega$ the law of the process $X$ started at $x\in\cC_\infty(\om)$.
Let $p_\th^{\om}(t,x,y)$ be the transition densities of $X$ with respect to the reversible measure (or the \emph{heat kernel} associated with $\cL^{\om}_\th$), i.e.\
\begin{align*}
p_\th^\omega(t,x,y) \; := \; \frac{P_x^\omega\left(X_t=y\right)}{\th^\om(y)}, \qquad t\geq 0, \, x,y\in\cC_\infty(\om).
\end{align*}

In this subsection our focus will be on Gaussian heat kernel estimates, see  e.g.\ \cite{De99, Ba04, BD10, Fo11, ADS16a, ADS19, Sa17, AH21} and references therein for previous results. We recall that, due to  a trapping phenomenon,  Gaussian bounds do not hold in general: for example, under i.i.d.\ conductances with fat tails at zero, the heat kernel decay may be sub-diffusive, see e.g.\ \cite{BBHK08, BKM15}. Recently,  local limit theorems for the heat kernel of RCMs on random graphs or with a general speed measures have been obtained in \cite{ACS21, AT21}, and a quantitative local limit theorem with an optimal rate of convergence for random  walks on supercritical i.i.d.\ percolation clusters has been shown in \cite{DG21}.

Here our aim is to use our results in Section~\ref{sec:FPP} in order to improve the already established heat kernel upper bounds in \cite[Theorem~3.2]{ADS19} which we recall next.  First, we note that the decay of $p^\om_\th(t,x,y)$ is naturally  governed by a distance function  $d_{\th}^{\om}$ of FPP-type defined by
\begin{align} \label{eq:def_intMetric}
  d_{\th}^{\om}(x,y)
  \;\ldef\;
  \inf_{\gamma\in \Gamma_{xy}^{\om}}
  \Bigg\{
    \sum_{i=0}^{l_{\gamma}-1} 
    \bigg( 
      1 \wedge \frac{\th^\om(z_i) \wedge \th^\om(z_{i+1})}{a^{\om}(z_i,z_{i+1})}
    \bigg)^{\!\!1/2}
  \Bigg\},
  \qquad x,y \in \cC_\infty(\om),
\end{align}
where $\Ga_{xy}^{\om}$ is the set of all nearest-neighbor paths $\ga = (z_0, \ldots, z_{l_\ga})$ in $\cC_\infty(\om)$  connecting $x$ and $y$ (cf.\ e.g.\ \cite{Da93, BD10, Fo11, ADS19}).  Note that $d_{\th}^{\om}$ is a metric which is adapted to the transition rates and the speed measure of the random walk. Further, for the CSRW, i.e.\ $\th^{\om} = \mu^{\om}$, the metric $d_{\th}^{\om}$ coincides with the usual graph distance $\rho^\om$.  In general, $d_{\th}^{\om}$ can be identified with the intrinsic metric generated by the Dirichlet form associated with $\cL_{\th}^{\om}$ and $X$, see e.g.\  \cite[Proposition~2.3]{ADS19}.  Further, notice that $d_\theta^\om(x,y) \leq \rho^\om(x,y)$ for all $x,y \in \cC_\infty$. In fact, the distance $d_\theta^\om$ can become much smaller than the graph distance, see \cite[Lemma~1.12]{ADS16a} for an example on $\bbZ^d$.

\begin{theorem}[\cite{ADS19}] \label{thm:hke}
  Suppose Assumption~\ref{ass:graph} holds and there exist $p,q,r\in(1,\infty]$ satisfying
\begin{align} \label{eq:cond_pqr'}
\frac{1}{r}+\frac{1}{p}\frac{r-1}{r}+\frac{1}{q}<\frac{2}{d'},
\end{align}
  such that, for every $x\in \cC_\infty(\om)$, there exists $N_6(\om,x)$ such that
\begin{align} \label{eq:cond_int}
\sup_{n\geq N_6(\om,x)} \! \max\Big\{
  \Big\|  \frac{\mu^{\om}}{\th^\om}\Big\|_{p, B^\om(x,n), \th},
    \Norm{ \nu^{\om}}{q, B^\om(x, n)},
    \Norm{ \th^\om}{r, B^\om(x,n)},
    \Big\|\frac1{\th^\om}\Big\|_{q, B^\om(x,n)} \Big\}
    \leq
C_{\mathrm{int}}
\end{align}
for some $C_{\mathrm{int}} <\infty$. Then, there exist $c_i = c_i(d, p, q, C_{\mathrm{int}})>0$ and  $\gamma = \gamma(d, p, q, C_{\mathrm{int}})>0$  such that for any given $t$ and $x\in \cC_\infty(\om)$ with $\sqrt{t} \geq \max\{ N_i(\om, x), \,  i=4,5,6\}$ and all $y \in \cC_\infty(\om)$ the following hold.
  \begin{enumerate}
  \item [(i)] If $d_{\th}^{\om}(x, y)\leq c_{13} t$ then
    \begin{align*}
      p^{\om}_{\th}(t, x, y)
      \;\leq\;
      c_{14}\, t^{-d/2}\,
      \bigg( 1 + \frac{\rho^\om(x,y)}{\sqrt{t}} \bigg)^{\!\!\ga}\, 
      \exp\!\bigg( \!-\! c_{15}\, \frac{d_{\th}^{\om}(x,y)^2}{t}\bigg).
    \end{align*}
  \item [(ii)] If $d_{\th}^{\om}(x,y) \geq c_{17} t$ then
    \begin{align*}
      p^{\om}_{\th}(t, x, y)
      \;\leq\;
      c_{14}\, t^{-d/2}\,  
      \bigg( 1 + \frac{\rho^\om(x,y)}{\sqrt{t}} \bigg)^{\!\!\ga}\,
      \exp\!\bigg( 
        \!-\! c_{16}\, d_{\th}^{\om}(x,y) 
        \bigg(1 \vee \log \frac{d_{\th}^{\om}(x,y)}{t}\bigg)
      \bigg).
    \end{align*}
  \end{enumerate}
\end{theorem} 

For later use, we also recall the following bound in $d\geq 3$ on the
 Green kernel $g^\omega(x,y)$ defined by
\begin{equation*}
g^\omega(x,y)\ldef \int_{0}^{\infty} p^\omega_\theta(t,x,y)\	 \textrm{d}t, \qquad x,y\in \cC_\infty(\om).
\end{equation*}

\begin{prop} \label{prop:green}
Let $d\geq 3$. Suppose that $(\cC_\infty(\om),\cO(\om))$ satisfies Assumption~\ref{ass:graph} and there exist $p,q\in (1,\infty]$ with $1/p+1/q<2/d'$ such that, for every $x\in \cC_\infty(\om)$,  there exists $N_7(\om,x)$ such that
\begin{align*} 
\sup_{n\geq N_7(\om,x)} \! \max\big\{
 \Norm{\mu^{\om}}{p, B^\om(x,n)},
    \Norm{ \nu^{\om}}{q, B^\om(x, n)}
  \big\}
    \leq
C_{\mathrm{int}}
\end{align*}
for some $C_{\mathrm{int}} <\infty$.  Then, for any $x\in \cC_\infty(\om)$, there exist $c_{18}\in (0,\infty)$ and $N_7(\om,x)$ such that for all $y\in \cC_\infty(\om)$ with $\rho^\om(x,y)\geq N_7(\om,x)$, 
\begin{align}  \label{eq:GU}
g^\omega(x,y)\leq c_{18} \,  \rho^\om(x,y)^{2-d}.
\end{align}
\end{prop}
\begin{proof}
Notice that the Green kernel does not depend on the speed measure $\th^\om$. Hence, it suffices to consider the special case of the CSRW, for which we have $d_\th^\om = \rho^\om$ by the definition \eqref{eq:def_intMetric} so that the bounds in Theorem~\ref{thm:hke} turn immediately into Gaussian estimates with respect to $ \rho^\om$. The result follows now, for instance,  by the same arguments as in the proof of \cite[Theorem~1.6-(i)]{AH21}.
\end{proof}

Note that one could replace $\rho^{\omega}(x,y)$ by $|x-y|$ in Proposition \ref{prop:green} since $\rho^{\omega}(x,y)\geq |x-y|,$ but the formulation with $\rho^{\omega}(x,y)$ is in general stronger.  We refer to \cite[Theorem 1.2]{ABDH13} for precise estimates
and asymptotics in the case of general non-negative i.i.d.\ conductances, and to \cite{ADS20} for recent results on the Green kernel in dimension $d=2$.

\medskip

 Henceforth, we consider random conductances that are distributed according to a family of probability measures $(\prob^u)_{u\in I}$ on $(\Omega, \cF)$ for some open interval $I$ similarly as in Section \ref{sec:FPP}. Note that our setup allows for any choice of distribution for $\theta^{\om}$ and $a^{\om},$ for instance by taking $\mathcal{N}=\{0\},$ $\Omega_E=[0,\infty),$ $\Omega_V=(0,\infty),$ $\mathbf{a}(\omega_V^1,\omega_V^2,\omega_E)=\omega_E,$ $\boldsymbol{\theta}(\omega_V)=\omega_V,$ and choosing $\bbP^u$ appropriately. Our seemingly more complicated setup is useful when the weights and speed measure depend on a common environment $\omega,$ for instance the Gaussian free field as in Corollary~\ref{cor:optidecay}. Similar remarks could be made about the setup with killing measure in Section \ref{sec:Agmon_result}.
 
  Next we recall the conditions {\bf S1} and {\bf S2} from \cite{DRS14, Sa17}. For $r\in  [0, \infty]$, we denote by $\cC_r(\om)$ the set of vertices  which are in connected components of $( \bbZ^d, \cO(\om))$ of $\ell^1$-diameter at least $r$. In particular, $\cC_\infty(\om)$ is the subset of vertices which are in infinite connected components.

\begin{changemargin}{4mm}
{\bf S1 }{\it(Local uniqueness)}. There exists a function $f_S:I\times \bbZ_+ \rightarrow \bbR$ such that for
each $u \in I$, there exist $\Delta_S=\Delta_S(u)>0$ and $R_S=R_S(u)<\infty$ such that $f_S(u,R)\geq \log(R)^{1+\Delta_S}$ for all $R\geq R_S$, and for all $u\in I$ and $R\geq 1$,
\begin{align*}
\prob^u\big[ \cC_R \cap [-R,R]^d \not= \varnothing \big] \geq 1- e^{-f_S(u,R)},
\end{align*}
and
\begin{align*}
&\prob^u\Big[ \forall x,y \in \cC_{R/10}\cap [-R,R]^d: \text{$x$ is connected to $y$ by an open path in $[-2R,2R]^d$}  \Big] \\
&\mspace{36mu} \geq 1- e^{-f_S(u,R)}.
\end{align*}
\end{changemargin}

\begin{changemargin}{4mm}
{\bf S2 }{\it(Continuity)}. The function $u\mapsto \prob^u[0\in \cC_\infty]$ is positive and continuous on $I$.
\end{changemargin}

Note that in \cite{DRS14, Sa17} the conditions {\bf S1}--{\bf S2}, as well as the conditions {\bf P1}-{\bf P3}, are stated for site percolation models rather than bond percolation as here. The proof of their results can however be adapted to our setting by simple notational changes. In particular,  if the family $(\prob^u)_{u \in  I}$ satisfies {\bf S1--S2}, then, for every $u\in I$, there exists $\prob^u$-a.s.\ a unique infinite cluster $\cC_\infty$, cf.\ \cite[Remark~1.9-(2)]{Sa17} and $\prob^u[0\in \cC_\infty]>0$.  Write $\prob^u_0[ \,\cdot\, ] \ldef \prob^u[ \cdot \,|\, 0 \in \cC_{\infty} ]$ and $\mean^u_0$ for the corresponding expectation.

 \begin{assumption}
\label{ass:moment}
(i) The mapping $\mathbf{t}: \Omega_{\mathcal{N}}\times\Omega_{\mathcal{N}}\times\Omega_E\rightarrow [0,\infty)\cup\{\infty\}:$
\begin{equation*}
    (\omega_{\mathcal{N}}^1,\omega_{\mathcal{N}}^2,\omega_E) \mapsto \begin{cases}\Big( 
      1 \wedge \frac{\boldsymbol{\th}(\omega_{\mathcal{N}}^1) \wedge \boldsymbol{\th}(\omega_{\mathcal{N}}^2)}{\mathbf{a}(\omega_{\mathcal{N}}^1,\omega_{\mathcal{N}}^2,\omega_E)}
    \Big)^{\frac12}&\text{if }\mathbf{a}(\omega_{\mathcal{N}}^1,\omega_{\mathcal{N}}^2,\omega_E)\neq0
    \\\infty&\text{otherwise.}
    \end{cases}
\end{equation*}
is monotone.

(ii) There exist $p,q,r\in(1,\infty]$ satisfying
\begin{align} \label{eq:cond_pqr}
\frac{1}{r}+\frac{1}{p}\frac{r-1}{r}+\frac{1}{q}<\frac{2}{d}
\end{align}
such that, for every $u\in I$,
\begin{align*}
\max\Bigg\{ \mathbb{E}^u_0 \!\left[\Big( \frac{\mu^\om(0)}{\theta^\om(0)}\Big)^{p} \, \theta^\om(0) \right], \, \mathbb{E}^u_0\left[\nu^\om(0)^{q}\right], \,
\mathbb{E}^u_0\left[\theta^\om(0)^{r}\right], \, \mathbb{E}^u_0\! \left[\theta^\om(0)^{-q} \right] \Bigg\} \; < \; \infty.
\end{align*}
\end{assumption}
\begin{remark} 
In the case of the CSRW or VSRW the moment condition in Assumption~\ref{ass:moment}-(ii) can be written as $\mathbb{E}^u_0[a^\om(e)^p]<\infty$ and $\mathbb{E}^u_0[a^\om(e)^{-q}]<\infty$ for $p,q\in (1,\infty]$ satisfying $1/p+1/q<2/d$. Indeed, for the CSRW, $\theta^\om\equiv\mu^\om$, choose  $p=\infty$ and relabel $r$ by $p$; for the VSRW, $\theta^\om\equiv1$, choose $r=\infty$. For the CSRW, this condition is known to be optimal for a local limit theorem and two-sided near-diagonal Gaussian estimates to hold, see \cite[Section~5]{ADS16}.
\end{remark}

Note that the distance $d_{\theta}^{\omega}$ from \eqref{eq:def_intMetric} corresponds to the first passage percolation metric considered in Section \ref{sec:FPP} for $\mathbf{t}$ as in Assumption~\ref{ass:moment}-(i). We will now exploit Theorem~\ref{The:main} in order to improve  the upper heat kernel bounds in Theorem~\ref{thm:hke} into genuine Gaussian  bounds. More precisely, under conditions {\bf P1--P3} and {\bf S1--S2} and Assumption~\ref{ass:moment}-(i)  (note that \eqref{intro:BLnotconnectedtoB2L} is trivially satisfied since the weight $\mathbf{t}$ from Assumption~\ref{ass:moment}-(i) is strictly positive) Theorem~\ref{The:main} and its analogue for the chemical distance established in \cite{DRS14} ensure  that both, the FPP-distance $d_\theta^\om$ and the chemical distance $\rho^\om$ are comparable to the Euclidean distance.  In particular, there exist constants $c_{19},c_{20}>0$ such that for $\prob_0$-a.e.\ $\om$ and $x\in{\cC_\infty(\om)},$ there exists $N_8(\om,x)$ such that for all $y\in \cC_\infty(\om)$ with $|x-y|\geq N_8(\om,x)$,
  \begin{align} \label{eq:comp_metrics}
 c_{19}  \, \rho^\om(x,y) \; \leq \; |x-y| \; \leq \; c_{20} \,  d_\th^\om(x,y).
 \end{align}

\begin{theorem} \label{thm:hkeGauss}
Suppose that the family of measures $(\prob^u)_{u\in{I}}$ satisfies assumptions {\bf P1--P3},  {\bf S1--S2} and suppose that Assumption~\ref{ass:moment} holds.  Then, for any $u\in I$ there exist positive constants $c_i=c_i(u)$  such that for $\prob^u$-a.e.\ all $\om$ and $x\in \cC_\infty(\om),$ there exists $N_9(\omega,x)$ such that for all $t \geq N_9(\omega,x)$ and $y \in \cC_\infty(\om)$ with $|x-y|\geq N_8(\om, x)$ the following hold.
  \begin{enumerate}
  \item [(i)] If $|x-y|\leq c_{21} t$ then
    \begin{align*}
      p^{\om}_{\th}(t, x, y)
      \;\leq\;
      c_{22}\, t^{-d/2}\,
      \exp\!\bigg( \!-\! c_{23}\, \frac{|x-y|^2}{t}\bigg).
    \end{align*}
  \item [(ii)] If $|x-y| \geq c_{25} t$ then
    \begin{align*}
      p^{\om}_{\th}(t, x, y)
      \;\leq\;
      c_{22}\, t^{-d/2}\, 
      \exp\!\bigg( 
        \!-\! c_{24}\, |x-y| \,
        \bigg(1 \vee \log \frac{|x-y|}{t}\bigg)
      \bigg).
    \end{align*}
  \end{enumerate}
\end{theorem}  
\begin{proof}
We shall first verify the assumptions of Theorem~\ref{thm:hke}. First recall that our condition {\bf P3} is stronger than condition {\bf P3} in \cite{DRS14, Sa17}. In particular, for any $\vartheta \in (0,1)$, \cite[Proposition~4.3]{Sa17}  guarantees the $\prob^u_0$-a.s.\ existence of a large $\vartheta$-very regular ball (see e.g.\ \cite[Assumption~1.3 and Example~1.12]{DNS18} for details), which immediately implies the volume regularity in \eqref{eq:reg:large_balls}.  Furthermore, this also implies under $\prob^u_0$ an isoperimetric inequality on large sets (see \cite[Lemma~2.10]{DNS18}), which in conjunction with the volume regularity implies the Sobolev inequality in \eqref{eq:sobolev} with  $d' \ldef (d-\vartheta)/(1-\vartheta)$, see \cite[Proposition~3.5]{DNS18}. Since $\vartheta \in (0,1)$ is arbitrary, $d'$ can be chosen arbitrarily close to $d$, so that condition \eqref{eq:cond_pqr'} becomes condition  \eqref{eq:cond_pqr}.

Next, using the volume regularity \eqref{eq:reg:large_balls} and the stationarity of $\theta^\om$, we get for sufficiently large $n$,
\begin{align*}
 \Big\|  \frac{\mu^{\om}}{\th^\om}\Big\|_{p, B^\om(x,n), \th}^p
&  \; \leq \; 
 \frac 1 {|B^\om(x,n)|} \sum_{y\in B_{\mathrm{E}}(x,n)} \Big( \frac{\mu^\om(y)}{\theta^\om(y)}\Big)^{p} \, \theta^\om(y) \, \indicator_{\{y\in \cC_\infty(\om)\}} \\
 &  \; \leq \; 
 c n^{-d}  \sum_{y\in B_{\mathrm{E}}(x,n)} \Big( \frac{\mu^{\tau_y \om}(x)}{\theta^{\tau_y \om}(x)}\Big)^{p} \, \theta^{\tau_y \om}(x) \, \indicator_{\{x \in \cC_\infty(\tau_y \om)\}}.
\end{align*}
Now, the spatial ergodic theorem and the shift-invariance of $\prob^u$ gives that, for $\prob^u$-a.e.\ $\om$ and all $x\in \bbZ^d$,  there exists $N(\om,x)$ such that for all $n\geq N(\om,x)$,
\begin{align*}
& n^{-d}  \sum_{y\in B_{\mathrm{E}}(x,n)} \Big( \frac{\mu^{\tau_y \om}(x)}{\theta^{\tau_y \om}(x)}\Big)^{p} \, \theta^{\tau_y \om}(x) \, \indicator_{\{x \in \cC_\infty(\tau_y \om)\}} 
 \leq \; 2 \,   \mathbb{E}^u \!\left[\Big( \frac{\mu^\om(x)}{\theta^\om(x)}\Big)^{p} \, \theta^\om(x) \indicator_{\{x \in \cC_\infty\}} \right] \\
 & 
 \mspace{36mu} \; = \; c \, \mathbb{E}^u_0 \!\left[\Big( \frac{\mu^\om(0)}{\theta^\om(0)}\Big)^{p} \, \theta^\om(0) \right] \; < \; \infty.
\end{align*}
Hence for $\prob^u$-a.e.\ $\om$ and each $x\in \cC_\infty(\om)$ we obtain $\Norm{\mu^{\om}/\th^\om}{p, B^\om(x,n), \th} \; \leq \; c$ for all $n\geq N(\om,x)$. 
Using the same argument we get similar estimates on $ \Norm{ \nu^{\om}}{q, B^\om(x, n)}$,    $ \Norm{ \th^\om}{r, B^\om(x,n)}$ and 
 $ \Norm{(\th^\om)^{-1}}{q, B^\om(x,n)}$. 
In particular, for $\prob^u_0$-a.e.\ $\om$ and each $x\in \cC_\infty(\om)$, there exists $N_6(\om,x)$ such that \eqref{eq:cond_int} holds. Thus, the assumptions of Theorem~\ref{thm:hke} are satisfied. 
Finally, choosing $N_9(\omega,x)$ large enough and noting that the constants $c_i$ in Theorem~\ref{thm:hke} do not depend on $\om,$ by using \eqref{eq:comp_metrics} and the inequality $d_{\theta}^{\om}(x,y)\leq\rho^{\om}(x,y),$  the estimates in Theorem~\ref{thm:hke} can be turned into the desired Gaussian upper bounds under $\prob^u_0$ since  the additional term $(1+ \rho^\om(x,y)/\sqrt{t})^{\ga}$ can  be absorbed by the exponential term into a constant. By ergodicity there exists $\prob^u$-a.s.\ $x\in{\mathcal{C}_{\infty}},$ and by translation invariance we conclude that the same Gaussian upper bounds also hold under $\prob^u.$
\end{proof}
 
In \cite{Sa17} two-sided Gaussian estimates have been shown for the simple random walk on correlated percolation clusters satisfying the original conditions  {\bf P1--P3},  {\bf S1--S2} from \cite{DRS14}. Theorem~\ref{thm:hkeGauss} extends the upper bound in \cite{Sa17} to a more general class of RCMs under slightly modified assumptions.

\subsection{Exponential Green kernel decay for RCMs with random killing rates} \label{sec:Agmon_result}
In this subsection we consider a VSRW with a random potential, i.e.\ a continuous time Markov chain $X = (X_t)_{t\geq 0}$, taking values in  $\cC_\infty(\om)\cup\{\partial\}$ where $\partial$ is an isolated point called the cemetery state,  with generator $\cL^{\om}$ acting on bounded functions $f\!: \cC_\infty(\om) \to \bbR$ as
\begin{align} \label{eq:defL}
\big(\cL^{\om} f)(x)
  \;=\; 
  \sum_{y \sim x} a^\om(x,y) \, \big(f(y) - f(x)\big) - h \, \kappa^\om(x) \, f(x).
\end{align}
Here $h\in [0,1]$ is a scalar and the potentials are given by a killing measure $\kappa^\om \!: \bbZ^d \to (0, \infty)$ describing  the random killing rates of the random walk. 
 More precisely we let $\boldsymbol{\kappa}:\Omega_{\mathcal{N}}\rightarrow(0,\infty)$ be a measurable function  and $\kappa^\om(x):=\boldsymbol{\kappa}(\omega_{|\mathcal{N}+x})$ for each $x\in{\bbZ^d}.$ When visiting a vertex $x$, the random walk jumps to a neighbor $y$ at rate $a^\om(x,y)$, and it is killed, i.e.\ it is sent to the cemetery state $\partial$, at rate $h k^\om(x)$. We denote by $\Prob_x^\om$ the law of the process starting at the vertex $x \in \bbZ^d$ and by $\Mean_x^\om$ the corresponding expectation. We denote by $\zeta\ldef \inf\{ t\geq 0: X_t=\partial\}$ the killing time of $X$.
For $x, y \in \cC_\infty(\om)$ and $t \geq 0$ let $p^{\om}(t,x,y)$ be the associated heat kernel given by
$ p^{\om}(t,x,y)  \;\ldef\;  \Prob_x^{\om}\big[X_t = y; \, t<\zeta\big]$. The Green's function of $X$ is defined by
\begin{align}
\label{eq:defgreen}
  g^{\om}(x,y) 
  \ldef
  \Mean_x^{\om}\!\Big[ 
    \int_0^{\zeta} {\indicator_{\{X_t = y\}}}\, \md t 
  \Big] 
  =
  \int_0^\infty 
   \Prob_x^{\om}\!\big[X_t = y;\, t < \zeta \big] \,  \md t.
\end{align}
Recall that, for every $y \in \cC_{\infty}(\om)$, the function $x \mapsto g^{\om}(x, y)$ is a fundamental solution of $\cL^\om u = -\indicator_{\{y\}}$.
We define the distance function  $d_{\kappa}^{\om}$ by
\begin{align} \label{eq:def_chemdist}
  d_{\kappa}^{\om}(x,y)
  \ldef
  \inf_{\gamma\in \Gamma_{xy}^{\om}}
  \Bigg\{
    \sum_{i=0}^{l_{\gamma}-1} 
    \bigg( 
      1 \wedge \frac{\kappa^\om(z_i) \wedge \kappa^\om(z_{i+1})}{a^{\om}(z_i,z_{i+1})}
    \bigg)^{\!\!1/2}
  \Bigg\}, \quad x,y\in \cC_\infty(\om),
\end{align}
where $\Ga_{xy}^{\om}$ is again the set of all nearest-neighbor paths $\ga = (z_0, \ldots, z_{l_\ga})$ in $\cC_\infty(\om)$  connecting $x$ and $y$. Notice that $d_\kappa^\om(x,y) \leq \rho^\om(x,y)$ for all $x,y \in \cC_\infty(\om)$. 

We will apply our results on the positivity of the time constant on the distance $d^\om_\kappa$ to obtain an exponential decay estimate on $g^\om(x,y)$. The main step will be to establish the following deterministic result, that we prove in Section \ref{sec:Agmon}.

\begin{theorem} \label{thm:Agmon}
Let $d\geq 3$ and $\omega\in{\Omega}.$ Suppose that $(\cC_\infty(\om),\cO(\om))$ satisfies Assumption~\ref{ass:graph} and suppose there exist $p,q\in (1,\infty]$ with $1/p+1/q<2/d'$ such that, for every $x\in \cC_\infty(\om)$,  there exists $N_{10}(\om,x)$ such that
\begin{align} \label{eq:int_condAgmon} 
\sup_{n\geq N_{10}(\om,x)} \! \max\Big\{
 \Norm{\mu^{\om}}{p, B^\om(x,n)},
    \Norm{ \nu^{\om}}{q, B^\om(x, n)},
    \Norm{ \kappa^\om}{p,B^\om(x,n)},
     \Norm{1/\kappa^\om}{1, B^\om(x,n)}
  \Big\}
    \leq
C_{\mathrm{int}}
\end{align}
for some $C_{\mathrm{int}} <\infty$.
Then, there exist $\lambda=\lambda(d) \in (0,1),$ $\gamma=\gamma(d',p,q)\in (1,\infty)$,  and  $c_{26}=c_{26}(d',p,q,C_{\mathrm{int}})\in (0,\infty)$ such that  the following holds. For any $x\in \cC_\infty(\om)$, there exist $N_{11}(\om,x)$ such that for all $y\in \cC_\infty(\om)$ with $ \rho^\om(x,y)  \geq N_{11}(\om,x)$ and all $h\in{[0,1]},$
\begin{align}  \label{eq:apiori1}
g^\om(x,y)  \; \leq \;  c_{26} \,
F_\gamma \big(h \, \rho^\om(x,y)^2 \big) \, \rho^\om(x,y)^{2-d} \!\max_{z\in{B^{\om}(x,n)^c}}\Big(e^{-\lambda \sqrt{h}  d_\kappa^\om(x,z)}\Big),
\end{align}
where $n=\rho^{\om}(x,y)/4$ and $F_\gamma(r)\ldef (1+r)^\ga (1+1/r)^{1/2}$.
\end{theorem}
First we note that Theorem~\ref{thm:Agmon_intro} is a direct consequence of Theorem~\ref{thm:Agmon}.

\begin{proof}[Proof of Theorem~\ref{thm:Agmon_intro}] Since $\mean[a^\om(e)^{-q}]<\infty$ we have $a^\om(e)>0$ almost surely for any $e\in E_d$.  Therefore $(\cC_\infty(\om),\cO(\om)=(\bbZ^d, E_d)$ and Assumption~\ref{ass:graph} is satisfied with $d=d'$ and $\rho^\om(x,y)=|x-y|$.
For $\prob$-a.e.\ $\om$, the integrability condition \eqref{eq:int_condAgmon} follows from \eqref{eq:Agmon_moment} by the ergodic theorem.
\end{proof}


Similarly as in Theorem \ref{thm:hkeGauss}, we need in  addition the following monotonicity and moment assumption to state the main result of this subsection.

 \begin{assumption}
\label{ass:moment_killing}
(i) The mapping $\mathbf{t}: \Omega_{\mathcal{N}}\times\Omega_{\mathcal{N}}\times\Omega_E\rightarrow [0,\infty)\cup\{\infty\}:$
\begin{equation*}
    (\omega_{\mathcal{N}}^1,\omega_{\mathcal{N}}^2,\omega_E) \mapsto \begin{cases}\Big( 
      1 \wedge \frac{\boldsymbol{\kappa}(\omega_{\mathcal{N}}^1) \wedge \boldsymbol{\kappa}(\omega_{\mathcal{N}}^2)}{\mathbf{a}(\omega_{\mathcal{N}}^1,\omega_{\mathcal{N}}^2,\omega_E)}
    \Big)^{\frac12}&\text{if }\mathbf{a}(\omega_{\mathcal{N}}^1,\omega_{\mathcal{N}}^2,\omega_E)\neq0
    \\\infty&\text{otherwise.}
    \end{cases}
\end{equation*}
is monotone.

(ii) There exist $p,q\in(1,\infty]$ satisfying $1/p+1/q<2/d$ such that, for every $u\in I$,
\begin{align*}
\max\Big\{ \mathbb{E}^u_0\left[ \mu^\om(0)^{p}  \right], \mathbb{E}^u_0\left[\nu^\om(0)^{q}\right], \mathbb{E}^u_0 \left[\ka^\om(0)\right],  \mathbb{E}^u_0\left[\ka^\om(0)^{-1}\right] \Big\} <\infty.
\end{align*}
\end{assumption}

\begin{theorem} \label{thm:GKdecay_killing}
Let $d\geq 3$ and suppose that the family of measures $(\prob^u)_{u \in I}$ satisfies assumptions {\bf P1--P3},  {\bf S1--S2} and Assumption~\ref{ass:moment_killing}.  Then, for any $u\in I$, there exist $c_{27}=c_{27}(u), c_{28}=c_{28}(u) \in (0,\infty)$ such that for $\prob^u$-a.e.\ all $\om$ and $x\in \cC_\infty(\om)$ there exists $N_{12}(\om,x)$ such that for all $y\in \cC_\infty(\om)$ with $|x-y|> N_{12}(\om,x)$ and all $h\in{[0,1]}$,
\begin{align}
\label{eq:boundongom}
g^\om(x,y)  \; \leq \; & c_{27} \, |x-y|^{2-d} \, \exp\big(-c_{28} \sqrt{h} |x-y|\big).
\end{align}
\end{theorem}
\begin{proof}
We aim to apply Theorem~\ref{thm:Agmon}. First, note that 
Assumption~\ref{ass:graph} follows from {\bf P1--P3},  {\bf S1--S2} under $\bbP_0^u$ as explained in the proof of Theorem~\ref{thm:hkeGauss} above, and, by the same argument used to derive \eqref{eq:cond_int} above, based on an application of the ergodic theorem, {\bf P1} and Assumption~\ref{ass:moment_killing} ensure that \eqref{eq:int_condAgmon} holds.
Moreover, similarly to the derivation of \eqref{eq:comp_metrics} above, by Theorem~\ref{The:main} and the main results in \cite{DRS14} we have that for $\prob^u_0$-a.e.\ $\om$ and any $x\in \cC_\infty(\om)$ there exist $c_i \in (0,\infty)$ and  $N_{13}(\om,x)$ such that for all  $y\in \cC_\infty(\om)$ with $|x-y| \geq N_{13}(\om,x)$,
\begin{align} \label{eq:comp_metrics2}
 d^\om_\ka(x,y)   \geq c_{29} \rho^{\om}(x,y), \qquad
 \rho^\om(x,y)  \leq c_{30} |x-y|.
\end{align}    
 We now apply Theorem~\ref{thm:Agmon}. Since $n=\rho(x,y)/4$, the last term in the right-hand side of \eqref{eq:apiori1} can be estimated as
\begin{align*}
\max_{z \in  B^\om(x,n)^c} \Big( e^{-\lambda \sqrt{h} d^\om_{\kappa}(x,z) } \Big)
\leq 
e^{-\lambda \sqrt{h} c_{29} n} 
\leq 
e^{-c \sqrt{h}\rho^\om(x,y)}.
\end{align*}
Recalling that $\rho^\om(x,y)\geq |x-y|$, we combine this with \eqref{eq:comp_metrics2} to obtain that $\bbP^u_0$-a.s, if $|x-y|\geq N_{12}(\om,x):= N_{11}(\om,x)\vee N_{13}(\om,x)$,
\begin{align*}
g^\om(x,y) \;  \; \leq \;  c \,  F_\gamma\big( h |x-y|^2\big) \, \big|x-y\big|^{2-d} \, e^{-c \sqrt{h} |x-y|}. 
\end{align*}
By choosing a smaller  constant in the exponential, the factor $F_\gamma\big(h |x-y|^2\big)$ can be absorbed into a constant if $\sqrt{h}|x-y|\geq 1.$ Moreover, if $\sqrt{h}|x-y|\leq1,$ \eqref{eq:boundongom} follows $\bbP_0^u$-a.s.\ from \eqref{eq:GU} similarly as before. Therefore \eqref{eq:boundongom} holds $\bbP_0^u$-a.s.\ for all $h\in{[0,1]}$ and $x,y\in{\mathcal{C}_{\infty}(\omega)}$ with $|x-y|\geq N_{12}(\om,x),$ and we can conclude by translation invariance. 
\end{proof}

\begin{remark} \label{rem:P3cond}
(i) In the proofs of Theorems~\ref{thm:hkeGauss} and \ref{thm:GKdecay_killing} above, assumptions {\bf P1--P3},  {\bf S1--S2} are only needed to ensure the validity of Assumption~\ref{ass:graph} and the comparability of  the Euclidean distance with the chemical distance $\rho^\om$ and the FPP-distances $d^\om_\th$ and $d^\om_\kappa$, respectively. In particular, in both theorems condition {\bf P3} may be replaced by the conjunction of its weaker version in \cite{DRS14, Sa17} and condition {\bf P3'}, or condition {\bf P3''}, see Remark \ref{rk:weakerP3}. 

(ii) Similarly as in Section~\ref{sec:RCMhk} we can  also introduce a speed measure $\theta^\om$ and consider a discrete Schr\"odinger operator of the form
\begin{align*}
\big(\cL^{\om}_{\theta,\kappa} f)(x)
  =\; 
  \frac{1}{\th^\om(x)} \bigg(\sum_{y \sim x} a^\om(x,y) \, \big(f(y) - f(x)\big) - h \, \kappa^\om(x) \, f(x)\bigg).
\end{align*}
The results immediately extend to this setting, and the associated Agmon-type FPP distance is still given by \eqref{eq:def_chemdist}. This is due to the fact that the Green's function of the operator does not depend on the speed measure $\theta^\om,$ hence it is sufficient to consider the case $\theta^{\om}=1$ in the context of Theorem~\ref{thm:GKdecay_killing}.
 
 (iii) It follows from Remark~\ref{rk:fppvertices}-(ii) that condition \eqref{eq:amonotone} and Assumption~\ref{ass:moment_killing}-(i) in Theorem~\ref{thm:GKdecay_killing} (or Assumption~\ref{ass:moment}-(i) in Theorem~\ref{thm:hkeGauss}) can be removed whenever condition {\bf P3} still holds without the monotonicity assumption on the event $A$.
\end{remark}

 \subsection{Proof of Theorem~\ref{thm:Agmon}} \label{sec:Agmon}
We will show Theorem~\ref{thm:Agmon} by purely analytic and deterministic arguments. Therefore, throughout the remainder of this section, we fix $\om\in \Omega$ such that the assumptions of Proposition~\ref{thm:Agmon} hold, and,  to simplify notation, we set $(V,E)=(\cC_\infty(\om), \cO(\om))$. Note that the proof of Theorem~ \ref{thm:Agmon} would actually also work if $(V,E)$ was any fixed graph with bounded degree satisfying Assumption~\ref{ass:graph}, and not necessarily a subset of $\bbZ^d.$

\subsubsection{Notation}
For a given set $A\subset V$, we define the internal boundary of $A$ by
\begin{align*}
\partial_{\mathrm{int}} A \; \ldef \; \big\{ x\in A \, | \, \exists\, y \in A^c \text{ such that } \{x,y\}\in E  \big\},
\end{align*}
and the outer boundary of $A$ by
\begin{align*}
\partial_{\mathrm{out}} A \; \ldef \; \big\{ x\in A^c \, | \, \exists\, y \in A \text{ such that } \{x,y\}\in E  \big\}.
\end{align*}
For any edges $e\in E$ we denote by $e^-,e^+\in V$ the unique vertices such that $e=\{e^-,e^+\}$ and $e^+-e^-\in\{e_1,\ldots,e_d\},$ where $e_1,\ldots,e_d$ denotes the canonical basis of $\bbZ^d.$ For $f:V\to\bbR$ and $e\in E$ we define the \textit{discrete derivative} 
\begin{align*}
  \nabla f: E \rightarrow \bbR, \qquad 
  \nabla f(e)\; \ldef \; f(e^+)-f(e^-),
\end{align*}
and note that for $f,g:V \to\bbR$, the discrete product rule takes the form 
\begin{align}\label{eq:rule:prod}
  \nabla(f g)
  \;=\;
  \av{f} \nabla g + \av{g} \nabla f,
\end{align}
where $\av{f}(e) \ldef \frac{1}{2} (f(e^+) + f(e^-))$. Further, for any $f:V\rightarrow [0,\infty)$, note that 
\begin{equation}
\label{eq:Holderforav}
	\av{f^{\al_1}} \av{f^{\al_2}} \leq \av{f^{\al_1+\al_2}}
\end{equation} 
for any $\al_1, \al_2 \geq 0$. We define the \textit{discrete divergence} of a function $F:E\to\bbR$ by
\begin{align*}
\nabla^*F(x):=\sum_{\substack{e\in E\\ e^+=x}}F(e)-\sum_{\substack{e\in E\\ e^-=x}}F(e)=\sum_{i=1}^dF(\{x-e_i,x\})-F(\{x,x+e_i\}). 
\end{align*}
Since for all $f\in \ell^2(V)$ and $F \in \ell^2(E)$ we have
\begin{align} \label{eq:adjoint}
\langle \nabla f, F\rangle_{\ell^2(E)} = \langle f, \nabla^*F \rangle_{\ell^2(V)},
\end{align}
$\nabla^*$ can be seen as the adjoint of $\nabla$.  
Note that the generator $\cL^\om$ defined in \eqref{eq:defL} above is a finite-difference operator in divergence form as it can be rewritten as
  \begin{align} \label{eq:L_divform}
    \big(\cL^{\om} f)(x)
    \;=\;
   - \nabla^*(a^\omega\nabla f)(x) - h \,\kappa^\om(x) \, f(x).
  \end{align}
 On the  Hilbert space $\ell^2(V)$ the \emph{Dirichlet form} associated with $\cL^{\om}$ is given by
\begin{align} \label{eq:def:dform}
  \cE^{\om}(f,g)
 & \;\ldef\;
  \scpr{f}{-\cL^{\om} g}{V}
  \;=\;
  \scpr{\nabla f}{a^\om \nabla g}{E} + h \, \scpr{f}{\kappa^\om g}{V} \nonumber \\
 & \;=\;
  \scpr{1}{\md \Ga^{\om}(f,g)}{E} + h \, \scpr{f}{\kappa^\om g}{V},
\end{align}
where $\md \Ga^{\om}(f,g) \ldef a^\om \nabla f \nabla g$. Further, set $\cE^{\om}(f):=\cE^{\om}(f,f)$ and,  for any $\eta: V \rightarrow \bbR$,
\begin{align*}
\Gamma^\om_\eta (f,g)(x) & \; \ldef \; \sum_{y\sim x} \av{\eta^2}(\{x,y\}) \, \md \Ga^{\om}(f,g)(\{x,y\}), \qquad \Gamma^{\omega}(f,g)\ldef \Gamma^{\omega}_{1}(f,g).
\end{align*}
 Note  that for all $f,g,v,\eta:V\rightarrow\bbR$,
\begin{equation}
\label{eq:fromdGammatoGamma}
\scpr{\av{v}\av{\eta^2}}{\md\Ga^{\om}(f,g)}{E}=\frac12\scpr{v}{\Ga^{\om}_{\eta}(f,g)}{V}.
\end{equation}

\subsubsection{Maximal inequality}

The starting point to prove Theorem~\ref{thm:Agmon} is to show the following maximal inequality for (super-)harmonic functions. Recall the constant $d'$ from Assumption~\ref{ass:graph}-(ii).

\begin{prop} \label{prop:max_ineq}
 For any $x_0 \in V$, $n \geq 2( N_4(x_0) \vee N_5(x_0))$ fixed, write $B(\si) \equiv B^\om(x_0, \sigma n)$, $\sigma\in[1/2,1]$. Let $u \geq 0$ be such that $\cL^\om\, u \geq 0$ on $B^\om(x_0, n)$ and set $v \ldef \varphi \cdot u$ for any  $\varphi: V \rightarrow (0,\infty)$.  Then, under Assumption \ref{ass:graph},  for all $\al \in (0, \infty)$  and any $p,q \in (1, \infty]$ with $1/p+1/q<2/d'$,   there exist $\gamma =\gamma(d',p,q, \alpha) \in (1, \infty)$ and $c_{31}=c_{31}(d,p,q,\alpha) \in (0, \infty)$ such that
for all  $1/2 \leq \si' < \si \leq 1$, 
  \begin{align} \label{eq:max_ineq}
    \max_{x \in B(\si')} v(x)
    \;\leq\;   
    \bigg(
      c_{31} \,
      \frac{m^\om(n)}
           {(\si - \si')^{2}}
    \bigg)^{\!\!\gamma}\; 
    \Norm{v}{\al, B(\si)},
  \end{align}
 where $  m^\om(n)  \ldef \Norm{1 \vee \mu^\om}{p,B^\om(x_0,n)}\, \Norm{1 \vee \nu^\om}{q,B^\om(x_0,n)} \, \Big(1+ \Norm{\Gamma^\om(\varphi,\varphi^{-1})}{p,B^\om(x_0,n)} \, n^2 \Big)$. 
\end{prop}
The rest of this subsection is devoted to the proof of Proposition~\ref{prop:max_ineq}. Similar maximal inequalities have been shown in \cite{ADS16} and \cite{ADS19} in order to obtain  Harnack inequalities and the heat kernel estimates restated in Theorem~\ref{thm:hke}, respectively. Here we will follow the arguments in \cite{ADS16,ADS19} rather closely with some adjustments being required. The main argument in Proposition~\ref{prop:mos_it} below will be based on a Moser iteration scheme. As a first step we show the following Caccioppoli-type estimate. Its proof is an adaptation of the arguments in \cite[Lemma~3.7]{ADS19} which provides an energy bound for perturbed space-time harmonic functions. Since this estimate plays also a crucial role for the implementation of Agmon's method below and since we need to keep track of an additional term containing the killing measure, we give a full proof here.

\begin{lemma} \label{lem:EHI:DF_ualpha}
 Consider a connected, finite subset $B \subset V$ and a function $\eta$ on $V$ with
  \begin{align*}
    \supp \eta \;\subset\; B, \qquad
    0 \;\leq\; \eta \;\leq\; 1 \qquad \text{and} \qquad
    \eta \equiv 0 \quad \text{on} \quad \partial_{\mathrm{int}}B.
  \end{align*}
  Further, let $u \geq 0$ be such that $\cL^\om\, u \geq 0$ on $B$ and set $v \ldef \varphi \cdot u$ for any $\varphi: V \rightarrow (0,\infty)$. There exists an absolute constant $c_{32}<\infty$ such that for all $\al \geq 1$, 
  \begin{equation*}
  \cE^{\om}(\eta v^{\al})\leq c_{32} \alpha^2\scpr{v^{2\alpha}}{\Gamma^{\om}(\eta,\eta)-\Gamma_{\eta}^{\om}(\varphi,\varphi^{-1})}{V}.
  \end{equation*}
%
%
\end{lemma}
\begin{proof}
  Since $\cL^\om u \geq 0$ and $u=\varphi^{-1}v$, \eqref{eq:adjoint}, \eqref{eq:L_divform} and an application of the  product rule \eqref{eq:rule:prod} yield
  \begin{align}\label{eq:DF:est0}
    0 
    \;\geq\; &
    \scpr{\eta^2 \varphi v^{2\al-1}}{-\cL^\om (\varphi^{-1}v)}{V} \nonumber  \\
    \;=\; &
    \scpr{\nabla (\eta^2 \varphi v^{2\al-1})}{a^\om\, \nabla (\varphi^{-1} v)}{E}
    +  h \,  \scpr{\eta^2 v^{2\al}}{\kappa^\om}{V} \nonumber \\
   \;=\; &
    \scpr{\av{\eta^2}}{\md \Ga^{\om}(\varphi v^{2\al-1}\!, \varphi^{-1} v)}{E}
    \,+\,
    \scpr{\av{\varphi v^{2\al-1}}}{\md \Ga^{\om}(\eta^2\!, \varphi^{-1} v)}{E} \nonumber  \\
   & \, + \,  h \, \scpr{\eta^2 v^{2\al}}{\kappa^\om}{V} \nonumber \\
     \;\rdef\; &
    T_1 \,+\, T_2 +T_3.
  \end{align}
   Let us first bound the term $T_1$. Note that  $ -\av{\varphi^{-1}} (\nabla \varphi) = \av{\varphi} (\nabla \varphi^{-1})$ and $(\nabla \varphi)(\nabla \varphi^{-1}) \leq 0.$ Combining these observations with   \eqref{eq:Holderforav} and the product rule \eqref{eq:rule:prod}, we obtain the following lower bound
  \begin{align}\label{eq:firstboundT1}
&    \md \Ga^{\om}(\varphi v^{2\al-1}\!, \varphi^{-1} v)
    \;\geq\;
    \av{\varphi} \av{\varphi^{-1}}\, 
    \md \Ga^{\om}(v^{2\al-1}\!, v)
    \,+\,
    \av{v^{2\al}}\, \md \Ga^{\om}(\varphi, \varphi^{-1})
 \nonumber  \\
    &\mspace{36mu}+\,
    \av{\varphi}
    \Big(
      \av{v}\, \md \Ga^{\om}(v^{2\al-1}, \varphi^{-1}) 
      - 
      \av{v^{2\al-1}}\, \md \Ga^{\om}(v, \varphi^{-1})
    \Big).
  \end{align}
Further, by \cite[Lemma~A.1-(ii)]{ADS16},  we have for all $a,b \geq 0$,
      \begin{align*}
        \big(a^{\al} - b^{\al}\big)^2
        \;\leq\;\frac{\al^2}{2\al-1}\, \big(a - b\big)\,
        \big(a^{2\al-1} - b^{2\al-1}\big),
      \end{align*}
 so that
  \begin{align} \label{eq:boundGammav2a-1v}
    \md \Ga^{\om}(v^{2\al-1}, v)
    &\;\geq \;
    \frac{2\al-1}{\al^{2}}\, \md\Ga^{\om}(v^{\al}\!, v^{\al}).
  \end{align}
Moreover, by \cite[Lemma~B.1-(ii)]{ADS16a},  for $a, b \geq 0$, 
      \begin{align*}
        \big| a^{2 \al - 1} b \,-\, a b^{2 \al -1} \big|
        \;\leq\;
      \frac{\al-1}{\al} 
        \big| a^{2 \al} - b^{2 \al} \big|,
      \end{align*}
  which implies that
  \begin{align}\label{eq:est:v}
    &\big|
      \av{v}(e) \nabla v^{2\al-1}(e) - \av{v^{2\al-1}}(e) \nabla v(e)
    \big|
    \nonumber\\
    &\mspace{36mu}=\;
    \big|v^{2\al -1}(e^{+}) v(e^{-}) - v^{2\al-1}(e^{-}) v(e^{+}) \big|
    \;\leq\;
    \frac{2(\al-1)}{\al}\, \big| \av{v^{\al}}(e)\, \nabla v^{\al}(e) \big|
  \end{align}
  for all $e \in E$. By using the estimates \eqref{eq:boundGammav2a-1v} and \eqref{eq:est:v} in \eqref{eq:firstboundT1} we get
  \begin{align}\label{eq:secondboundT1}
  &    \md \Ga^{\om}(\varphi v^{2\al-1}\!, \varphi^{-1} v)
    \;\geq\;
    \frac{2\al-1}{\al^2}\,    \av{\varphi} \av{\varphi^{-1}}\, \md\Ga^{\om}(v^{\al}\!, v^{\al}) 
    \,+\,
    \av{v^{2\al}}\, \md \Ga^{\om}(\varphi, \varphi^{-1})
 \nonumber  \\
    &\mspace{36mu}-\,
    \frac{2(\al-1)}{\al} \av{\varphi} \av{v^\alpha} \big| \md \Ga^{\om}(v^\alpha, \varphi^{-1}) \big|
  \end{align}
  Notice that
  \begin{align}    \label{eq:av+phi:sqrt}
  \begin{split}
    \av{\varphi} \big|\nabla \varphi^{-1}\big| 
    \;&=\; 
    \sqrt{ \av{\varphi} \av{\varphi^{-1}} } 
    \cdot \sqrt{-(\nabla \varphi)(\nabla \varphi^{-1}) }.
  \end{split}
  \end{align}
Hence, we use  Young's inequality in the form 
\begin{align}
\label{eq:youngs}
|a b| \leq \frac{1}{2}(\ve a^2 + b^2/\ve).
\end{align}
with $\ve = 1/(2\al),$  $a=\nabla v^{\alpha}\sqrt{\av{\varphi}\av{\varphi^{-1}}}$ and $b=\av{v^{\alpha}} \sqrt{-(\nabla \varphi)(\nabla \varphi^{-1}) }$, which  together with another application of \eqref{eq:Holderforav} yields
\begin{align}
\av{\varphi} \av{v^\alpha} \big| \md \Ga^{\om}(v^\alpha, \varphi^{-1}) \big| \; \leq \;
\frac 1 {4\alpha}  \av{\varphi} \av{\varphi^{-1}}   \md \Ga^{\om}(v^\alpha) 
-
\al \av{v^{2\al}}  \md \Ga^{\om}(\varphi, \varphi^{-1}).
\end{align}
Combining this with \eqref{eq:secondboundT1} gives
  \begin{align}  \label{eq:enTerm1}
    T_1
    \;\geq\; &
    \frac{3\al-1}{2\al^2} 
    \scpr{\av{\eta^2} \av{\varphi} \av{\varphi^{-1}}}
    {\md \Ga^{\om}(v^{\al}\!, v^{\al})}{E} \nonumber \\
   & \,+\, 
    (2\al-1)\, \scpr{\av{\eta^2} \av{v^{2\al}}}{\md \Ga^{\om}(\varphi, \varphi^{-1}}{E}. 
  \end{align}

 Next we bound the term $T_2$.  Observe that since $\av{\varphi v^{2\alpha-1}}\leq2\av{\varphi}\av{v^{2\alpha-1}}$ it follows from the product rule \eqref{eq:rule:prod} that
  \begin{align} \label{eq:firstboundT2}
    &\av{\varphi v^{2\al-1}}\,
    \md \Ga^{\om}(\varphi^{-1} v, \eta^2) \nonumber
    \\
    &\mspace{36mu}\geq\;
    -4 \av{\eta} \av{\varphi} \av{v^{2\al-1}}\,
    \Big( 
      \av{\varphi^{-1}} \, \big|\md \Ga^{\om}(v, \eta)\big|
      \,+\,
      \av{v}\, \big|\md \Ga^{\om}(\varphi^{-1}, \eta)\big|
    \Big).
  \end{align}
 Using again \eqref{eq:Holderforav}, \eqref{eq:av+phi:sqrt} and \eqref{eq:youngs} with $\ve=4,$ $a=\nabla \eta\sqrt{\av{\varphi}\av{\varphi^{-1}}}$ and $b=\av{\eta} \sqrt{-(\nabla \varphi)(\nabla \varphi^{-1}) }$  we get on the one hand
  \begin{align} \label{eq:est_T2_1}
    4 \av{\eta} \av{\varphi} \,
    \big|\md \Ga^{\om}(\varphi^{-1}, \eta)\big|
    \leq8\,\av{\varphi} \av{\varphi^{-1}} \,
    \md \Ga^{\om}(\eta, \eta)
    \,-\,
    \frac{1}{2} \av{\eta^2} \, 
    \md \Ga^{\om}(\varphi, \varphi^{-1}).
  \end{align}
  On the other hand,
  \begin{align*}
  &  \big| \av{v^{2\al-1}}(e) (\nabla v)(e) \big| \\
  & \mspace{36mu}   \leq\;
    \Big|
      \av{v^{\al}}(e) (\nabla v^{\al})(e)
     \Big| \,+\,
      \frac{1}{2}\, \Big|
      \big(v^{2\al -1}(e^{+}) v(e^{-}) - v^{2\al-1}(e^{-}) v(e^{+}) \big)
    \Big|
    \\
    & \mspace{36mu} \overset{\!\!\!\eqref{eq:est:v}\!\!\!}{\leq} \;
    \frac{2\al-1}{\al}\,
    \big| \av{v^{\al}}(e)\, \nabla v^{\al}(e) \big|.
  \end{align*}
  Thus, using again \eqref{eq:Holderforav} and \eqref{eq:youngs} with  $\ve = 1/(4\al)$, $a=\av{\eta}\nabla v^{\alpha}$ and $b=\av{v^{\alpha}}\nabla\eta$ we get
  \begin{align} \label{eq:est_T2_2}
    4 \av{\eta}  \av{v^{2\al-1}}\, 
    \big|\md &\Ga^{\om}(v, \eta)\big|
    \leq 4\, \frac{2\al-1}{\al}\, 
     \av{\eta} \av{v^{\al}}\,
    \big| \md \Ga^{\om}(v^{\al}, \eta) \big| 
 \nonumber   \\
    &\leq
      \frac{2\al-1}{2\al^2}\, \av{\eta^2}\, 
      \md \Ga^{\om}(v^{\al}\!, v^{\al})
      \,+\,
      8(2\al-1)\, \av{v^{2\al}}\,
      \md \Ga^{\om}(\eta, \eta).
  \end{align}
  Hence, by using the estimates \eqref{eq:est_T2_1} and \eqref{eq:est_T2_2} in \eqref{eq:firstboundT2} and applying again \eqref{eq:Holderforav} we get
  \begin{align*}
    &\av{\varphi v^{2\al-1}}\,
    \md \Ga^{\om}(\varphi^{-1} v, \eta^2) 
   \geq \;
      -\frac{2\al-1}{2\al^2} 
\av{\eta^2} \av{\varphi} \av{\varphi^{-1}}
    \md \Ga^{\om}(v^{\al}\!, v^{\al}) \\
   & \mspace{36mu} 
   -16 \al   \av{\varphi} \av{\varphi^{-1}}\av{v^{2\alpha}}  \, \md \Ga^{\om}(\eta, \eta)
   +   \frac{1}{2} \av{\eta^2} \,  \av{v^{2\al}} \, 
    \md \Ga^{\om}(\varphi, \varphi^{-1}).
  \end{align*} 
Since
  \begin{align*}
    \av{\varphi^{-1}} \av{\varphi}
    \;=\;
    1 -
    \dfrac{1}{4}\, (\nabla \varphi) (\nabla \varphi^{-1})
  \end{align*}
  and $|\nabla \eta|^2 \leq 2 \av{\eta^2}$ we obtain the lower bound
  \begin{align}  \label{eq:enTerm2}
    T_2
    \;\geq\; &
    -\frac{2\al-1}{2\al^2} 
    \scpr{\av{\eta^2} \av{\varphi} \av{\varphi^{-1}}}
    {\md \Ga^{\om}(v^{\al}\!, v^{\al})}{E} \nonumber \\
  &  -\,
    16 \al\, \scpr{\av{v^{2\al}}}{ \md \Ga^{\om}(\eta, \eta)}{E} \nonumber \\
  & + \Big(8\al + \frac 1 2 \Big) \scpr{\av{\eta^2} \av{v^{2\al}}} 
    {\md \Ga^{\om}(\varphi, \varphi^{-1})}{E}.
  \end{align} 
Therefore, by combining \eqref{eq:DF:est0} with \eqref{eq:enTerm1} and \eqref{eq:enTerm2} we get
\begin{align*}
0   \geq  &  \;
\frac{1}{2\al} 
    \scpr{\av{\eta^2} \av{\varphi} \av{\varphi^{-1}}}
    {\md \Ga^{\om}(v^{\al}\!, v^{\al})}{E}-16 \al  \scpr{\av{v^{2\al}}}{ \md \Ga^{\om}(\eta, \eta)}{E}
    \nonumber\\[.5ex]
    &+ \Big(10\al - \frac 1 2 \Big) \scpr{\av{\eta^2} \av{v^{2\al}}} 
    {\md \Ga^{\om}(\varphi, \varphi^{-1})}{E}+h \, \scpr{\eta^2 v^{2\al}}{\kappa^\om}{V}
\nonumber\\[.5ex]
   \geq & \;
\frac{1}{4\al}   
\cE^\om(\eta v^\al) \,-\,
    \Big(16 \al + \frac{1}{2\alpha} \Big)   \, \scpr{\av{v^{2\al}}}{ \md \Ga^{\om}(\eta, \eta)}{E}\,  \nonumber\\[.5ex]
    &\
   + \Big(10\al - \frac 1 2 \Big) \scpr{\av{\eta^2} \av{v^{2\al}}} 
    {\md \Ga^{\om}(\varphi, \varphi^{-1})}{E},
\end{align*}
where we used in the last step that $\av{\varphi} \av{\varphi^{-1}} \geq 1$ and that by \eqref{eq:rule:prod} and \eqref{eq:Holderforav}
\begin{align*}
    \av{\eta^2}\, \md \Ga^{\om}(v^{\al}\!, v^{\al})
    \;\geq\; \frac  1 2
    \md \Ga^{\om}(\eta v^{\al}\!, \eta v^{\al})
    \,-\,
    \av{v^{2\al}}\, \md \Ga^{\om}(\eta, \eta).
  \end{align*}  
Using \eqref{eq:fromdGammatoGamma}, we can conclude.
\end{proof}

The maximal inequality for $v=u\cdot \varphi$ in Proposition~\ref{prop:max_ineq} will be obtained via a Moser iteration which is carried out in the next proposition. The argument is similar to the one used in \cite[Proposition~3.2]{ADS16} in order to show maximal inequality for harmonic functions and to obtain an elliptic Harnack inequality. Here some extra care is needed due to the presence of the perturbation $\varphi$. Besides the Cacciopoli-type estimate in Lemma~\ref{lem:EHI:DF_ualpha} the second main ingredient is the following weighted Sobolev inequality.  
 For any $x_0 \in V$ and $n \geq N_5(x_0)$,  Assumption~\ref{ass:graph}-(ii) ensures \cite[Equation~(23)]{ADS15} to hold for functions supported in $B^{\om}(x_0,n)$, but with $C_{S_1}$ replaced by $C_{S_1} n |B^{\om}(x_0,n)|^{\frac{d'-1}{d'}}/|B^{\om}(x_0,n)|$.  Therefore, adapting the proof of \cite[Equation~(28)]{ADS15}, one can show that  for any $f:V\rightarrow\bbR$ with $\supp f\subset B^{\om}(x_0,n)$, 
\begin{equation} \label{eq:weightedSobolev}
\begin{split}
\Norm{f^2}{\varrho,B^{\om}(x_0,n)}&\leq \Bigg(\frac{C_{S_1}n \big|B^{\om}(x_0,n)\big|^{\frac{d'-1}{d'}}}{|B^{\om}(x_0,n)|}\Bigg)^{\!2}   \big|B^{\om}(x_0,n)\big|^{2/d'} \Norm{\nu^{\om}}{q,B^{\om}(x_0,n)}\frac{\cE^{\om}(f)}{|B^{\om}(x_0,n)|}
\\&\leq C_{S_1}^2n^2\Norm{\nu^{\om}}{q,B^{\om}(x_0,n)}\frac{\cE^{\om}(f)}{|B^{\om}(x_0,n)|},
\end{split}
\end{equation}
where $\varrho\ldef   d'/(d'-2+d'/q)$. Combining this with Lemma~\ref{lem:EHI:DF_ualpha} then yields an upper bound of   $\Norm{v}{2 \al^{k+1} p_*, B}$ in terms of  $ \Norm{v}{2 \al^k p_*, B'}$ with a ball $B'$ slightly bigger than $B$, $p_*\ldef p/(p-1)$ and $k\in \bbN_0$, for some $\alpha=\alpha(d',p,q)$ which is strictly bigger than one thanks to our condition on $p$ and $q$, see \eqref{eq:EHI:iter:1} below. Iterating this inequality then gives the desired maximal inequality.

\begin{prop} \label{prop:mos_it}
  For any $x_0 \in V$, $n \geq 2( N_4(x_0) \vee N_5(x_0))$ fixed, write $B (\si) \equiv B^\om(x_0, \sigma n)$, $\sigma\in[1/2,1]$.
   Let $u \geq 0$ be such that $\cL^\om\, u \geq 0$ on $B(1)$ and set $v \ldef \varphi \cdot u$ for any  $\varphi: V \rightarrow (0,\infty)$.  Then, under Assumption \ref{ass:graph}, for any $p,q \in (1, \infty]$ with $1/p+1/q<2/d'$,
  there exist $\gamma' =\gamma'(d',p,q) \in (1, \infty)$ and $c_{33} = c_{33}(d',p,q)$ such that for all $\be \in [2 p_*,\infty]$ and for all $1/2 \leq \si' < \si \leq 1$,
  \begin{align}\label{eq:EHI:mos_it}
    \Norm{v}{\be,B(\si')}
    \;\leq\;
    c_{33}\,
    \bigg( 
      \frac{m^\om(n)}
           { (\si - \si')^{2}}
    \bigg)^{\!\!\gamma'}\; \Norm{v}{2 p_*,B(\sigma)},
  \end{align}
  where $p_* = p / (p-1)$ and $ m^\om(n)$ as in Proposition~\ref{prop:max_ineq}.

\end{prop}
\begin{proof}
  For fixed $1/2 \leq \si' < \si \leq 1$, let $\{B(\si_k)\}_k$ be a sequence of balls with radii $\si_k n$ centred at $x_0$, where  
  \begin{align*}
    \si_k \;=\; \si' + \tau_{k-1}
    \qquad \text{and} \qquad
    \tau_k \;=\; 2^{-k-1} (\si - \si'),
    \quad k \in{\bbN_0}.
  \end{align*}
  Note that $\si_k = \si_{k+1} + \tau_k,$ $\si_0 = \si$ and $\si_k\rightarrow\sigma'$ as $k\rightarrow\infty$.  Further, for any $k\in \bbN_0$ set 
  $\al_k \ldef \al^k$ where $\al\ldef \varrho/p_*$  and $\varrho\ldef   d'/(d'-2+d'/q)$ as in \eqref{eq:weightedSobolev} above.  Since $1/p + 1/q < 2/d'$ we have $\varrho > p_*$ and therefore $\al_k > 1$ for every $k$.  Due to the discreteness of the underlying space $\bbZ^d$, we distinguish two different cases.

  Let us first consider  the case $\tau_k n < 1$, that is $B(\si_{k+1}  ) = B(\si_k )$.  Set $\al_*\ldef \al/(\al-1)$. Then, note that
  \begin{align*}
  &  \Norm{v^{2 \al_k}}{\al p_*, B(\si_{k+1} )}
     \; = \; 
     \Norm{v^{2 \al_k} v^{2 \al_k(\al-1)}}{p_*, B(\si_{k+1} )}^{1/\al} \\
     & \mspace{36mu}
    \;\leq\;
    \Norm{v^{2 \al_k}}{p_*, B(\si_k )}^{1/\al}
    \Big( \max_{x \in B(\si_k )} v(x)^{2 \al_k} \Big)^{\!\!1/\al_*}   
    \;\leq\;
    |B(\si_k )|^{1/(\al_* p_*)} \Norm{v^{2 \al_k}}{p_*, B(\si_k )},
  \end{align*}
 where we used the trivial estimate $\max_{x \in B(\si_k )} v(x)^{2 \al_k} \leq \big(\sum_{x \in B(\si_k )} v(x)^{2 \al_k p_*} \big)^{1/p_*}$ in the last step. Since   $d/(2 \al_* p_*) \leq 1$ and $n < 1 / \tau_k$, we get by Assumption~\ref{ass:graph}-(i),
 \begin{align*}
   |B(\si_k )|^{1/(\al_* p_*)} \; \leq \;
     |B^\om(x_0,n )|^{1/(\al_* p_*)} \; \leq \; c \tau_k^{-1},
 \end{align*}
  and therefore
  \begin{align}\label{eq:EHI:iter:0}
    \Norm{v}{2 \al_{k+1} p_*,B(\si_{k+1})}
    \;\leq\;
    \bigg(
      c\, \frac{2^{2k}}{(\si - \si')^{2}}\,
    \bigg)^{\!\!1/(2\al_k)}\,
    \Norm{v}{2 \al_k p_*,B(\si_k)}.
  \end{align} 

  Consider now the case $\tau_k n \geq 1$.  Let $\eta_k$ be a cut-off function with $\supp \eta_k \subset B(\si_{k} )$ such  that $\eta_k \equiv 1$ on $B(\si_{k+1})$, $\eta_k \equiv 0$ on $\partial_{\mathrm{int}} B(\si_{k})$ and with linear decay on $B(\si_{k} ) \setminus B(\si_{k+1})$ so that $\Norm{\nabla \eta}{\ell^\infty(E)} \leq 1/(\tau_{k} n)$.  Recall that by volume regularity, see Assumption \ref{ass:graph}-(i),
 $|B(\si)| / |B(\si')| \leq C^2_{\mathrm{reg}} 2^d$.
Then, by taking  $f=\eta_kv^{\alpha_k}$ in \eqref{eq:weightedSobolev} we get
  \begin{align*}
    &\Norm{ v^{2\al_k}}{\varrho, B(\si_{k+1}) }
    \; \leq \;
   c  \, \Norm{(\eta_k\, v^{\al_k})^2}{\varrho, B(\si_{k} )}
    \; \leq\;
cn^2 \Norm{\nu^{\om}}{q,B^{\om}(x_0,n)}\,   
      \frac{\cE^{\om}(\eta_k v^{\al_k})}{|B(\si_{k} )|}.
  \end{align*}
  On the other hand, using Lemma~\ref{lem:EHI:DF_ualpha}, the fact that $\supp(\eta_k)\subseteq B(\sigma_k)$ with $\av{\eta_k^2}\leq1$ and $\Gamma^{\om}(\eta,\eta)\leq \mu^{\om}/(\tau_kn)^2$, and  H\"older's inequality we find
  \begin{align*}
    \frac{\cE^{\om}(\eta_k v^{\al_k})}{|B(\si_k)|}
   & \; \leq \;
  c \, \alpha_k^2  \, \Big( \Norm{v^{2 \al_k} \Gamma^\om(\eta_k,\eta_k)}{1,B(\si_{k})} + \Norm{v^{2 \al_k} \Gamma^\om(\varphi,\varphi^{-1})}{1,B(\si_{k})} \Big)
 \\    
  & \; \leq \;  
    c \, \alpha_k^2 \, \bigg(\frac{1}{(\tau_kn)^2} +\Norm{\Gamma^\om(\varphi,\varphi^{-1})}{p,B(\si_k )} \bigg) 
    \Norm{1 \vee \mu^{\om}}{p,B(\si_k)}
    \Norm{v^{2 \al_k}}{p_*,B(\si_{k})}.
  \end{align*}
Hence, using $\alpha_{k+1} p_* = \alpha_k \varrho$ we  combine the last two estimates to obtain that
  \begin{align}\label{eq:EHI:iter:1}
    \Norm{v}{2 \al_{k+1} p_*, B(\si_{k+1})}
    &\;\leq\;
    \bigg(
      c\;
      \frac{2^{2k}\, \al_k^2}{(\si -\si')^{2}}\;
      m^\om(n)
    \bigg)^{\!\!1/(2\al_k)}\;
    \Norm{v}{2 \al_k p_*,B(\si_{k})}.
  \end{align}
 By iterating the inequalities \eqref{eq:EHI:iter:0} and \eqref{eq:EHI:iter:1}, respectively, and using the fact that $\sum_{k=0}^{\infty}k/\al_k < \infty$ and $B(\sigma')\subseteq B(\sigma_K)$ there exists $c=c(d',p,q)< \infty$ such that, for any $K \in \bbN$,
  \begin{align*}
    \Norm{v}{2 \al_K p_*, B(\si')}
    &\;\leq\;
    c\, \prod_{k=0}^{K-1}
    \bigg(
      \frac{m^\om(n)}
      {(\si -\si')^{2}}
    \bigg)^{\!\!1/(2\al_k)}\;
    \Norm{v}{2 p_*,B(\si)}.
  \end{align*}
  Setting $\gamma' \ldef \frac{1}{2} \sum_{k=0}^\infty (1/\al_k) < \infty$  we get
  \begin{align*}
    \max_{x \in B(\si')} v(x)
    &\;=\;
    \lim_{K \to \infty} \Norm{v}{2 \al_K p_*, B(\si')}
  \;\leq\;
    c\,
    \bigg(
      \frac{m^\om(n)}
      {(\si -\si')^{2}}
    \bigg)^{\!\!\gamma}\;
    \Norm{v}{2 p_*, B(\si)}.
  \end{align*}
  For any $\be \in [2p_*, \infty)$ the claim is immediate since $\|v\|_{\be, B(\si')} \leq \max_{x \in B(\si' )} v(x)$. 
\end{proof}

\begin{proof}[Proof of Proposition~\ref{prop:max_ineq}]
 The maximal inequality \eqref{eq:max_ineq}, with  $\gamma=\big(1\vee \frac{2p_*}{\alpha}\big) \gamma'>1$, follows now from Proposition~\ref{prop:mos_it} as in \cite[Corollary~3.9]{ADS15}, cf.\ also \cite[Corollary~3.4]{ADS16}.
\end{proof}

\subsubsection{Exponential decay via Agmon's method}

We will  now use Lemma~\ref{lem:EHI:DF_ualpha} to bound averages of $v^2$ in $B'$ by averages of $v^2$ in $B\setminus B'$ for $B'\subset B.$ Recall the constant $c_{32}$ from Lemma~\ref{lem:EHI:DF_ualpha}. 

\begin{lemma} \label{lem:agmon}
 Let  $B \subsetneq V$ be connected and $B'\subset B$. Let $\eta:V\rightarrow [0,1]$ such that $ \supp \eta \subset B$,    
    $\eta \equiv 1$ on $B'\cup \partial_{\mathrm{out}} B'$ and 
    $\eta \equiv 0$ on $\partial_{\mathrm{int}} B$.
  Further, let $u \geq 0$ be such that $\cL^\om\, u \geq 0$ on $B$ and let $\varphi > 0$ be such that 
 \begin{align} \label{eq:cond_phi}
 \big| \Gamma^\om (\varphi, \varphi^{-1}) \big| \; \leq \; \frac{h}{2c_{32}} \, \kappa^\om \quad \text{on $V$,}
\end{align}  
and set $v \ldef \varphi \cdot u$. Then there exists an absolute constant $c_{34}<\infty$ such that
\begin{align*}
 \sum_{x \in B'} v^2(x) \, \kappa^\om(x) \; \leq \; c_{34} \, \left(1 + h^{-1} \Norm{\nabla \eta}{\ell^\infty(E)}^2 \right)  \sum_{x\in B \setminus B'} v^2(x)  \, \big(  \mu^\om(x) \vee \kappa^\om(x) \big).
\end{align*}
\end{lemma}

\begin{proof}
Note that $\Gamma^\om_\eta(\varphi,\varphi^{-1}) \geq \Gamma^\om(\varphi,\varphi^{-1})\indicator_{B}$ as $\av{\eta^2} \leq 1,$ $\supp\eta\subset B\setminus \partial_{\mathrm{int}} B$ and $\md \Gamma^\om (\varphi, \varphi^{-1})\leq 0$. Hence, by Lemma~\ref{lem:EHI:DF_ualpha} with $\al=1$, 
\begin{equation*}
h\scpr{v^2}{\kappa^{\om}\eta^2}{B} \leq \cE^{\om}(\eta v) \leq c_{32} \scpr{v^{2}}{\Gamma^{\om}(\eta,\eta)}{V}-c_{32}\scpr{v^2}{\Gamma^{\om}(\varphi,\varphi^{-1})}{B}.
\end{equation*}
By rearranging and using \eqref{eq:cond_phi} we obtain that
\begin{align*}
  c_{32}\scpr{v^2}{\Gamma^\om(\eta,\eta)}{V} & 
\geq h\scpr{v^2}{ (\eta^2 -\tfrac 1 2)  \, \kappa^\om}{B}  \\
& 
= \frac h 2  \scpr{v^2}{ \kappa^\om}{B'}
+  
h\scpr{v^2}{( \eta^2 - \tfrac 1 2) \, \kappa^\om}{B\setminus B'}
\\
& 
\geq \;
\frac h 2  \scpr{v^2}{ \kappa^\om}{B'}
- \frac{h}2
\scpr{v^2}{\kappa^\om}{B\setminus B'}.
\end{align*}
Since $\Gamma^\om(\eta,\eta)(x) \leq  \Norm{\nabla \eta}{\ell^\infty(E)}^2 \, \mu^\om(x)\indicator_{B\setminus B'}$,  $x\in V$,  rearranging gives the claim.
\end{proof}

In the proof of Theorem~\ref{thm:GKdecay_killing} below we will apply Lemma~\ref{lem:agmon} for a perturbation function $\varphi$ of the form $\varphi(y)=\exp(-\lambda \sqrt{h} d_\kappa^\om(x,y))$ with $d^\om_\kappa$ as defined in \eqref{eq:def_chemdist}. In the next lemma we show that for such $\varphi$ the condition \eqref{eq:cond_phi} is satisfied, provided $\lambda >0$ is sufficiently small.

\begin{lemma} \label{lem:GammaPhi}
Let $\varphi=e^\psi$ with $\psi(y) \ldef - \lambda \sqrt{h} \, d^\om_\kappa (x,y)$ for $\lambda>0$. For every $\delta >0$  there exists $\lambda_0=\lambda_0(\delta,d)>0$ such that for all $\lambda \in (0, \lambda_0)$,
\begin{align*}
\big| \Gamma^\om (\varphi, \varphi^{-1}) \big| \; \leq \;  \delta \, h\, \kappa^\om \quad \text{on $V$.}
\end{align*}
\end{lemma}
\begin{proof}
Since the case $h=0$ is trivial, we consider $h\in (0,1]$ only. For any $y\in V$,
\begin{align*}
\big| \Gamma^\om (\varphi, \varphi^{-1}) (y) \big| & \; \leq \;
\sum_{z\sim y} a^\om(y,z) \, \Big| \big( e^{\psi(z)} - e^{\psi(y)} \big) \, \big( e^{-\psi(z)} - e^{-\psi(y)} \big) \Big| \\
& \; = \; 
2 \sum_{z\sim y} a^\om(y,z) \, \Big( \cosh\big( |\nabla \psi(\{y,z\}|\big) -1 \Big),
\end{align*}
and for any $e\in E$,
 \begin{align*}
    \big|\nabla \psi(e)\big|
    \;\leq\;
    \la\, \sqrt{h}\, \big|d_{\kappa}^{\om}(x,e^+) - d_{\kappa}^{\om}(x,e^-)\big|
    \overset{\eqref{eq:def_chemdist}}{\;\leq\;} \lambda \, \sqrt{h}
    \bigg(
      1 \wedge \frac{\kappa^\om(e^+) \wedge \kappa^\om(e^-)}{a^\om(e)}
    \bigg)^{\!\!1/2}.
  \end{align*}
  Recall that $h\in (0,1]$. Using that $b \big(\cosh(z)-1\big)\leq \cosh(\sqrt{b} z)-1$ for all $z\in \bbR$  and any $b \geq 1$, we obtain that
\begin{align*}
\big| \Gamma^\om (\varphi, \varphi^{-1}) (y) \big| & \; \leq \;
2h\, \kappa^\om(y) \sum_{z\sim y} \frac 1 h \Big( 1 \vee \frac{a^\om(y,z)}{\kappa^\om (y) \wedge \kappa^\om (z)} \Big) \, \Big(\cosh\big( |\nabla \psi(\{y,z\}|\big) -1 \Big) \\
& \; \leq \; 4d \,\big( \cosh(\lambda)- 1 \big) \, h \, \kappa^\om(y).
\end{align*}
Since $4d \big(\cosh(\lambda)- 1 \big) \leq \delta$ for $\lambda$ sufficiently small, the result follows.
\end{proof}

\begin{proof}[Proof of Theorem~\ref{thm:Agmon}]
Let us fix $x,y\in \cC_\infty(\om)$ such that
$ \rho^\om(x,y)  \geq N_{11}(\om,x)\ldef  8(N_4(\om,x) \vee N_5(\om,x))\vee2( N_7(\om,x) \vee N_{10}(\om,x))$ and set 
 $n\ldef  \rho^\om(x,y)/4$. In particular, $n\geq 2(N_4(\om,x) \vee N_5(\om,x))$. Take $B'=B^{\om}(x,n)$ and $B=B^{\om}(x,2n)$, which satisfy $|B|\leq c|B'|$ by Assumption \ref{ass:graph}-(i). Recall that the function $u(z)= g^\om(y,z)$ is harmonic on $V \setminus \{y\}$, and thus also on $B$.
 Finally, let $\varphi(z)= \exp(-\lambda \sqrt{h} d_\kappa^\om(x,z))$ for some $\lambda \in (0,1)$ only depending on $c_{32}$ such that \eqref{eq:cond_phi} holds (see Lemma~\ref{lem:GammaPhi}). 
Then, setting again $v\ldef u\cdot \varphi$, by the maximal inequality in Proposition~\ref{prop:max_ineq}, H\"older's inequality and Lemma~\ref{lem:agmon} we obtain that
\begin{align*}
& \max_{z\in B^{\om}(x, n/2)}  v(z)   \; \leq \; 
c \, m^\om(n)^{\gamma}\; 
    \Norm{v}{1, B'}
    \; \leq \; c \, m^\om(n)^{\gamma}\; \Norm{1/\kappa^\om}{1, B'}^{1/2} \, \Norm{v \sqrt{\kappa^\om}}{2, B'}    \\
 \leq \; &
     c \, m^\om(n)^{\gamma}\; \Norm{1/\kappa^\om}{1, B}^{1/2} 
     \left( \frac{ 1+ h^{-1} \, \Norm{\nabla \eta}{\ell^\infty(E)}^2 }{|B'|} \, \sum_{z\in B\setminus B'} v^2(z)  \, \big(  \mu^\om(z) \vee \kappa^\om(z) \big) \right)^{\!1/2} \\
     \leq   \; &
      c \, m^\om(n)^{\gamma}\; \Norm{1/\kappa^\om}{1, B}^{1/2}  \, \Norm{\mu^\om\vee \kappa^\om}{1, B}^{1/2} \,  \Big(1+ h^{-1} \, \Norm{\nabla \eta}{\ell^\infty(E)}^2 \Big)^{\! 1/2} \, \max_{z\in B\setminus B'} v(z).
\end{align*}
Take $\eta$ a linear cutoff function between $B'$ and $B$ so that  $\Norm{\nabla \eta}{\ell^\infty(E)} \leq c/n$. Recall \eqref{eq:int_condAgmon} and that by Lemma~\ref{lem:GammaPhi}, $\Norm{\Gamma^\om(\varphi,\varphi^{-1})}{p,B^\om(x,3n)} \leq c h \Norm{\kappa^\om}{p,B^\om(x,3n)} \leq c h$.
Hence, using the definition of $m^{\om}(n)$ in Proposition \ref{prop:max_ineq}, \eqref{eq:int_condAgmon} and the equality $\varphi(x)=1,$ there exists $c=c(d',p,q,C_{\mathrm{int}})$ such that
\begin{align} \label{eq:u_pointwise}
u(x) \leq    c \,  \big(1+hn^2\big)^\gamma \Big(1+\frac 1 {h n^2}\Big)^{\!1/2} \Big(\max_{z\in B\setminus B'}  u(z) \Big) \,    \Big( \max_{z \in (B')^c} \varphi(z)\Big). 
\end{align} 
 Moreover, recall that the Green's function with killing is trivially bounded from above by the Green's function without killing. Hence, as $\rho^\om(y,z)\geq  2n \geq c \rho^{\om}(x,y)$ for any $z \in  B\setminus B'$, we may apply Proposition~\ref{prop:green}, which implies that 
\begin{align*}
u(z) \; \leq \;  \frac{c}{\rho^\om(y,z)^{d-2}}    \; \leq \; \frac{c}{\rho^\om(x,y)^{d-2}}, \qquad \forall z  \in  B\setminus B'.
\end{align*}
Combining this with \eqref{eq:u_pointwise} yields the claim.
\end{proof}

\section{Examples}
\label{sec:examples}
In this section, we present a few models which satisfy the main conditions required in Theorems~\ref{The:main}, \ref{thm:hkeGauss} and \ref{thm:GKdecay_killing}, that is the conditions {\bf P1}, {\bf P2} and {\bf P3}, or {\bf P3'} or {\bf P3''}, introduced in Section~\ref{sec:FPP}, as well as the conditions {\bf S1} and {\bf S2} introduced in Section~\ref{sec:RCMhk}. Our main examples are the Gaussian free field, the Ginzburg-Landau $\nabla \phi$ interface model and random interlacements, for which conditions {\bf P1}, {\bf P2}, {\bf S1} and {\bf S2} were proved in \cite{MR3390739,2016arXiv161202385R}. However since our condition {\bf P3} is stronger than the one from \cite{MR3390739}, one needs to verify that it is still satisfied for these models. For the Gaussian free field and random interlacements, {\bf P3} will follow from an easy adaptation of the techniques from \cite{MR3325312,MR3420516}, see Propositions~\ref{app:GFFverifiesall} and~\ref{prop:condforRIverified}. For the interface model however the decoupling inequalities from \cite{2016arXiv161202385R} will not be sufficient for our purposes, and we will prove condition {\bf P3'} in Proposition~\ref{app:GLverifiesall} in dimension $d\geq4$ using an approach slightly different from \cite{2016arXiv161202385R}. Let us stress that in \cite{2016arXiv161202385R} the existence of an ergodic infinite-volume Gibbs measure was assumed, but we actually prove that such a measure always exists in Lemma~\ref{lemma:uniquenesslimitGibbsmeasure}. Finally, our last class of examples are models satisfying a certain weak mixing property \eqref{eq:ratioweakmixing} that  implies condition {\bf P3''} without any sprinkling, see Proposition~\ref{prop:upperratioweakmixing}. It contains, for instance, the two-dimensional massive Gaussian free field or the Ising model with an external field.

\subsection{Discrete Gaussian free field}
\label{sec:gff}
\label{sub:GFF}
Consider the graph $G=(\bbZ^d, E_d, \bar{a},\kappa),$ $d\geq3$, equipped with symmetric weights   $\bar{a}:E_d\rightarrow (0,\infty)$  and a killing measure $\bar{\kappa}:\bbZ^d\rightarrow [0,\infty)$, possibly equal to zero. Assume that the Green's function $\bar{g}(x,y)$, $x,y\in \bbZ^d$, associated with the random walk on $\bbZ^d$ with generator given by \eqref{eq:defL_intro} when $a^{\omega}=\bar{a}$, $h=1$, and $\kappa^{\omega}=\bar{\kappa}$, satisfies
\begin{equation}
\label{eq:condongreen}
    \bar{g}(x,y)\leq \, C_G \, |x-y|^{2-d}, \qquad\forall\, x\neq y\in{\bbZ^d},
\end{equation}
for some constant $C_G<\infty$. For instance one can consider constant weights with zero killing measure. Let  $(\phi_x)_{x\in{\bbZ^d}}$ be the Gaussian free field on $G$, i.e.\   the centred Gaussian field, under a probability measure $\bbP^G$, with covariance function
\begin{equation}
\label{defGFF}
    \bbE^G[\phi_x\phi_y]=\bar{g}(x,y), \qquad x,y\in\bbZ^d.
\end{equation}

\begin{prop}
\label{app:GFFverifiesall}
Under \eqref{eq:condongreen}, assume that $(\omega_x)_{x\in{\bbZ^d}}$ has the same law under $\bbP^u$ as $(\phi_x+u)_{x\in{\bbZ^d}}$ under $\bbP^G$ for all $u\in\bbR,$ and that $\omega_e=0$ for all $e\in{E_d}.$ Then for all intervals $I\subset\bbR,$ $\xi_P>1,$ $\chi_P\in{(0,\frac{d-2}2)},$ $\eps_P=1$ and $a_P=d-2,$ there exist constants $C_P,R_P,L_P<\infty$, only depending on $C_G,$ such that $(\bbP^u)_{u\in{I}}$ satisfies {\bf P2} and {\bf P3}.
\end{prop}
\begin{proof}
Condition {\bf P2} is clearly satisfied by definition. In order to prove {\bf P3}, we proceed as in \cite{MR3325312}. Let us fix some $R,L,x_1,x_2,u$ and $\hat{u}$ as in {\bf P3} for some $\xi_P>1$ and some  $\chi_P<\frac{d-2}{2}$.  It follows from the Markov property of the Gaussian free field, see \cite[Lemma~1.2]{MR3053773}, that there exist two independent Gaussian fields $\tilde{\phi}^{(1)}$ and $h^{(1)}$ such that $\phi=\tilde{\phi}^{(1)}+h^{(1)}$, where $\tilde{\phi}^{(1)}$ is independent of $\phi_{|Q(x_1,L^{\xi_P})}$ and $h^{(1)}$ is $\sigma(\phi_x,x\in{Q(x_1,L^{\xi_P})})$-measurable. Moreover, setting
\begin{equation}
\label{app:defBi}
    B^{(1)} \ldef \bigg\{\sup_{Q(x_{2},L^{\xi_P})}|h^{(1)}|\leq \frac{1}2R^{-\chi_P}\bigg\},
\end{equation}
by Proposition~1.4 and Remark~1.5 in \cite{MR3325312} (whose proof also works for the Gaussian free field with non-constant weights), there exist constants $R_P,L_P<\infty,$ only depending on $C_G,$  such that if $L\geq L_P$ and $R\geq R_P$,
\begin{equation}
\label{app:boundonB1}
\begin{split}
    \bbP^G\big((B^{(1)})^c\big)&\leq 2 \, |Q(x_2,L^{\xi_P}+RL)| \, \exp\bigg(-\frac{R^{-2\chi_P}}{8\sup_{\{x:\,\|x\|\geq RL\}} \bar g(0,x)}\bigg)
    \\&\leq 2 \, (L^{\xi_P}+RL)^d \, \exp\big(-(C_G/8)R^{d-2-2\chi_P}L^{d-2}\big)
    \\&\leq\frac12\exp(-(C_G/9)L^{d-2}).
\end{split}
\end{equation}
Let $\tilde{h}^{(1)}$ be a random variable independent of $\phi$ and with the same law  as $h^{(1)},$ $\tilde{B}^{(1)}$ as in \eqref{app:defBi} but with $\tilde{h}^{(1)}$ instead of $h^{(1)},$ and let $B=B^{(1)}\cap\tilde{B}^{(1)}$. Then, for all $x\in{Q(x_2,L^{\xi_P})}$,
\begin{equation*}
    (\tilde{\phi}^{(1)}_x+\tilde{h}^{(1)}_x+{u})\indicator_{B}\geq (\tilde{\phi}^{(1)}_x+h^{(1)}_x-R^{-\chi_P}+{u})\indicator_{B}\geq (\phi_x+\hat{u})\indicator_{B}.
\end{equation*}
Therefore for all increasing events $A\subset[0,\infty)^{Q(x_1,L^{\xi_P})}\times [0,\infty)^{Q(x_2,L^{\xi_P})},$
\begin{equation}
\label{app:proofdecoup}
\begin{split}
    &\bbP^G\big(B,\big((\phi_x+ \hat{u})_{x\in{Q(x_1,L^{\xi_P})}},(\phi_x+ \hat{u})_{x\in{Q(x_2,L^{\xi_P})}}\big)\in{A}\,\big|\,\phi_{|Q(x_1,L^{\xi_P})}\big)
    \\&\leq\bbP^G\big(\big((\phi_x+ u)_{x\in{Q(x_1,L^{\xi_P})}},(\tilde{\phi}^{(1)}_x+\tilde{h}^{(1)}_x+ u)_{x\in{Q(x_2,L^{\xi_P})}}\big)\in{A}\,\big|\,\phi_{|Q(x_1,L^{\xi_P})}\big),
\end{split}
\end{equation}
and similarly for decreasing events when exchanging $\hat{u}$ and $u.$ Since $\tilde{\phi}_x^{(1)}+\tilde{h}_x^{(1)}$ has the same law as $\phi$ and is independent of $\phi_{|Q(x_1,L^{\xi_P})},$ condition {\bf P3} follows easily from \eqref{app:boundonB1} and \eqref{app:proofdecoup}.
\end{proof}

For the next result, we extend the definition of $d^{\omega}$ and $\mu_u$ from \eqref{intro:defT} and \eqref{intro:deftimeconstant} to the case where, instead of a family of weights $(t_e^{\omega})_{e\in{E_d}}$ on the edges, we have a family of weights $(t_x^{\omega})_{x\in{\bbZ^d}}$ on the vertices, simply by considering the minimal length over paths of vertices instead of paths of edges. Recall that $h_*$ denotes the critical parameter for the level sets $\{x\in{\bbZ^d}:\,\phi_x\geq h\},$ $h\in\bbR,$ of the Gaussian free field on $(\bbZ^d, E_d,1,0)$ under $\bbP^G.$

\begin{coro}
\label{app:1stcorGFF}
Let $\bar{a}_{e}=1$ for all $e\in{E_d}$ and $\bar{\kappa}_x=0$ for all $x\in{\bbZ^d}$. Fix a decreasing function $f:\bbR\mapsto[0,\infty)$ such that $\bbE^G[f(\phi_0)]<\infty$. Set $h_f\ldef \inf\{t\in\bbR:\,f(t)=0\}$ with the convention $\inf\varnothing\ldef +\infty$, and take $(t_x^{\omega})_{x\in{\bbZ^d}}$ with the same law under $\bbP^0$ as $(f(\phi_x))_{x\in{\bbZ^d}}$ under $\bbP^G$. Then, for all $x\in\bbZ^d$,
\begin{equation}
\label{app:mu_f>0}
    \mu_0(x)>0\text{ if }h_f>h_* \text{ and }\mu_0(x)=0\text{ if }h_f<h_*.
\end{equation}
Moreover, if $h_f>h_*$, then for all $\delta>4$ there exist positive constants $C_{\text{FPP}}$ and $c_{35}$ such that for all $n\in\bbN$, 
\begin{equation}
\label{app:Tfsmallercn}
    \bbP^0\left(d^\omega(0,nx)\leq C_{\text{FPP}}n\right) \leq \exp\left(-\frac{c_{35}n}{\log(n)^{\delta\indicator_{d=3}}}\right).
\end{equation}
\end{coro}
\begin{proof}
Let us first assume that $h_f<h_*.$ Since $f$ is decreasing, for all $h\in{(h_f,h_*)}$,
\begin{equation*}
    \big\{x\in{\bbZ^d}:\,\phi_x\geq h\big\}\subset\big\{x\in{\bbZ^d}:\,f(\phi_x)=0\big\},
\end{equation*}
and in particular $\big\{x\in{\bbZ^d}:\,f(\phi_x)=0\big\}$ contains at least one infinite connected component. It follows from the Burton-Keane theorem, see e.g.\  \cite[Theorem~12.2]{MR2280297}, that  $\big\{x\in{\bbZ^d}:\,f(\phi_x)=0\big\}$ contains a unique infinite connected component which we denote by $\cC_{\infty}$. Moreover, by the FKG inequality, see the remark above Lemma~1.4 in \cite{MR3053773}, we have for all $x\in\bbZ^d$ and $n\in\bbN$,
\begin{equation*}
    \bbP^G\big(0\longleftrightarrow nx\text{ in }\big\{x\in{\bbZ^d}:\,f(\phi_x)=0\big\}\big)\geq\bbP^G\big(0\in{\cC_{\infty}}\big)^2>0.
\end{equation*}
The claim $\mu_0(x)=0$ then follows directly from Proposition~\ref{Prop:condformu=0}. 

Let us now assume that $h_f>h_*,$ and take $I=(-1,h_f-h_*)$. Let  $\omega$ be as in Proposition~\ref{app:GFFverifiesall}, $\mathcal{N}=\{0\},$ and  $\mathbf{t}(\omega_V^1,\omega_V^2,\omega_E)=\frac12(f(\omega_V^1)+f(\omega_V^2))$ satisfying \eqref{eq:temonotone}. Noting that for all $u\in{I}$ ,
\begin{equation*}
    \big\{x\in{\bbZ^d}:\,f(\phi_x+u)=0\big\}\subset\big\{x\in{\bbZ^d}:\,\phi_x\geq h_f-u\big\},
\end{equation*} 
condition \eqref{intro:BLnotconnectedtoB2L} is fulfilled  (for $t_x^{\omega}=f(\omega_x)$ instead of $t_e^{\omega}$) by  \cite[Theorem~1.1]{DuGoRoSe} since $h_f-u>h_*$ for all $u\in{I}$. Moreover,  condition \eqref{eq:boundonmement} with $u=0$ is 
satisfied by our assumption on $f$, and  condition {\bf P1} is proved above Lemma~1.5 in \cite{MR3053773}. The result now follows from \eqref{eq:defgp}, Theorem \ref{The:main} and Remark~\ref{rk:fppvertices}-(i) together with Proposition~\ref{app:GFFverifiesall}.
\end{proof}

\begin{remark}
\label{rk:optidecay}
When $d\geq 4,$ the probability that $d^\om(0,nx)$ is smaller than $cn$ decays exponentially fast by \eqref{app:Tfsmallercn}, and one can easily use the FKG inequality to show that this decay is optimal. However in dimension $d=3$, there is an additional logarithmic correction in \eqref{app:Tfsmallercn}, contrary to the case of independent percolation, see \cite[Proposition~5.8]{MR876084}. This is not an artefact of our proof, since one can easily adapt the proof of \cite[Theorem~3.1]{GRS21} to show that when $f(x)=\indicator_{\{x< h\}}$ for $h>h_*$,
\begin{equation*}
    \bbP^u(d^\om(0,nx)\leq cn)\geq \bbP^u(0\longleftrightarrow nx\text{ in }E^{\geq h})\geq\exp\Big(-\frac{c'n}{\log(n)}\Big).
\end{equation*}
Similarly as in \cite{GRS21}, we believe that the optimal logarithmic correction is indeed $\log(n),$ and not $\log(n)^{\delta}$ for some $\delta>4$ as in \eqref{app:Tfsmallercn}.
\end{remark}

Note that Theorem~\ref{the:mainGFF} is a particular case of Corollary~\ref{app:1stcorGFF} for the choice $f(s)=\indicator_{\{s<h\}}$ since $h_f=h$ for this choice of $f$. Let us now explain how one can use Proposition~\ref{app:GFFverifiesall} to obtain the bound \eqref{intro:green_bound} for the choice of weights from below \eqref{intro:green_bound}. It is for instance enough to verify that the conditions of Theorem~\ref{thm:GKdecay_killing} are satisfied, when $\omega$ is as in Corollary~\ref{app:1stcorGFF}, $\mathcal{N}=\{0\},$ $\mathbf{a}(\omega_V^1,\omega_V^2,\omega_E)=e^{\gamma(\omega_V^1+\omega_V^2)}$,  and $\boldsymbol{\kappa}(\omega_V)=e^{\gamma\omega_V}$, respectively.  The monotonicity condition   \eqref{eq:amonotone} is clearly fulfilled, while Assumption~\ref{ass:moment_killing}-(i) holds since $(\boldsymbol{\kappa}(\omega_V^1)\wedge\boldsymbol{\kappa}(\omega_V^2))/\mathbf{a}(\omega_V^1,\omega_V^2,\omega_E)=e^{-\gamma\omega_V^1}\wedge e^{-\gamma\omega_V^2}$ is decreasing.   Condition {\bf P1} holds by \cite[Lemma 1.5]{MR3053773}, conditions {\bf P2} and {\bf P3} follow from Proposition~\ref{app:GFFverifiesall}, and conditions {\bf S1} and {\bf S2} trivially hold since $\mathcal{C}_{\infty}=\bbZ^d$. Finally, Assumption~\ref{ass:moment_killing}-(ii)  follows from the fact that $e^{\gamma\phi_0}\in{L^p}$ for any $p>0$ and the symmetry of the Gaussian free field, and we can conclude. Alternatively one could also prove that the time constant $\mu_\kappa^\om$ is positive using Theorem~\ref{the:mainGFF}, as well as the bounds $e^{-\gamma\phi_x}\geq e^{-\gamma h}$ if $\phi_x\leq h$ and $e^{-\gamma\phi_x}\geq 0$ otherwise for $h$ large enough,  and conclude by Corollary~\ref{cor:optidecay}.

\begin{remark}
\label{rk:finalGFF}
(i)  In Corollary~\ref{cor:optidecay}, the weights $a^{\omega}(x,y),$ $x\sim y,$ depend on $\phi_x,\phi_y$ via the function $f(t,s)=e^{\gamma (t+s)}.$ This choice is however quite arbitrary, and one could in fact prove the same result for any symmetric monotone and strictly positive function $f$ in view of Theorem~\ref{thm:GKdecay_killing} under the integrability condition Assumption~\ref{ass:moment_killing}-(ii). One can even allow $f$ to be equal to $0$ as long as conditions {\bf S1} and {\bf S2} are satisfied. For instance, one can take $f(t,s)=f'(t,s)\indicator_{\{t\geq h,s\geq h\}}$ for $h< h_*$ and some symmetric, monotone and strictly positive  $f'$, and conditions {\bf S1} and {\bf S2} are then proved in \cite{DRS14, DuGoRoSe}. 
  
  (ii) The advantage of considering general symmetric weights $(\bar{a}_{e})_{e\in{E^d}}$ under condition \eqref{eq:condongreen} instead of unit weights in Proposition~\ref{app:GFFverifiesall} is that it allows us to treat examples  similar to Corollary~\ref{cor:optidecay} but for the Gaussian free field with random conductances, as studied for instance in \cite{MR4299133}. Indeed, assume that the conductances $(\bar{a}_{x,y})_{x,y\in{\bbZ^d}}$ are chosen at random under some probability $\bbQ$ under which they are stationary and ergodic with respect to shifts and almost surely satisfy \eqref{eq:condongreen} for some non-random constant $C_G$.  Then conditions {\bf P2} and {\bf P3} hold $\bbQ$-a.s.\ by Proposition~\ref{app:GFFverifiesall}. Moreover, the constants appearing in condition {\bf P3} only depend on $C_G,$ in particular not on $\bar{a},$ and thus conditions {\bf P2} and {\bf P3} still hold after integration with respect to $\bbQ$. Condition {\bf P1} also holds by assumption, so one can  use Theorems~\ref{thm:hkeGauss} and \ref{thm:GKdecay_killing} to prove results for examples similar to Corollary~\ref{cor:optidecay} when $\phi$ is the (annealed) Gaussian free field on the graph with the random conductances $\bar{a}$ under $\bbQ$. For instance, when the conductances $(\bar{a}_{e})_{e\in{E^d}}$ are uniformly elliptic,   \eqref{eq:condongreen} follows from the heat kernel bounds in \cite{De99}. Note that for an appropriate choice of the random conductances $(\bar{a}_e)_{e\in{E_d}}$ this corresponds to a Ginzburg-Landau $\nabla\phi$ interface model with non-convex potentials as explained in \cite{MR2322690}, which are not already covered by the setting of the following Section~\ref{sec:ginzburglandau}. 
  
  (iii) The Gaussian free field is actually only one example of a class $\mathcal{F}_{\alpha},$ $\alpha>0,$ of Gaussian fields all satisfying condition {\bf P3} recently studied in \cite[Section~2]{Mui22}. More precisely one can respectively replace the fields $h^{(1)}$ and $\tilde{\phi}^{(1)}$ in the proof of Proposition~\ref{app:GFFverifiesall} by the fields $g_{R}$ and $f_{R}$ from \cite[Proposition~3.10]{Mui22} with $R=L^{\xi_P},$ and find a bound similar to \eqref{app:boundonB1} by proceeding similarly as in the proof of \cite[(3.12)]{Mui22} in \cite[Section~3.6]{Mui22}. This shows that any discrete field in $\mathcal{F}_{\alpha}$ satisfies {\bf P3} with $\eps_P=1$ and any $a_P=\alpha.$ The Gaussian free field belongs to $\mathcal{F}_{(d-2)/2}$ by \cite[Lemma~2.9]{Mui22}, and in fact also to $\mathcal{F}_{d-2}$ by the recent article \cite{Sch22}, which gives another proof of Proposition~\ref{app:GFFverifiesall}. The class $\mathcal{F}_{\alpha}$ also contains other discrete Gaussian fields such as the membrane model in dimension $d\geq5$ for $\alpha=(d-4)/2,$ see \cite[Section~3]{Mui22} for details.
\end{remark}

\subsection{Ginzburg-Landau $\nabla\phi$ Interface Model}
\label{sec:ginzburglandau}
The Ginzburg-Landau $\nabla\phi$ interface model is a well established  model for an interface separating two pure thermodynamical phases. The model  is a direct generalization of the discrete Gaussian free field. We refer to the monograph \cite{MR2228384} for an introduction. The interface is  described by a random field of height variables  $\phi=\{\phi(x): x\in \bbZ^d \}$ sampled from a Gibbs measure formally given by  $Z^{-1} \exp(-H(\varphi)) \, \prod_{x\in \bbZ^d} d \varphi(x)$ with formal Hamiltonian $H(\varphi)= \sum_{e\in E_d} V(\nabla \varphi(e))$ and potential  $V\in  C^2(\bbR; \bbR_+)$, which we suppose to be even and strictly convex.  For this model, decoupling inequalities similar to {\bf P3'} have been obtained   \cite[Theorem~2.1]{2016arXiv161202385R}, but the Gibbs measure considered therein has the disadvantage of not being clearly shift-invariant and ergodic with respect to lattice shifts, see the assumption (4.23) in \cite{2016arXiv161202385R}. In particular, condition {\bf P1} might not hold for this Gibbs measure, and thus we cannot apply Corollary~ \ref{cor:shape}, or Theorems~\ref{thm:hkeGauss} and \ref{thm:GKdecay_killing}. To avoid this problem, we now adapt the arguments from \cite{MR2228384} to obtain decoupling inequalities for a Gibbs measure which is shift-invariant and ergodic with respect to lattice shifts, see Lemma~\ref{lemma:uniquenesslimitGibbsmeasure} and Proposition~\ref{prop:decoupGL} below. This is actually useful whenever one intends to use the conditions {\bf P1}--{\bf P3} from \cite{DRS14, Sa17} for interface models as explained in Remark~ \ref{rk:P1P3GL} below.

Throughout this section, we consider the graph $\bbZ^d$ for $d\geq3,$ and fix a symmetric potential $V\in{C^2(\bbR)}$ such that 
\begin{equation}
\label{Velliptic}
    C_-\leq V''(t)\leq C_+ \quad \text{for all }t\in{\bbR},
\end{equation}
for some constants $0<C_-\leq C_+<\infty$. If $\Lambda\subset\bbZ^d$ is a finite subset of $\bbZ^d,$ we denote by $\bar{\Lambda}=\{x\in{\bbZ^d}:x\text{ has a neighbor in }\Lambda\},$ and by $\Lambda_*$ the set of edges between vertices in $\bar{\Lambda}.$ Let us denote by $\Gamma_n^d=(\bbZ/n\bbZ)^d$ the $d$-dimensional torus, and if $\Lambda=\Gamma_n^d,$ we take $\bar{\Lambda}=\Lambda$ and $\Lambda_*$ the set of edges of $\Lambda$. If $\Lambda\subset\bbZ^d$ is finite, or $\Lambda=\Gamma_n^d$ for some $n\in\bbN$, for any fixed  $m\geq0$ and $\xi\in{\bbR^{\bbZ^d}}$  we define the Hamiltonian
\begin{equation*}
    H_{\Lambda,m}^{\xi}(\phi)=\frac12\sum_{\substack{x,y\in{\bar{\Lambda}}:\\(x,y)\in{\Lambda_*}}}V(\phi_x-\phi_y)+\frac{m^2}{2}\sum_{x\in{\Lambda}}\phi_x^2 \quad \text{ for all }\phi\in{\bbR^{\bar{\Lambda}}}\text{ with }\phi_{|\bar{\Lambda}\setminus\Lambda}=\xi_{|\bar{\Lambda}\setminus\Lambda}.
\end{equation*}
Note that $H_{\Lambda,m}^{\xi}(\phi)$ does not depend on the boundary condition $\xi$ when $\Lambda$ is the torus $\Gamma_n^d.$ The associated Gibbs probability measure on $(\bbR^{\Lambda},\mathcal{B}(\bbR^{\Lambda}))$ is given by
\begin{equation*}
    \mu_{\Lambda,m}^{\xi}(\mathrm{d}{\phi})=\frac1{Z_\Lambda^\xi}\exp\big(-H_{\Lambda}^{\xi}(\phi)\big)\prod_{x\in{\Lambda}}\mathrm{d}\phi_x,
\end{equation*}
where $Z_\Lambda^\xi$ is a suitable normalizing constant. This is well-defined (i.e.\ $Z_\Lambda^{\xi}<\infty$) as long as either $\Lambda$ is a finite subset of $\bbZ^d$ and $m\geq0,$ or $\Lambda=\Gamma_n^d$ and $m>0,$ which we will assume from now on. Let us denote by $\mu_{\Lambda,m}^{\xi,G}$ the associated Gaussian free field, that is the Gibbs measure associated to the choice $V(t)=C_-t^2,$ with $C_-$ as in \eqref{Velliptic}. It satisfies the following exponential Brascamp-Lieb inequality: for all $\nu\in{\bbR^{\Lambda}}$,
\begin{equation}
\label{eq:BrascampLieb}
    \bbE_{\mu_{\Lambda,m}^{\xi}}\left[\exp\left(\Big| \big\langle\nu,\phi-\bbE_{\mu_{\Lambda,m}^{\xi}}(\phi)\big\rangle_{l^2(\Lambda)}\Big|\right)\right]\leq 2\exp\Big(\frac12\mathrm{Var}_{\mu_{\Lambda,m}^{\xi,G}}\left(\langle\nu,\phi\rangle_{l^2(\Lambda)}\right)\Big).
\end{equation}
When $m=0,$ \eqref{eq:BrascampLieb} was proved in  \cite[Lemma~2.9]{MR1759509}. The case $m>0$ can be proven as follows. Since  $\Lambda$ can be identified  with $\Lambda\times\{0\},$ $H_{\Lambda,m}^{\xi}$ corresponds to the massless Gibbs probability measure for the graph $\Lambda\times\{0,1\},$ with potential $V$ between $(x,i)$ and $(y,i),$ $x\sim y\in{\Lambda}$ and $i\in{\{0,1\}},$ and potential $m^2t^2/2$ between $(x,0)$ and $(x,1),$ $x\in{\Lambda},$ and with $0$ boundary condition on $\Lambda\times\{1\},$ as defined in \cite[equation~(2.3)]{MR1759509}. Therefore  we can adapt the proof of \cite[Lemma 2.9]{MR1759509} to obtain \eqref{eq:BrascampLieb} by bounding the second derivative of the potential by $C_-$ for edges in $\Lambda\times\{0\}$ or $\Lambda\times\{1\},$ and by $m^2$ for edges between $\Lambda\times\{0\}$ and $\Lambda\times\{1\}$. This extension \eqref{eq:BrascampLieb} of the classical Brascamp-Lieb inequality for exponential moments to the massive model has already been noted in item b) of the theorem on page~56 of \cite{MR1461951} (with a typo $m$ instead of $m^2$ on the right-hand side).

As we now explain, the Brascamp-Lieb inequality \eqref{eq:BrascampLieb} classically yields the existence of a $\phi$-Gibbs measure on ${\bbZ^d}$, i.e.\ a probability measure on $\bbR^{\bbZ^d}$ satisfying the DLR equations, see for instance  \cite[Definition~2.1]{MR2228384}, which is unique under the assumption of invariance and ergodicity, introduced in condition {\bf P1}.

\begin{lemma}
\label{lemma:uniquenesslimitGibbsmeasure}
 The weak limit
\begin{equation}
\label{eq:infinitevolumeGibbsmeasure}
    \mu^0:=\lim_{m\rightarrow0}\lim_{n\rightarrow\infty}\mu_{\Gamma_{n}^d,m}^{0}
\end{equation}
exists, and $\mu^0$ is the unique $\phi$-Gibbs measure on $\bbZ^d$ which is invariant and ergodic with respect to lattice shifts, and under which $\phi_0$ is square integrable and has zero mean.
\end{lemma}
\begin{proof}
One can use a tightness argument based on the Brascamp-Lieb inequality \eqref{eq:BrascampLieb} to show that there exist a sequence $(n_k)_{k\geq 0}$ increasing to infinity, and a sequence $(m_l)_{l\geq 0}$ decreasing to zero, such that the weak limit
\begin{equation}
\label{eq:infinitevolumeGibbsmeasureseq}
    \mu^0:=\lim_{l\rightarrow\infty}\lim_{k\rightarrow\infty}\mu_{\Gamma_{n_k}^d,m_l}^{0}\text{ exists.}
\end{equation}
This was first indicated in \cite[Remark~4.6]{MR2228384}, and we refer to the proof of \cite[Theorem~5.6]{AT21} for a more detailed analysis. It is easy to see that $\mu^0$ is a $\phi$-Gibbs measure on $\bbZ^d$ and that it is invariant with respect to lattice shifts. Moreover, note that for each $m>0$, $\text{Var}_{\mu_{\Gamma_n^d,m}^{0,G}}(\phi_0)$ is the Green's function associated to the random walk on $\Gamma_n^d$ with weights $C_-$ and killing measure $m^2$, and is thus bounded uniformly in $n$ by the reciprocal of the probability that the random walk on $\bbZ^d$ with weights $C_-$ and killing measure $m^2$ is immediately killed, which is finite. Therefore, since $\phi_0$ is centered under $\mu^0_{\Gamma_n^d,m},$ one can use \eqref{eq:BrascampLieb} to prove that, for each $l\in{\bbN}$ and $\nu>0,$ $\exp(\nu\phi_0)$ is bounded in $L^2$ under $(\mu^0_{\Gamma_n^d,m_l})_{n\in{\bbN}}.$ In particular, under some other probability space, a sequence of random variables with the same law as $\exp(\nu\phi_0)$ under $(\mu^0_{\Gamma_{n_k}^d,m_l})_{k\in{\bbN}}$ converges in $L^1$ to a random variable with the same law as $\exp(\nu\phi_0)$ under $\mu^0_{m_l},$ where $\mu^0_{m_l}$ is the limit of $\mu_{\Gamma_{n_k}^d,m_l}^{0}$ as $k\rightarrow\infty.$ Taking the limit, we thus deduce that \eqref{eq:BrascampLieb} for $\nu\delta_0$ still holds for $\mu^0_{m_l}.$ Since $\text{Var}_{\mu_{m_l}^{0,G}}(\phi_0)$ is uniformly bounded in $l$ by the Green's function associated with the random walk on $\bbZ^d$ with weights $C_-$ and zero killing measure, we can similarly deduce that $\phi_0$ has finite exponential moments under $\mu^0.$ In particular, $\phi_0$ is square integrable and has mean zero under $\mu^0$ by symmetry. 

In order to prove uniqueness and ergodicity, let us first prove the uniqueness and existence of ergodic $\nabla\phi$-Gibbs measure on $\bbZ^d,$ see e.g.\ \cite[Definition~2.2]{MR2228384}.
We proceed similarly to \cite{MR1463032} where a weaker (and more common) notion of an  ergodic measure $\mu$ than the one in condition {\bf P1} is used.  More precisely,  that notion of ergodicity requires that $\mu(A)\in{\{0,1\}}$ for any $A\in{\cF}$ satisfying $\tau_x(A)=A$ for \emph{all} $x\in{\bbZ^d}$.  Nevertheless, as we now explain, the arguments in  \cite{MR1463032} can  be adapted to our setting. For each $x\in{\bbZ^d}$, we say that a probability measure $\mu$ is $x$-directionally ergodic if $\mu(A)\in{\{0,1\}}$ for any $A\in{\cF}$ such that $\tau_x(A)=A,$ and that $\mu$ is ergodic if $\mu$ is $x$-directionally ergodic for all $x\in{\bbZ^d}$ (which corresponds to the definition of ergodicity  in {\bf P1}). Then, by  \cite[Theorem~3.1]{MR1463032} there exists at most one square integrable $\nabla\phi$-Gibbs measure $\mu$ on $E_d$ with zero mean and such that $\mu(A)\in{\{0,1\}}$ for any $A\in{\cF}$ satisfying $\tau_x(A)=A$ for all $x\in{\bbZ^d}.$ Therefore, for each $x\in{\bbZ^d},$ there is also at most one square integrable $x$-directionally ergodic $\nabla\phi$-Gibbs measure on $E_d$ with zero mean.

We now turn to the existence of an ergodic $\nabla\phi$-Gibbs measure. For each $x\in{\bbZ^d}$,   \cite[Proposition~14.9]{MR2807681} remains valid (see the comment below its proof) when the $\sigma$-algebra of invariant sets is replaced by the larger set of events that are only $\tau_x$-invariant. In fact, note that $\{\tau_{kx},\,k\in{\bbN}\}$ is an infinite subgroup of $\{\tau_{y},\,y\in{\bbZ^d}\}$, and any $\tau_x$-invariant event is $\tau_{kx}$-invariant for all $k\in{\bbN}$. In particular,  \cite[Theorems~14.15 and 14.17]{MR2807681} remain valid when replacing ergodic by $x$-directionally ergodic, and thus an $x$-directionally ergodic $\nabla\phi$-Gibbs measure exists as an extremal element of the convex set of $\tau_x$-invariant $\nabla\phi$-Gibbs measures. Moreover, for each $x\in{\bbZ^d}$, one can use an $x$-directionally ergodic decomposition of $\tau_x$-invariant Gibbs measure similarly as below equation (3.5) in \cite{MR1463032} to prove that there exists a square integrable $x$-directionally ergodic  $\nabla\phi$-Gibbs measure with zero mean, and we proved that such measures are unique.

Following the proof of  \cite[Theorem~9.10]{MR2228384}, which still works since $x$-directional ergodicity is stronger than the notion of ergodicity used in \cite{MR2228384}, we deduce that there exists at most one square integrable $x$-directionally ergodic $\phi$-Gibbs measure on $\bbZ^d$ with zero mean, and we denote it by $\mu_x$ when it exists. Thus, in view of   \cite[Theorem~4.13]{MR2228384}, one can follow the same arguments as below equation (3.5) in \cite{MR1463032} but for $x$-directional ergodicity and $\phi$-Gibbs measures, to show that all square-integrable $\tau_x$-invariant $\phi$-Gibbs measures with zero mean can be uniquely decomposed via square integrable $x$-directionally ergodic  $\phi$-Gibbs measure with zero mean, and thus either there are no such $\tau_x$-invariant measures, or they are all equal to $\mu_x$. Since the limiting measure
 $\mu^0$ from \eqref{eq:infinitevolumeGibbsmeasure} is a square-integrable $\tau_x$-invariant $\phi$-Gibbs measure with zero mean, we conclude that $\mu_x=\mu^0$ for all $x\in{\bbZ^d}$. The limit in \eqref{eq:infinitevolumeGibbsmeasureseq} does not depend on the choice of the subsequences $(n_k)$ and $(m_l)$ by uniqueness, and we obtain \eqref{eq:infinitevolumeGibbsmeasure} by a sub-subsequence argument.
\end{proof}


We now turn to the proof of the decoupling inequality, which follows from a modification of the ideas from \cite{2016arXiv161202385R}. For $S_1,S_2\subset\bbZ^d$ disjoint we define 
\begin{equation}
\label{eq:xiS1S2}
    \Xi(S_1,S_2)=\sup_{x\in{S_1}}\frac{\sup_{y\in{S_2}}g(x,y)}{\inf_{y\in{S_2}}g(x,y)}P_x(H_{S_2}<\infty).
\end{equation}
Here, $g(\cdot,\cdot)$ denotes the Green's function associated with the simple random walk $((X_t)_{t\geq 0}, (P_x)_{x\in \bbZ^d})$ on $\bbZ^d$, and  $H_{S}\ldef\inf\{t\geq0:\,X_t\in{S}\}$  the first hitting time of  any $S\subset\bbZ^d$ with the convention $\inf\varnothing\ldef \infty$. 

\begin{prop}
\label{prop:decoupGL}
There exist constants $c_{36},c_{37}>0$ such that for any disjoint $S_1,S_2\subset\bbZ^d$, any increasing $f_i:\bbZ^d\rightarrow[0,1]$ supported on $S_i$, $i\in{\{1,2\}}$, and all $\eps\in{(0,1]}$,
\begin{equation}
\label{eq:decoupineqGL}
    \bbE_{\mu^0}\left[f_1f_2\right]\leq \bbE_{\mu^0}\left[f_1(\phi+\eps)\right]\bbE_{\mu^0}\left[f_2(\phi+\eps)\right]+c_{36} |S_2| \exp\left(-c_{37}\eps^2\Xi(S_1,S_2)^{-2}\right).
\end{equation}
\end{prop}
\begin{proof}
For each $\Lambda\subset\Gamma_n^d$, $x\in{\Lambda},$ $\psi\in{\bbR^{\Lambda}},$ $\xi\in{\bbR^{\Gamma_n^d\setminus\Lambda}}$ and $m>0$ we denote by $P_{x,\psi,m}^{\Lambda,\xi}$ the transition kernel associated with the time-inhomogenous Markov process on $\Gamma_n^d$ starting at $x$ at time $t=0$ with generator given by \eqref{eq:defL_intro} but with time-dependent jump rates $a_t^{\omega}(x,y)=V''(\phi_x(t)-\phi_y(t)),$ $h=1$ and $\kappa^{\omega}(x)=m^2$ for all $x\in\Gamma_n^d$, where 
$\phi_t=(\phi_t(x), x\in \Gamma_n^d)$ is the solution to the Langevin equation
\begin{align*}
\begin{cases}
\phi_t(x)=\psi(x)-\int_0^t  m^2 \phi_s(x)+\sum_{y:|x-y|=1}V'(\phi_s(x)-\phi_s(y)) \, ds + \sqrt{2} \, w_t(x),\!\!\!   & x\in \Lambda ,\\
\phi_t(x)= \xi(x),  & x\not\in \Lambda,
\end{cases}
\end{align*}
with $(w_t(x))_{t\geq 0}$, $x\in \Lambda$, being independent Brownian motions. 
This transition kernel then corresponds to the Markov process $X$ appearing in the Helffer-Sj\"ostrand representation with mass $m$ (see e.g.\  \cite[Theorem~4.2]{MR2228384}), which can easily be deduced from the discussion below \eqref{eq:BrascampLieb}. Let us also define a field $\tilde{\phi},$ independent and with the same law as $\phi$ under $\mu_{\Gamma_n^d,m}^0$, and
\begin{equation}
\label{eq:proofdecoupginzurg}
    \Sigma_{\Gamma_n^d}^m(S_1,S_2)\ldef \sup_{x\in{S_1}}\sup_{\xi\in{\bbR^{S_2}}}\sup_{\psi\in{\bbR^{\Gamma_n^d\setminus S_2}}}\frac{P_{x,\psi,m}^{\Gamma_n^d\setminus S_2,\xi}(H_{S_2}<\infty)}{1-P_{x,\psi,m}^{\Gamma_n^d\setminus S_2,\xi}(H_{S_2}<\infty)},
\end{equation}
where we identify $\Gamma_n^d$ with $\Lambda_n:=\{\lfloor(n-1)/2\rfloor,\dots,\lceil(n-1)/2\rceil\}^d\subset\bbZ^d$ and assume from now on that $n$ is large enough so that $S_1,S_2\subset\Gamma_n^d$. Then one can easily see that the statement of  \cite[Theorem~2.1]{2016arXiv161202385R} still holds, after some straightforward adjustments in the proof, when the function $\indicator_{A^h}$ with $A^h$ only depending on $\{ \indicator_{\{\phi_x\geq h\}}, x\in{S_1}\}$ is replaced by $f_1(\phi-h)$,  the set $\Lambda \subset \bbZ^d$ is replaced  by the torus $\Gamma_n^d$, and  a mass $m>0$ is added. Thus, conditionally on $(\phi_x,\tilde{\phi}_x)_{x\in{S_2}}$, on the event
\begin{equation*}
    G_{\Gamma_n^d,m,M}:=\bigcap_{\xi\in{\bbQ^{S_2}}}\bigcap_{(x,\psi)\in{S_1\times\bbQ^{\Gamma_n^d\setminus S_2}}}\!\bigg\{E^{\Gamma_n^d\setminus S_2,\xi}_{x,\psi,m}\left[\phi_{H_{S_2}}-\tilde{\phi}_{H_{S_2}}\,|\,H_{S_2}<\infty\right]\leq M \bigg\},
\end{equation*}
with $M\leq \eps\big(1+(\Sigma_{\Gamma_n^d}^m(S_1,S_2))^{-1}\big)$, we have
\begin{equation}
\label{eq:proofdecoupginzurg2}
    E_{\mu_{\Gamma_n^d,m}}^0\big[f_1(\phi)\,\big| (\phi_x)_{x\in{S_2}}\big]\leq E_{\mu_{\Gamma_n^d,m}}^0\big[f_1(\tilde{\phi}+\eps)\,\big|(\tilde{\phi}_x)_{x\in{S_2}}\big].
\end{equation}
 Using \eqref{eq:proofdecoupginzurg2}, the inequality
\begin{equation*}
\bbE_{ \mu^0_{\Gamma_n^d,m}}[f_1f_2]\leq \bbE_{ \mu^0_{\Gamma_n^d,m}}\Big[\bbE_{ \mu^0_{\Gamma_n^d,m}}\big[f_1(\phi)\,\big|\,(\phi_x)_{x\in{S_2}}\big]f_2(\phi)\indicator_{G_{\Gamma_n^d,m,M}}\Big]+\mu^0_{\Gamma_n^d,m}\left(G_{\Gamma_n^d,m,M}^c\right),
\end{equation*}
the independence of $\tilde{\phi}$ and $\phi,$ and \eqref{eq:infinitevolumeGibbsmeasure}, we deduce that in order to obtain \eqref{eq:decoupineqGL} it suffices to show that 
\begin{equation}
\label{eq:proofdecoupginzurg3}
    \liminf_{m\rightarrow0}\liminf_{n\rightarrow\infty}\mu^0_{\Gamma_n^d,m}\Big(G_{\Gamma_n^d,m,\eps(1+(\Sigma_{\Gamma_n^d}^m(S_1,S_2))^{-1})}^c\Big) \leq  c  \, |S_2| \, \exp\left(-c'\eps^2\Xi(S_1,S_2)^{-2}\right).
\end{equation}
First,  for all $\nu>0$ and $M>0$, by a union bound, Markov's inequality and \eqref{eq:BrascampLieb},
\begin{align*}
 &   \mu^0_{\Gamma_n^d,m}\big(G_{\Gamma_n^d,m,M}^c\big)
    \leq 2\mu^0_{\Gamma_n^d,m}\Big(\exists\, x\in{S_2}:\,|\phi_x|>\frac{M}2\Big) \\
  & \mspace{36mu}  \leq  2\sum_{x\in{S_2}}\bbE_{\mu^{0}_{\Gamma_n^d,m}}\Big[\exp\big(\nu |\phi_x| \big)\Big] \, \exp(-\nu M/2)
 \leq 4|S_2|\exp\left(\frac{\nu^2g(0,0)}{2C_-}-\frac{\nu M}{2}\right).
\end{align*}
Taking $\nu=cM$ for a small enough constant $c$ we thus obtain for all $M>0$,
\begin{equation}
    \label{eq:proofdecoupginzurg4}
    \mu^0_{\Gamma_n^d,m}\big(G_{\Gamma_n^d,m,M}^c\big)\leq 2|S_2|\exp(-c'M^2).
\end{equation}
Next we bound $\Sigma_{\Gamma_n^d}^m(S_1,S_2)$. Define $\partial\Gamma_n^d$ the internal boundary of $\Gamma_n^d$ seen as the subset $\Lambda_n$ of $\bbZ^d$. Then
\begin{equation}
\label{eq:proofdecoupginzurg5}
    P_{x,\psi,m}^{\Gamma_n^d\setminus S_2,\xi}(H_{S_2}<\infty)\leq P_{x,\psi,m}^{\Gamma_n^d\setminus S_2,\xi}(H_{S_2}<H_{\partial\Gamma_n^d})+P_{x,\psi,m}^{\Gamma_n^d\setminus S_2,\xi}(H_{\partial\Gamma_n^d}<\infty).
\end{equation}
Since before time $H_{\partial\Gamma_n^d}$ the Markov chain on $\Gamma_n^d$ has the same law as the Markov chain on $\Lambda_n$ before $H_{\partial\Lambda_n},$ it follows from in \cite[Corollary~3.2]{2016arXiv161202385R} that there exists a constant $c=c(d)$ such that, uniformly in $m>0,$ $x\in{S_1},$ $n\in{\bbN},$ $\psi\in{\bbR^{\Gamma_n^d\setminus S_2}}$  and $\xi\in{\bbR^{S_2}}$, 
\begin{equation}
\label{eq:proofdecoupginzurg6}
    P_{x,\psi,m}^{\Gamma_n^d\setminus S_2,\xi}(H_{S_2}<H_{\partial\Gamma_n^d})\leq c \, \Xi(S_1,S_2).
\end{equation}
Moreover, when started at any $x\in{S_1}$, if $H_{\partial\Gamma_n^d}<\infty$, the random walk $X$ must have survived for at least $n/4$ steps for $n$ large enough. Hence, by \eqref{Velliptic}, for each $m>0$ and uniformly in $x,$ $n,$ $\psi$ and $\xi,$
\begin{equation}
\label{eq:proofdecoupginzurg7}
    P_{x,\psi,m}^{\Gamma_n^d\setminus S_2,\xi}(H_{\partial\Gamma_n^d}<\infty)\leq \bigg(1-\frac{m}{m+dC_+}\bigg)^{\! n/4}\tend{n}{\infty}0.
\end{equation}
Taking $M= \eps\big(1+(\Sigma_{\Gamma_n^d}^m(S_1,S_2))^{-1}\big)$, we combine  \eqref{eq:proofdecoupginzurg5}, \eqref{eq:proofdecoupginzurg6} and \eqref{eq:proofdecoupginzurg7} to obtain that, uniformly in $m>0$,
\begin{equation*}
    \limsup_{n\rightarrow\infty}\Sigma_{\Gamma_n^d}^{m}(S_1,S_2)\leq c \, \Xi(S_1,S_2),
\end{equation*}
provided $\Xi(S_1,S_2)$ is small enough so that, $\lim\limits_{n\rightarrow\infty}P_{x,\psi,m}^{\Gamma_n^d\setminus S_2,\xi}(H_{S_2}<\infty)\leq 1/2$ uniformly in $x$, $\psi,$ $\xi$ and $m.$ Note that $\Xi(S_1,S_2)$ can be assumed to be small without loss of generality, up to changing the constants $c_{36}$ and $c_{37}$ in \eqref{eq:decoupineqGL}. Combining this with \eqref{eq:infinitevolumeGibbsmeasure} and \eqref{eq:proofdecoupginzurg4}, \eqref{eq:proofdecoupginzurg3} follows readily.
\end{proof}

Note that the strategy used in the proof of Proposition~\ref{prop:decoupGL} could also be applied if, instead of defining $\mu^0$ as in \eqref{eq:infinitevolumeGibbsmeasure}, we simply take the limit as $n\rightarrow\infty$ of $\mu^0_{\Lambda_n,0}$ for a sequence $(\Lambda_n)_{n\in{\bbN}}$ of finite sets increasing to $\bbZ^d$. In this context Proposition~\ref{prop:decoupGL} can be seen as a simpler and stronger version of  \cite[Theorem~2.1]{2016arXiv161202385R}. In order to deduce condition {\bf P3'} from Proposition \ref{prop:decoupGL}, we now only need to bound $\Xi(S_1,S_2)$ appropriately which requires $d\geq 4$. 

\begin{prop}
\label{app:GLverifiesall}
Suppose $d\geq4$,  \eqref{Velliptic} holds, and  $(\omega_x)_{x\in{\bbZ^d}}$ has the same law under $\bbP^u$ as $(\phi_x+u)_{x\in{\bbZ^d}}$ under $\mu^0$ for all $u\in\bbR$, and take $\omega_e=0$ for all $e\in{E_d}$. Then for all open intervals $I\subset\bbR$ and $\chi_P\in{(0,d-2]},$ there exist constants $\xi_P>1,$ $C_P,R_P,L_P<\infty,$ such that $(\bbP^u)_{u\in{I}}$ satisfies {\bf P1}, {\bf P2} and {\bf P3'}.
\end{prop}
\begin{proof}
Condition {\bf P1} follows from Lemma~\ref{lemma:uniquenesslimitGibbsmeasure} and condition {\bf P2}  clearly holds by definition. Fix an open interval $I\subset\bbR,$ $C_P=R_P=1,$ $\chi_P\in{(0,d-2]},$ and $S_1,$ $S_2$ as in {\bf P3'}, for constants $\xi_P>1$ and $L_P<\infty$ to be determined. Let us first bound the quantity $\Xi(S_1,S_2)$ from \eqref{eq:xiS1S2} . First note that since $c|x-y|^{2-d}\leq g(x,y)\leq c'|x-y|^{2-d}$ for all $x\neq y$ in $\bbZ^d,$
\begin{equation}
\label{eq:proofP3GL1}
    \sup_{x\in{S_1}}\frac{\sup_{y\in{S_2}}g(x,y)}{\inf_{y\in{S_2}}g(x,y)}\leq c\left(\frac{RL+4L^{\xi_P}}{RL}\right)^{d-2}\leq cL^{(\xi_P-1)(d-2)},
\end{equation}
Moreover by (1.57) in \cite{MR2932978} we have
\begin{equation}
\label{eq:proofP3GL3}
    \sup_{x\in{S_1}}P_x(H_{S_2}<\infty)\leq |S_2|\sup_{x\in{S_1}}\sup_{y\in{S_2}}g(x,y)\leq \frac{cL^{\xi_P}}{(RL)^{d-2}}.
\end{equation}
Therefore, combining \eqref{eq:xiS1S2}, \eqref{eq:proofP3GL1} and \eqref{eq:proofP3GL3}, we obtain
\begin{equation*}
     R^{\chi_P} \, \Xi(S_1,S_2)\leq cL^{(\xi_P-1)(d-1)-(d-3)}R^{\chi_P-(d-2)}\leq cL^{-\xi_P/2}
\end{equation*}
if $\xi_P\leq 2- 3/(d-1/2)$. We thus obtain \eqref{eq:P3''replacebyindependentinc} by \eqref{eq:decoupineqGL} with the choice $\varepsilon= R^{-\chi_P}$. The inequality \eqref{eq:P3''replacebyindependentdec} can easily be deduced by symmetry, that is considering $f_1(-\phi)$ and $f_2(-\phi),$ which are increasing when $f_1$ and $f_2$ are decreasing, and using the fact that $-\phi$ has the same law as $\phi$ under $\mu^0.$
\end{proof}

We can now deduce results similar to Corollary \ref{app:1stcorGFF} for the Ginzburg-Landau $\nabla\phi$ interface model. More precisely, let us define 
\begin{equation*}
    h^+(\mu^{0})\ldef \inf\big\{h\in{\bbR}:\,\mu^0(Q_L\leftrightarrow Q_{2L}^c\text{ in }\{x\in{\bbZ^d}:\,\phi_x\geq h\})=0\big\}.
\end{equation*}
Using a union bound and a similar reasoning as in the proof of  \cite[Theorem~4.4]{2016arXiv161202385R}, that is using Proposition~\ref{prop:decoupGL}  instead of \cite[Theorem~2.1]{2016arXiv161202385R},  we know that $h^+(\mu^{0})<\infty.$ Let $f:\bbR\mapsto[0,\infty)$ be a decreasing function and $h_f=\inf\{t\in\bbR:\,f(t)=0\},$ with the convention $\inf\varnothing=+\infty.$ Proceeding similarly as in Corollary~\ref{app:1stcorGFF} we obtain from Theorem~ \ref{The:main} and Proposition~\ref{app:GLverifiesall} that if $d\geq4,$ $h_f>h_*$, and $(t_x^{\omega})_{x\in{\bbZ^d}}$ has the same law under $\bbP^0$ as $(f(\phi_x))_{x\in{\bbZ^d}}$ under $\mu^0$, then
\begin{equation*}
    \bbP^0\left(d^\omega(0,nx)\leq C_{\text{FPP}}n\right) \leq \exp\left(-cn\right).
\end{equation*}
In particular, the time constant is positive if it exists. However, it is not clear if the time constant is equal to $0$ if $h_f<h^+(\mu^0)$ similarly as in \eqref{app:mu_f>0}, since the sharpness result from \cite{DuGoRoSe} is only proved for the Gaussian free field, and not the general Ginzburg-Landau $\nabla\phi$ model.

\begin{remark}
\label{rk:P1P3GL}
Following \cite[ Section~4]{2016arXiv161202385R}, in particular see the proofs of Theorems~4.8 and 4.9 therein, one can prove that condition {\bf P3} from \cite{DRS14, Sa17} also holds for the ergodic measure $\mu^0$ from \eqref{eq:infinitevolumeGibbsmeasure} for all $d\geq3.$ In particular, all the results that can be deduced from  conditions {\bf P1}--{\bf P3} from \cite{DRS14, Sa17}, see for instance  \cite[Theorems~4.8 and 4.9]{2016arXiv161202385R}, also hold for $\mu^0$. Note that in \cite{2016arXiv161202385R} the shift-invariance and ergodicity of the $\phi$-Gibbs measure needed to be additionally assumed, see (4.23) therein.
\end{remark}

Finally, in view of Proposition~\ref{app:GLverifiesall} and Remarks~\ref{rk:P1P3GL} and \ref{rem:P3cond}-(i), one can apply Theorems~\ref{thm:hkeGauss} and \ref{thm:GKdecay_killing} on RCMs similar to Corollary~\ref{cor:optidecay}  in dimension $d\geq 4$. This could be generalized to other functionals of the field similarly as in Remark~\ref{rk:finalGFF}-(i). Another possible generalization is to consider the family of non-convex potentials  in \cite[Section~5]{2016arXiv161202385R}.

\subsection{Random interlacements}
\label{sec:inter}
In this subsection, we consider random interlacements on the graph $\bbZ^d,$ $d\geq 3$, with unit weights and zero killing measure, as introduced in \cite{MR2680403}. It can be seen as the natural limit of the random walk on the torus of size $N$, run up to time $uN^d$, and is closely related to the Gaussian free field, see \cite{MR2892408}. We denote by $\cI^u$ the interlacements set at level $u$ under the probability $\bbP^I$ characterized by the identity
\begin{equation*}
	\bbP^I(\cI^u\cap K=\varnothing)=\exp(-u \, \mathrm{cap}(K)) \quad \text{for all finite }K\subset\bbZ^d,
\end{equation*}
where $\mathrm{cap}(K)$ is the discrete capacity of $K.$ The set $\cI^u$ can be seen as the trace on $\bbZ^d$ of a Poisson point process of simple random walks on $\bbZ^d.$ We refer to the monograph \cite{MR3308116} for an introduction. We denote by $(L_{x,u})_{x\in{\bbZ^d}}$ the associated field of occupation times, that is $L_{x,u}$ is the total time spent in $x$ by the random walks in the point process, when the walks wait an exponentially distributed time with parameter one before jumping at each step. We refer to equation (1.8) in \cite{MR2892408} for a precise definition. We are now going to use the soft local times technique from \cite{MR3420516} to prove that condition {\bf P3} is satisfied.

\begin{prop}
\label{prop:condforRIverified}
Assume that $(\omega_x)_{x\in{\bbZ^d}}$ has the same law under $\bbP^u$ as $(L_{x,u})_{x\in{\bbZ^d}}$ under $\bbP^I$  for all $u>0,$ and that $\omega_e=0$ for all $e\in{E_d}.$ Then for all bounded open intervals $I\subset(0,\infty),$ $\xi_P>1$ and $\chi_P\in{(0,\frac{d-2}2)},$ $\eps_P=1$ and $a_P=d-2,$ there exist constants $C_P,R_P,L_P<\infty$ such that $(\bbP^u)_{u\in{I}}$ satisfies conditions ${\bf P1}$-${\bf P3}$. 
\end{prop}
\begin{proof}
Condition {\bf P1} follows from (2.7) in \cite{MR2680403}, and condition {\bf P2} clearly holds by definition. Moreover, one can easily adapt the proof of  \cite[Proposition~5.3]{MR3420516}, combined with equations (6.10) and (6.18) therein, to obtain that for all $u>0$, $\eps\in{(0,1)}$ and any sets $A_1,$ $A_2$ with diameter at most $r$ and at distance $s=d(A_1,A_2)$, there is a coupling $\hat{\bbP}^I$ under which 
\begin{equation}
\label{eq:decouplingRI}
\hat{\bbP}^I\left(L_{x,u(1-\eps)}^{(i)}\leq L_{x,u}\leq L_{x,u(1+\eps)}^{(i)},\, x\in{A_i},\, i=1,2 \right)\geq 1-c \,(r+s)^d \,\exp(-cu\eps^2s^{d-2}),
\end{equation}
where $(L_{x,u}^{(1)})_{x\in{\bbZ^d},u>0}$ and $(L_{x,u}^{(2)})_{x\in{\bbZ^d},u>0}$ are independent copies of $(L_{x,u})_{x\in{\bbZ^d},u>0}.$ Let us fix some bounded open interval $I\subset(0,\infty)$, $\xi_P>1,$ $\chi_P\in{(0,\frac{d-2}2)},$ $R\geq 1,$ $u,\hat{u}\in{I}$ with $u\geq\hat{u}+R^{-\chi_P},$ and $x_1,x_2\in{\bbZ^d}$ such that $d\big(Q(x_1,L^{\xi_P}),Q(x_2,L^{\xi_P})\big)\geq RL$. Setting 
\begin{equation*}
B\ldef \left\{L_{x,\hat{u}}^{(i)}\leq L_{x,u}\text{ and }L_{x,\hat{u}}\leq L_{x,u}^{(i)}\text{ for all }x\in{Q(x_i,L^{\xi_P})}\text{ and }i\in{\{1,2\}}\right\}
\end{equation*}
one can easily check that \eqref{eq:boundonprobaB^c} with $f_P(L)=L^{d-2}$ follows directly from \eqref{eq:decouplingRI} for all $L\geq L_P,$ for some large enough constants $L_P,C_P.$ Since $(\omega_x^{(1)},\omega_x^{(2)})_{x\in{\bbZ^d}}$ has the same law under $\hat{\bbP}^{{u}}$ as $\big(L_{x,u}^{(1)},L_{x,u}^{(2)}\big)_{x\in{\bbZ^d}}$ under $\bbP^I,$ one can then easily prove that \eqref{eq:P3'replacebyindependentinc} and \eqref{eq:P3'replacebyindependentdec} also hold since $u\mapsto L_{x,u}$ is increasing. 
\end{proof}

Proposition~\ref{prop:condforRIverified} can be used to show that the conditions of Theorem~\ref{The:main} are satisfied, and so the time constant is positive, when $(t_x^{\omega})_{x\in{\bbZ^d}}$ has the same law under $\bbP^u$ as $(\indicator_{\{x\in{\cI^u}\}})_{x\in{\bbZ^d}}$  for all $u>u_{**}$, where $u_{**}$ denotes the first level above which \eqref{intro:BLnotconnectedtoB2L} hold. In other words, $u_{**}$ is the minimal level $u$ for which the probability that $Q_L$ is connected to $Q_{2L}^c$ in $\cV^u$ converges to $0$ as $L\rightarrow\infty,$ where $\cV^u:=(\cI^u)^c$ is the vacant set of interlacements. Similarly as in the proof of Corollary~\ref{app:1stcorGFF} above, the FKG inequality (see e.g.\  \cite[Theorem 3.1]{MR2525105}) and the uniqueness of the infinite component of $\cV^u$ (see \cite{MR2498684}), imply by Proposition~\ref{Prop:condformu=0} that the time constant is equal to $0$ for all $u<u_*$, where $u_*$ is the critical parameter associated with the percolation of $\cV^u$. It is known that $u_*\leq u_{**}<\infty$, which essentially follows from \cite[Section~3]{MR2680403} (see \cite[Lemma~1.4]{MR2561432} for details), and that $u_*>0$ (see \cite{MR2512613}). The validity of $u_*=u_{**}$ is however still an open question, and we thus cannot reach a similar result as \eqref{app:mu>0} for the GFF, for which an analogue equality has been proved in \cite{DuGoRoSe}.
%

\subsection{Upper ratio weak mixing}
\label{sec:ratioweakmixing}
Another example is any model satisfying some weak mixing property, even without the need for an additional sprinkling parameter. We say that a probability $\bbP$ on $(\Omega,\cF)$ satisfies the upper ratio weak mixing property if there exist $\delta_P,\xi_P>1,$ $C_P>0$ and $L_P<\infty$ such that for all $L\geq L_P,$ $x_1,x_2\in{\bbZ^d}$ with $d\big(\bar{Q}(x_1,L^{\xi_P}),\bar{Q}(x_2,L^{\xi_P})\big)\geq L,$ and events $A_i$  supported on $\prod_{e\in{\bar{Q}(x_i,L^{\xi_P}})}\Omega_e$, $i\in \{1,2\}$,  we have
\begin{equation}
\label{eq:ratioweakmixing}
    \bbP\big(A_1\cap A_2\big)\leq\bbP(A_1) \,\bbP(A_2) \, \Big(1+\exp\big(-C_Pf_P(L)\big)\Big),
\end{equation}
with $f_P:[0,\infty)\rightarrow[0,\infty)$ satisfying $f_P(L)\geq \log(L)^{\delta_P}.$  
When $f_P(L)=L,$ a similar condition has been studied in \cite{MR1626951}, with an additional lower bound on $\bbP(A_1\cap A_2)/(\bbP(A_1)\bbP(A_2)),$ and shown to be equivalent to a weaker condition called weak mixing property under some additional hypotheses. In particular, the  upper ratio weak mixing property holds for the two-dimensional Potts model below the critical temperature, the two-dimensional Ising model with an external field, the Ising model below the critical temperature in all dimensions $d\geq2$, see \cite[Section~3]{MR1626951}, or the massive Gaussian free field. When $f_P$ is polynomial, \eqref{eq:ratioweakmixing} essentially corresponds to the exponential quasi independence assumption from \cite[2.4.6]{DeGa}, and thus discrete versions of the models considered in \cite[Section~3]{DeGa} also satisfy \eqref{eq:ratioweakmixing}. We now show that upper ratio weak mixing property implies the condition {\bf P3''} from Remark~\ref{rk:weakerP3} without sprinkling.

\begin{prop}
\label{prop:upperratioweakmixing}
Assume that $\bbP$ satisfies the upper ratio weak mixing property, and that $\bbP^u=\bbP$ for all $u\in{\bbR}$. Then $(\bbP^u)_{u\in{\bbR}}$ satisfies conditions {\bf P2} and {\bf P3''}. 
\end{prop}
\begin{proof}
Condition {\bf P2}  clearly holds by definition. Fix $R_P=\chi_P=1,$ $\xi_P,\eps_P,$ $a_P,$ $C_P,$ $L_P,$ $f_P,$ $L\geq L_P$ and $x_1,x_2$ as in the upper ratio weak mixing property. Let $\hat{\bbP}=\bbP\otimes\bbP,$ and write $\mathcal{E}$ the set of events $A \subset\prod_{e\in{{\bar{Q}(x_1,L^{\xi_P})}}}\Omega_{e}\times \prod_{e\in{\bar{Q}(x_2,L^{\xi_P})}}\Omega_{e}$ such that
\begin{equation*}
\frac{\bbP\big(\big((\omega_e)_{e\in{\bar{Q}(x_1,L^{\xi_P})}},(\omega_e)_{e\in{\bar{Q}(x_2,L^{\xi_P})}}\big)\in{A}\big)}{\hat{\bbP}\big(\big((\omega_e^{(1)})_{e\in{\bar{Q}(x_1,L^{\xi_P})}},(\omega_e^{(2)})_{e\in{\bar{Q}(x_2,L^{\xi_P})}}\big)\in{A}\big)}
    \leq1+\exp\big(-C_Pf_P(L)\big).
\end{equation*}
In view of \eqref{eq:ratioweakmixing}, $\cE$ contains all product sets, and it is easy to see that $\cE$ is stable by countable disjoint unions and decreasing intersections. Therefore, for all measurable functions $f_i:\Omega\rightarrow[0,\infty]$ supported on $\bar{Q}(x_i,L^{\xi_P}),$ $i\in{\{1,2\}}$, and all $s>0$, we have that $\{f_1+f_2\leq s\}\in{\mathcal{E}}$ since
\begin{equation}
\label{eq:f_1+f_2inE}
    \big\{f_1+f_2\leq s \big\}=\bigcap_{n\in\bbN}\bigcup_{k\in{\bbN_0}}\big\{f_1\in{[k2^{-n},(k+1)2^{-n})}\big\}\cap \big\{f_2\leq s-k2^{-n}\big\}.
\end{equation}
Now fix some event $B$ independent of $\omega$ under $\bbP$ (or on some extended probability space), with $\bbP(B)=1/(1+\exp(-C_Pf_P(L)))$. Then \eqref{eq:boundonprobaB^c} holds, upon changing $C_P$. Moreover, \eqref{eq:P3'replacebyindependentinc} holds for all $A\in{\mathcal{E}}$ and thus for all events of the form \eqref{eq:f_1+f_2inE}.
\end{proof}

Thus, in view of Remark~\ref{rk:weakerP3}, Theorem~\ref{The:main}  holds for ergodic models satisfying the upper ratio weak mixing property. If additionally the function $f_P$ from the upper ratio weak mixing property satisfies $f_P(L)\geq \exp(\log(L)^{\eps_P})$ for some $\eps_P>0$, then condition {\bf P3} from \cite{DRS14, Sa17} also holds, and the conclusions of Theorems~\ref{thm:hkeGauss} and \ref{thm:GKdecay_killing} are valid, see Remark~\ref{rem:P3cond}-(i). Note that the upper ratio weak mixing property does not depend on the choice of the partial order on $\Omega_E$ and $\Omega_V.$ In particular, condition {\bf P3''} still holds without the monotonicity assumption on $f_1,f_2$. Arguing similarly as in Remark~\ref{rk:fppvertices}-(ii), Theorem~\ref{The:main} thus still holds without the monotonicity assumption  \eqref{eq:temonotone} on $\mathbf{t}$. Moreover, by Remark~\ref{rem:P3cond}-(iii), Theorems~\ref{thm:hkeGauss} and  \ref{thm:GKdecay_killing} still hold without condition \eqref{eq:amonotone} and without Assumptions \ref{ass:moment}-(i) and \ref{ass:moment_killing}-(i), respectively.

\subsubsection*{Acknowledgment}
We are very grateful to Roland Bauerschmidt for a number of valuable discussions, especially in relation to the supersymmetric spin models mentioned in Example~\ref{ex:supersymmetric}, which initiated this work. We also thank Stephen Muirhead, Pierre-Fran\c{c}ois Rodriguez, Art\"{e}m Sapozhnikov (who communicated to us the content of Remark~\ref{rk:fppvertices}-(iii)) and Martin Slowik for helpful discussions, and we thank the referees for the careful reading and the constructive feedback. A.P.\ was partially supported by the Isaac Newton Trust grant G101121  ``Interplay of random media and statistical mechanics'' and the Engineering and Physical Sciences Research Council  grant EP/R022615/1 ``Random walks on dynamic graphs''.

\bibliographystyle{abbrv}
\bibliography{literature}

\end{document}